\documentclass{article}
\usepackage{amssymb}
\usepackage{latexsym}
\usepackage{amsmath}
\usepackage{mathrsfs}
\usepackage{color}
\usepackage{CJK}

\newtheorem{theorem}{Theorem}[section]
\newtheorem{corollary}{Corollary}[section]
\newtheorem{lemma}{Lemma}[section]
\newtheorem{proposition}{Proposition}[section]
\newtheorem{remark}{Remark}[section]

\numberwithin{equation}{section}
\newenvironment{proof}{\medskip\par\noindent{\bf Proof.}\ }{\qquad
\raisebox{-0.5mm}{\rule{1.5mm}{4mm}}\vspace{6pt}}

\newcommand{\bbr}{\mathbb{R}}
\newcommand{\h}{H^1(\bbr^3)}

\newcommand{\bbn}{\mathbb{N}}
\newcommand{\ve}{\varepsilon}

\begin{document}


\title
{\Large\bf On a two-component Bose-Einstein condensate with steep potential wells}

\author{Yuanze Wu$^{a},$\thanks{Corresponding author.  E-mail address: wuyz850306@cumt.edu.cn (Yuanze Wu)}\quad
Tsung-fang Wu$^b,$\thanks{E-mail address:  tfwu@nuk.edu.tw (Tsung-fang Wu)}
\quad
Wenming Zou$^c$\thanks{E-mail address:wzou@math.tsinghua.edu.cn (Wenming Zou)}\\
\footnotesize$^a${\em  College of Sciences, China
University of Mining and Technology, Xuzhou 221116, P.R. China }\\
\footnotesize$^b${\em  Department of Applied Mathematics, National University of Kaohsiung, Kaohsiung 811, Taiwan }\\
\footnotesize$^c${\em  Department of Mathematical Sciences, Tsinghua University, Beijing 100084, P.R. China }}
\date{}
\maketitle

\date{}
\maketitle

\noindent{\bf Abstract:} In this paper, we study the following two-component systems of nonlinear Schr\"odinger equations
\begin{equation*}
\left\{\aligned&\Delta u-(\lambda a(x)+a_0(x))u+\mu_1u^3+\beta v^2u=0\quad&\text{in }\bbr^3,\\
&\Delta v-(\lambda b(x)+b_0(x))v+\mu_2v^3+\beta u^2v=0\quad&\text{in }\bbr^3,\\
&u,v\in\h,\quad u,v>0\quad\text{in }\bbr^3,\endaligned\right.
\end{equation*}
where $\lambda,\mu_1,\mu_2>0$ and $\beta<0$ are parameters; $a(x), b(x)\geq0$ are steep potentials and $a_0(x),b_0(x)$ are sign-changing weight functions; $a(x)$, $b(x)$, $a_0(x)$ and $b_0(x)$ are not necessarily to be radial symmetric.  By the variational method, we obtain a ground state solution and multi-bump solutions for such systems with $\lambda$ sufficiently large.  The concentration behaviors of solutions as both $\lambda\to+\infty$ and $\beta\to-\infty$ are also considered. In particular, the phenomenon of phase separations is observed in the whole space $\bbr^3$.  In the Hartree-Fock theory, this provides a theoretical enlightenment of
phase separation in $\bbr^3$ for the 2-mixtures of Bose-Einstein condensates.

\vspace{6mm} \noindent{\bf Keywords:} Bose-Einstein condensate;   Steep potential well; Ground state solution; Multi-bump solution.

\vspace{6mm}\noindent {\bf AMS} Subject Classification 2010: 35B38; 35B40; 35J10; 35J20.

\section{Introduction}
In this paper, we consider the following two-component systems of nonlinear Schr\"odinger equations
\begin{equation*}
\left\{\aligned&\Delta u-(\lambda a(x)+a_0(x))u+\mu_1u^3+\beta v^2u=0\quad&\text{in }\bbr^3,\\
&\Delta v-(\lambda b(x)+b_0(x))v+\mu_2v^3+\beta u^2v=0\quad&\text{in }\bbr^3,\\
&u,v\in\h,\quad u,v>0\quad\text{in }\bbr^3,\endaligned\right.\eqno{(\mathcal{P}_{\lambda,\beta})}
\end{equation*}
where $\lambda,\mu_1,\mu_2>0$ and $\beta<0$ are parameters.  The potentials $a(x), b(x), a_0(x)$ and $b_0(x)$ satisfy the following conditions:
\begin{enumerate}
\item[$(D_1)$] $a(x), b(x)\in C(\bbr^3)$ and $a(x), b(x)\geq0$ on $\bbr^3$.
\item[$(D_2)$] There exist $a_\infty>0$ and $b_\infty>0$ such that $\mathcal{D}_a:=\{x\in\bbr^3\mid a(x)<a_\infty\}$ and $\mathcal{D}_b:=\{x\in\bbr^3\mid b(x)<b_\infty\}$ are nonempty and have finite measures.
\item[$(D_3)$] $\Omega_a=\text{int} a^{-1}(0)$ and $\Omega_b=\text{int} b^{-1}(0)$ are nonempty bounded sets and have smooth boundaries.  Moreover, $\overline{\Omega}_a=a^{-1}(0)$, $\overline{\Omega}_b=b^{-1}(0)$ and $\overline{\Omega}_a\cap\overline{\Omega}_b=\emptyset$.
 \item[$(D_4)$]$a_0(x),b_0(x)\in C(\bbr^3)$ and there exist $R,d_a,d_b>0$ such that
\begin{equation*}
a_0^-(x)\leq d_a(1+a(x))\quad\text{and}\quad b_0^-(x)\leq d_b(1+b(x))\quad\text{for }|x|\geq R,
\end{equation*}
where $a_0^-(x)=\max\{-a_0(x),0\}$ and $b_0^-(x)=\max\{-b_0(x),0\}$.
\item[$(D_5)$] $\inf\sigma_{a}(-\Delta+a_0(x))>0$ and $\inf\sigma_{b}(-\Delta+b_0(x))>0$, where $\sigma_{a}(-\Delta+a_0(x))$ is the spectrum of $-\Delta+a_0(x)$ on $H_0^1(\Omega_{a})$ and $\sigma_{b}(-\Delta+b_0(x))$ is the spectrum of $-\Delta+b_0(x)$ on $H_0^1(\Omega_{b})$.
\end{enumerate}

\begin{remark}
If $a_0(x),b_0(x)\in C(\bbr^3)$ are bounded, then the condition $(D_4)$ is trivial.  However,  under the assumptions of $(D_4)$-$(D_5)$, $a_0(x)$ and $b_0(x)$ may be sign-changing and unbounded.
\end{remark}

\vskip0.2in

 Two-component systems of nonlinear Schr\"odinger equations like $(\mathcal{P}_{\lambda,\beta})$ appear in the Hartree-Fock theory for a double condensate, that is, a binary mixture of Bose-Einstein condensates in two different hyperfine states $|1\rangle$ and $|2\rangle$ (cf. \cite{EGBB97}), where  the solutions $u$ and $v$ are the corresponding condensate amplitudes, $\mu_j$  are the intraspecies and interspecies scattering lengths.  The interaction is attractive if $\beta>0$ and repulsive if $\beta<0$.  When the interaction is repulsive, it is expected that the phenomenon of phase separations will happen, that is, the two components of the system  tend to separate in different regions as the interaction tends to infinity.    This kind of systems also arises in nonlinear optics (cf. \cite{AA99}).   Due to the important application in physics,  the following system

\begin{equation}\label{zou+01}
\left\{\aligned&\Delta u- \lambda_1 u+\mu_1u^3+\beta v^2u=0\quad&\text{in }  \Omega,\\
&\Delta v-\lambda_2v+\mu_2v^3+\beta u^2v=0\quad&\text{in }\Omega,\\
&u,v=0 \quad\text{on }\partial \Omega,\endaligned\right.
\end{equation}
where $\Omega \subset  \bbr^2  $  or    $\bbr^3$, has attracted many attentions of mathematicians in the past decade.  We refer the readers to \cite{BWW07,BDW10,B13,CZ13,CZ132,CLZ13,CLZ14,H09,I09,IK11,LW05,LW06,LW13,MS07,R14,S07}.  In these literatures, various existence theories of the solutions were established   for   the Bose-Einstein condensates in $\bbr^2$ and $\bbr^3$.    Recently, some mathematicians devoted their interest to the two coupled Schr\"odinger equations with critical Sobolev exponent in the high dimensions, and a number of the existence results of the solutions for such systems were also established.  See for example \cite{CZ121,CZ12,CZ13,CZ131,CZ14}.
\vskip0.12in

On the other hand, if the parameter $\lambda$ is sufficiently large, then $\lambda a(x)$ and $\lambda b(x)$ are called the steep potential wells under the conditions $(D_1)$-$(D_3)$.  The depth of the wells is controlled by the parameter $\lambda$.  Such potentials were first introduced by Bartsch and Wang in \cite{BW95} for the scalar Schr\"odinger equations.  An interesting phenomenon for this kind of Schr\"odinger equations is that, one can expect to find the solutions which are concentrated at the bottom of the wells as the depth goes to infinity.  Due to this interesting property, such topic for the scalar Schr\"odinger equations was studied extensively  in the past decade.  We refer the readers to \cite{BW00,BPW01,BT13,DT03,LHL11,SZ05,SZ06,ST09,WZ09} and the references therein.  In particular, in \cite{DT03}, by assuming that the bottom of the steep potential wells consists of finitely  many disjoint bounded domains, the authors obtained multi-bump solutions for  scalar  Schr\"odinger equations with steep potential wells,  which are concentrated at any given disjoint bounded domains of the bottom as the depth goes to infinity.

\vskip0.12in

We wonder what happens to the two-component Bose-Einstein condensate $(\mathcal{P}_{\lambda,\beta})$ with steep potential wells?   In the current paper, we shall explore this problem to find whether the solutions of such systems are concentrated at the bottom of the wells as $\lambda\to+\infty$ and when the phenomenon of phase separations of such systems can be observed in  the whole space  $\bbr^3$.

\vskip0.12in

We remark that the phenomenon of phase separations  for (\ref{zou+01}) was observed in \cite{CZ121,CZ131,CLZ14,CLZcommpde,NTTV10,WW08,WW081} for the ground state solution when $\Omega$ is a bounded domain.  In particular, this
phenomenon was also observed on the whole spaces $\bbr^2$ and $\bbr^3$ by \cite{TV09}, where the system is radial
symmetric!  However, when the system is not necessarily radial
symmetric, the phenomenon of phase seperations for Bose-Einstein condensates on the  whole space $\bbr^3$,  has not been obtained yet.  For other kinds of elliptic systems with strong competition, the phenomenon of phase separations has also been well studied;  we refer the readers  to  \cite{caf-1, caf-2, con-1}  and references therein.

\vskip0.12in

We recall some definitions in order to state  the main results in the current paper.  We say  that $(u_0, v_0)\in\h\times\h$ is a non-trivial solution of $(\mathcal{P}_{\lambda,\beta})$ if $(u_0, v_0)$ is a solution of $(\mathcal{P}_{\lambda,\beta})$ with $u_0\not=0$ and $v_0\not=0$.  We say $(u_0,v_0) \in\h\times\h$ is a ground state solution of $(\mathcal{P}_{\lambda,\beta})$ if $(u_0,v_0)$ is a non-trivial solution of $(\mathcal{P}_{\lambda,\beta})$ and
\begin{equation*}
J_{\lambda,\beta}(u_0,v_0)=\inf\{J_{\lambda,\beta}(u,v)\mid (u,v)\text{ is a non-trivial solution of }(\mathcal{P}_{\lambda,\beta})\},
\end{equation*}
where $J_{\lambda,\beta}(u,v)$ is the corresponding functional of $(\mathcal{P}_{\lambda,\beta})$ and given by
\begin{eqnarray}
J_{\lambda,\beta}(u,v)&=&\frac12\int_{\bbr^3}|\nabla u|^2+(\lambda a(x)+a_0(x))u^2dx+\frac12\int_{\bbr^3}|\nabla v|^2+(\lambda b(x)+b_0(x))v^2dx\notag\\
&&-\frac{\mu_1}{4}\int_{\bbr^3}u^4dx-\frac{\mu_2}{4}\int_{\bbr^3}v^4dx-\frac\beta2\int_{\bbr^3}u^2v^2dx.  \label{eq0081}
\end{eqnarray}
\begin{remark}
In section~2, we will give a variational setting of $(\mathcal{P}_{\lambda,\beta})$ and show that the solutions of $(\mathcal{P}_{\lambda,\beta})$ are  equivalent to the positive critical points of $J_{\lambda,\beta}(u,v)$ in a suitable Hilbert space $E$.
\end{remark}

\vskip0.12in

Let $I_{\Omega_a}(u)$ and $I_{\Omega_b}(v)$ be two functionals respectively defined on $H_0^1(\Omega_a)$ and $H_0^1(\Omega_b)$, which are given by
\begin{equation*}
I_{\Omega_a}(u)=\frac12\int_{\Omega_a}|\nabla u|^2+a_0(x)u^2dx-\frac{\mu_1}{4}\int_{\Omega_a}u^4dx
\end{equation*}
and by
\begin{equation*}
I_{\Omega_b}(v)=\frac12\int_{\Omega_b}|\nabla v|^2+b_0(x)v^2dx-\frac{\mu_2}{4}\int_{\Omega_b}v^4dx.
\end{equation*}
Then by the condition $(D_5)$, it is well-known that $I_{\Omega_a}(u)$ and $I_{\Omega_b}(v)$ have least energy nonzero critical points.  We denote the least energy of nonzero critical points for $I_{\Omega_a}(u)$ and $I_{\Omega_b}(v)$ by $m_a$ and $m_b$, respectively.  Now, our first result can be stated as the following.

\vskip0.21in

\begin{theorem}
\label{thm0002} Assume $(D_1)$-$(D_5)$.  Then there exists $\Lambda_{*}>0$ such that $(\mathcal{P}_{\lambda,\beta})$ has a ground state solution $(u_{\lambda,\beta}, v_{\lambda,\beta})$ for all $\lambda\geq\Lambda_{*}$ and $\beta<0$, which has  the following properties:
\begin{enumerate}
\item[$(1)$] $\int_{\bbr^3\backslash\Omega_{a}}|\nabla u_{\lambda,\beta}|^2+u_{\lambda,\beta}^2dx\to0$ and $\int_{\bbr^3\backslash\Omega_{b}}|\nabla v_{\lambda,\beta}|^2+v_{\lambda,\beta}^2dx\to0$ as $\lambda\to+\infty$.
\item[$(2)$] $\int_{\Omega_{a}}|\nabla u_{\lambda,\beta}|^2+a_0(x)u_{\lambda,\beta}^2dx\to 4m_{a}$ and $\int_{\Omega_{b}}|\nabla v_{\lambda,\beta}|^2+b_0(x)v_{\lambda,\beta}^2dx\to 4m_{b}$ as $\lambda\to+\infty$.
\end{enumerate}
Furthermore, for each $\{\lambda_n\}\subset[\Lambda_*, +\infty)$ satisfying $\lambda_n\to+\infty$ as $n\to\infty$ and $\beta<0$, there exists $(u_{0,\beta},v_{0,\beta})\in(\h\backslash\{0\})\times(\h\backslash\{0\})$ such that
\begin{enumerate}
\item[$(3)$] $(u_{0,\beta},v_{0,\beta})\in H_0^1(\Omega_{a})\times H_0^1(\Omega_{b})$ with $u_{0,\beta}\equiv0$ outside $\Omega_{a}$ and $v_{0,\beta}\equiv0$ outside $\Omega_{b}$.
\item[$(4)$] $(u_{\lambda_n,\beta},v_{\lambda_n,\beta})\to(u_{0,\beta},v_{0,\beta})$ strongly in $\h\times\h$ as $n\to\infty$ up to a subsequence.
\item[$(5)$] $u_{0,\beta}$ is a least energy nonzero critical point of $I_{\Omega_{a}}(u)$ and $v_{0,\beta}$ is a least energy nonzero critical point of $I_{\Omega_{b}}(v)$.
\end{enumerate}
\end{theorem}

\vskip0.3in

Next, we assume that the bottom of the steep potential wells consists of finitely many disjoint bounded domains. It is natural to ask whether the two-component Bose-Einstein condensate $(\mathcal{P}_{\lambda,\beta})$ with such steep potential wells has multi-bump solutions which are concentrated at any given disjoint bounded domains of the bottom as the depth goes to infinity.  Our second result is devoted to this study.  Similar to \cite{DT03}, we need the following conditions on the potentials $a(x)$, $b(x)$, $a_0(x)$ and $b_0(x)$.
\begin{enumerate}
\item[$(D_3')$] $\Omega_a=\text{int} a^{-1}(0)$ and $\Omega_b=\text{int} b^{-1}(0)$ satisfy $\Omega_a=\underset{i_a=1}{\overset{n_a}{\cup }}\Omega_{a,i_a}$ and $\Omega_b=\underset{j_b=1}{\overset{n_b}{\cup }}\Omega_{b,j_b}$, where $\{\Omega_{a,i_a}\}$ and $\{\Omega_{b,j_b}\}$ are all nonempty bounded domains with smooth boundaries, and $\overline{\Omega}_{a,i_a}\cap\overline{\Omega}_{a,j_a}=\emptyset$ for $i_a\not=j_a$ and $\overline{\Omega}_{b,i_b}\cap\overline{\Omega}_{b,j_b}=\emptyset$ for $i_b\not=j_b$.  Moreover, $\overline{\Omega}_a=a^{-1}(0)$ and $\overline{\Omega}_b=b^{-1}(0)$ with $\overline{\Omega}_a\cap\overline{\Omega}_b=\emptyset$.
\item[$(D_5')$] $\inf\sigma_{a,i_a}(-\Delta+a_0(x))>0$ for all $i_a=1,\cdots,n_a$ and $\inf\sigma_{b,j_b}(-\Delta+b_0(x))>0$ for all $j_b=1,\cdots,n_b$, where $\sigma_{a,i_a}(-\Delta+a_0(x))$ is the spectrum of $-\Delta+a_0(x)$ on $H_0^1(\Omega_{a,i_a})$ and $\sigma_{b,j_b}(-\Delta+b_0(x))$ is the spectrum of $-\Delta+b_0(x)$ on $H_0^1(\Omega_{b,j_b})$.
\end{enumerate}

\vskip0.12in

\begin{remark}
Under the conditions $(D_3')$ and $(D_4)$, it is easy to see that the condition $(D_5')$ is equivalent to the condition $(D_5)$.  For the sake of clarity, we use the condition $(D_5')$ in the study of multi-bump solutions.
\end{remark}

\vskip0.12in

We define $I_{\Omega_{a,i_a}}(u)$ on $H_0^1(\Omega_{a,i_a})$ for each $i_a=1,\cdots,n_a$ by
\begin{eqnarray*}
I_{\Omega_{a,i_a}}(u)=\frac12\int_{\Omega_{a,i_a}}|\nabla u|^2+a_0(x)u^2dx-\frac{\mu_1}{4}\int_{\Omega_{a,i_a}}u^4dx
\end{eqnarray*}
and $I_{\Omega_{b,j_b}}(v)$ on $H_0^1(\Omega_{b,j_b})$ for each $j_b=1,\cdots,n_b$ by
\begin{eqnarray*}
I_{\Omega_{b,j_b}}(v)=\frac12\int_{\Omega_{b,j_b}}|\nabla v|^2+b_0(x)v^2dx-\frac{\mu_2}{4}\int_{\Omega_{b,j_b}}v^4dx.
\end{eqnarray*}
Then by the conditions $(D_3')$ and $(D_5')$, it is well-known that $I_{\Omega_{a,i_a}}(u)$ and $I_{\Omega_{b,j_b}}(v)$ have least energy nonzero critical points for every $i_a=1,\cdots,n_a$ and every $j_b=1,\cdots,n_b$, respectively.  We   denote the least energy of nonzero critical points for $I_{\Omega_{a,i_a}}(u)$ and $I_{\Omega_{b,j_b}}(v)$ by $m_{a,i_a}$ and $m_{b,j_b}$, respectively.  Now, our second result can be stated as the following.

\vskip0.21in

\begin{theorem}
\label{thm0001}Assume $\beta<0$ and the conditions $(D_1)$-$(D_2)$, $(D_3')$, $(D_4)$ and $(D_5')$ hold.  If  the set $J_a\times J_b\subset\{1,\cdots,n_a\}\times\{1,\cdots,n_b\}$ satisfying $J_a\not=\emptyset$ and $J_b\not=\emptyset$, then there exists $\Lambda_{*}(\beta)>0$ such that $(\mathcal{P}_{\lambda,\beta})$ has a non-trivial solution $(u_{\lambda,\beta}^{J_a},v_{\lambda,\beta}^{J_b})$ for $\lambda\geq\Lambda_{*}(\beta)$ with the following properties:
\begin{enumerate}
\item[$(1)$] $\int_{\bbr^3\backslash\Omega_{a,0}^{J_a}}|\nabla u_{\lambda,\beta}^{J_a}|^2+(u_{\lambda,\beta}^{J_a})^2dx\to0$ and $\int_{\bbr^3\backslash\Omega_{b,0}^{J_b}}|\nabla v_{\lambda,\beta}^{J_b}|^2+(v_{\lambda,\beta}^{J_b})^2dx\to0$ as $\lambda\to+\infty$, where $\Omega_{a,0}^{J_a}=\underset{i_a\in J_a}{{\cup }}\Omega_{a,i_a}$ and $\Omega_{b,0}^{J_b}=\underset{j_b\in J_b}{{\cup }}\Omega_{b,j_b}$.
\item[$(2)$] $\int_{\Omega_{a,i_a}}|\nabla u_{\lambda,\beta}^{J_a}|^2+a_0(x)(u_{\lambda,\beta}^{J_a})^2dx\to 4m_{a,i_a}$ and $\int_{\Omega_{b,j_b}}|\nabla v_{\lambda,\beta}^{J_b}|^2+b_0(x)(v_{\lambda,\beta}^{J_b})^2dx\to 4m_{b,j_b}$ as $\lambda\to+\infty$ for all $i_a\in J_a$ and $j_b\in J_b$.
\end{enumerate}
Furthermore, for each $\beta<0$ and $\{\lambda_n\}\subset[\Lambda_*(\beta), +\infty)$ satisfying $\lambda_n\to+\infty$ as $n\to\infty$, there exists $(u_{0,\beta}^{J_a},v_{0,\beta}^{J_b})\in(\h\backslash\{0\})\times(\h\backslash\{0\})$ such that
\begin{enumerate}
\item[$(3)$] $(u_{0,\beta}^{J_a},v_{0,\beta}^{J_b})\in H_0^1(\Omega_{a,0}^{J_a})\times H_0^1(\Omega_{b,0}^{J_b})$ with $u_{0,\beta}^{J_a}\equiv0$ outside $\Omega_{a,0}^{J_a}$ and $v_{0,\beta}^{J_b}\equiv0$ outside $\Omega_{b,0}^{J_b}$.
\item[$(4)$] $(u_{\lambda_n,\beta}^{J_a},v_{\lambda_n,\beta}^{J_b})\to(u_{0,\beta}^{J_a},v_{0,\beta}^{J_b})$ strongly in $\h\times\h$ as $n\to\infty$ up to a subsequence.
\item[$(5)$] the restriction of $u_{0,\beta}^{J_a}$ on $\Omega_{a,i_a}$ lies in $H_0^1(\Omega_{a,i_a})$ and is a least energy nonzero critical point of $I_{\Omega_{a,i_a}}(u)$ for all $i_a\in J_a$, while the restriction of $v_{0,\beta}^{J_b}$ on $\Omega_{b,j_b}$ lies in $H_0^1(\Omega_{b,j_b})$ and is a least energy nonzero critical point of $I_{\Omega_{b,j_b}}(v)$ for all $j_b\in J_b$.
\end{enumerate}
\end{theorem}

\vskip0.21in

\begin{corollary}
Suppose $\beta<0$ and the conditions $(D_1)$-$(D_2)$, $(D_3')$, $(D_4)$ and $(D_5')$ hold.  Then $(\mathcal{P}_{\lambda,\beta})$ has at least $(2^{n_a}-1)(2^{n_b}-1)$ non-trivial solutions for $\lambda\geq\Lambda_{*}(\beta)$.
\end{corollary}

\begin{remark}
\begin{enumerate}
\item[$(i)$] To  the   best of our knowledge,  it seems that Theorem~\ref{thm0001} is the first result for
the existence of multi-bump solutions  to system~$(\mathcal{P}_{\lambda,\beta})$.

\item[$(ii)$] Under the condition $(D_3')$, we can see that
\begin{equation*}
I_{\Omega_a}(u)=\sum_{i_a=1}^{n_a}I_{\Omega_{a,i_a}}(u)\quad\text{and}\quad
I_{\Omega_b}(v)=\sum_{j_b=1}^{n_b}I_{\Omega_{b,j_b}}(v).
\end{equation*}
Let $m_{a,0}=\min\{m_{a,1},\cdots,m_{a,n_a}\}$ and $m_{b,0}=\min\{m_{b,1},\cdots,m_{b,n_b}\}$.  Then we must have $m_{a,0}=m_a$ and $m_{b,0}=m_b$.  Without loss of generality, we assume $m_{a,0}=m_{a,1}$ and $m_{b,0}=m_{b,1}$.  Now, by Theorem~\ref{thm0001}, we can find a solution of $(\mathcal{P}_{\lambda,\beta})$ with the same concentration behavior as  the ground state solution obtained in Theorem~\ref{thm0002} as $\lambda\to+\infty$.  However, we do not know these two solutions are  the same or not.
\end{enumerate}
\end{remark}

Next we consider the phenomenon of phase separations for System~$(\mathcal{P}_{\lambda,\beta})$, i.e., the concentration behavior of the solutions as $\beta\to-\infty$.   In the following theorem, we may observe such a phenomenon   on the whole space $\bbr^3$.

\vskip0.21in

\begin{theorem}
\label{thm0003}Assume $(D_1)$-$(D_5)$.  Then there exists $\Lambda_{**}\geq\Lambda_*$ such that  $\beta^2\int_{\bbr^3}u_{\lambda,\beta}^2v_{\lambda,\beta}^2\to0$ as $\beta\to-\infty$ for $\lambda\geq\Lambda_{**}$, where $(u_{\lambda,\beta},v_{\lambda,\beta})$ is the ground state solution of $(\mathcal{P}_{\lambda,\beta})$ obtained by Theorem~\ref{thm0002}.  Furthermore, for every $\{\beta_n\}\subset(-\infty, 0)$ with $\beta_n\to-\infty$ and $\lambda\geq\Lambda_{**}$, there exists $(u_{\lambda,0}, v_{\lambda,0})\in(\h\backslash\{0\})\times(\h\backslash\{0\})$ satisfying the following properties:
\begin{enumerate}
\item[$(1)$] $(u_{\lambda,\beta_n},v_{\lambda,\beta_n})\to(u_{\lambda,0}, v_{\lambda,0})$ strongly in $\h\times\h$ as $n\to\infty$ up to a subsequence.
\item[$(2)$] $u_{\lambda,0}\in C(\bbr^3)$ and $v_{\lambda,0}\in C(\bbr^3)$.
\item[$(3)$] $u_{\lambda,0}\geq0$ and $v_{\lambda,0}\geq0$ in $\bbr^3$ with $\{x\in\bbr^3\mid u_{\lambda,0}(x)>0\}=\bbr^3\backslash\overline{\{x\in\bbr^3\mid v_{\lambda,0}(x)>0\}}$.  Furthermore, $\{x\in\bbr^3\mid u_{\lambda,0}(x)>0\}$ and $\{x\in\bbr^3\mid v_{\lambda,0}(x)>0\}$ are connected domains.
\item[$(4)$] $u_{\lambda,0}\in H_0^1(\{u_{\lambda,0}>0\})$ and is a least energy solution of
  \begin{equation}     \label{eq0401}
  -\Delta u+(\lambda a(x)+a_0(x))u=\mu_1 u^3,\quad u\in H_0^1(\{u_{\lambda,0}>0\}),
  \end{equation}
  while $v_{\lambda,0}\in H_0^1(\{v_{\lambda,0}>0\})$ and is a least energy solution of
  \begin{equation}     \label{eq0402}
  -\Delta v+(\lambda a(x)+a_0(x))v=\mu_1 v^3,\quad v\in H_0^1(\{v_{\lambda,0}>0\}).
  \end{equation}
\end{enumerate}
\end{theorem}

\begin{remark}
In Theorem~\ref{thm0001}, the multi-bump solutions have been found for $\lambda\geq\Lambda_*(\beta)$.  By checking the proof of Theorem~\ref{thm0001}, we can see that $\Lambda_*(\beta)\to+\infty$ as $\beta\to-\infty$.  Due to this fact, the multi-bump solutions obtained in Theorem~\ref{thm0001} can not have the same phenomenon of phase separations as the ground state solution described in Theorem~\ref{thm0003}.
\end{remark}


\vskip0.12in

Before closing this section, we would like to cite  other references studying the equations with steep potential wells.  For example, in \cite{SW14}, the Kirchhoff type elliptic equation with a steep potential well was studied.  The Schr\"odinger-Poisson systems with a steep potential well were considered in \cite{JZ11,ZLZ13}.  Non-trivial solutions were obtained in \cite{FSX10,GT12,GT121} for quasilinear Schr\"odinger equations  with steep potential wells, while the multi-bump solutions were also obtained in \cite{GT121} for such equations.

\vskip0.12in


In this paper, we will always denote the usual norms in $\h$ and $L^p(\bbr^3)$ ($p\geq1$) by $\|\cdot\|$ and $\|\cdot\|_p$, respectively;   $C$ and $C'$ will be indiscriminately used to denote various positive constants;  $o_n(1)$ will always denote  the quantities  tending  towards zero as $n\to\infty$.

\section{The variational setting}
In this section, we mainly give a variational setting for $(\mathcal{P_{\lambda,\beta}})$.  Simultaneously, an important estimate is also established in this section, which is used frequently in this paper.

\vskip0.1in

Let
\begin{equation*}
E_{a}=\{u\in D^{1,2}(\bbr^3)\mid\int_{\bbr^3}(a(x)+a_0^+(x))u^2dx<+\infty\}
\end{equation*}
and
\begin{equation*}
E_{b}=\{u\in D^{1,2}(\bbr^3)\mid\int_{\bbr^3}(b(x)+b_0^+(x))u^2dx<+\infty\},
\end{equation*}
where $a_0^+(x)=\max\{a_0(x), 0\}$ and $b_0^+(x)=\max\{b_0(x), 0\}$.
Then by the conditions $(D_1)$ and $(D_4)$, $E_a$ and $E_b$ are Hilbert spaces equipped with the inner products
\begin{equation*}
\langle u,v\rangle_{a}=\int_{\bbr^3}\nabla u\nabla v+(a(x)+a_0^+(x))uvdx\quad\text{and}\quad
\langle u,v\rangle_{b}=\int_{\bbr^3}\nabla u\nabla v+(b(x)+b_0^+(x))uvdx,
\end{equation*}
respectively.  The corresponding norms of $E_a$ and $E_b$ are respectively given by
\begin{equation*}
\|u\|_{a}=\bigg(\int_{\bbr^3}|\nabla u|^2+(a(x)+a_0^+(x))u^2dx\bigg)^{\frac12}
\end{equation*}
and by
\begin{equation*}
\|v\|_{b}=\bigg(\int_{\bbr^3}|\nabla v|^2+(b(x)+b_0^+(x))v^2dx\bigg)^{\frac12}.
\end{equation*}
Since the conditions $(D_1)$-$(D_2)$ hold, by a similar argument as that in \cite{SW14}, we can see that
\begin{equation}
\|u\|\leq\bigg(\max\{1+|\mathcal{D}_a|^{\frac{2}{3}}S^{-1},\frac{1}{a_\infty}\}\bigg)^{\frac12}\|u\|_a\quad\text{for all }u\in E_{a}\label{eq0004}
\end{equation}
and
\begin{equation}
\|v\|\leq\bigg(\max\{1+|\mathcal{D}_b|^{\frac{2}{3}}S^{-1},\frac{1}{b_\infty}\}\bigg)^{\frac12}\|v\|_b\quad\text{for all }v\in E_{b},\label{eq0005}
\end{equation}
where $S$ is the best Sobolev embedding constant from $D^{1,2}(\bbr^3)$ to $L^6(\bbr^3)$ and given by
\begin{equation*}
S=\inf\{\|\nabla u\|_2^2 \mid u\in D^{1,2}(\bbr^3), \|u\|_6^2=1\}.
\end{equation*}
It follows that both $E_a$ and $E_b$  are embedded continuously into $\h$.
Moreover, by applying the H\"older and Sobolev inequalities, we also have
\begin{equation}\|u\|_4\leq\bigg(\max\{1+|\mathcal{D}_a|^{\frac{2}{3}}S^{-1},\frac{1}{a_\infty}\}\bigg)^{\frac12}S^{-\frac{3}{8}}\|u\|_a\quad\text{for all }u\in E_{a}\label{eq0006}
\end{equation}
and
\begin{equation}
\|v\|_4\leq\bigg(\max\{1+|\mathcal{D}_b|^{\frac{2}{3}}S^{-1},\frac{1}{b_\infty}\}\bigg)^{\frac12}S^{-\frac{3}{8}}\|v\|_b\quad\text{for all }v\in E_{b}.\label{eq0007}
\end{equation}
On the other hand, by the conditions $(D_2)$ and $(D_3)$, there exist two bounded open sets $\Omega_a'$ and $\Omega_b'$ with smooth boundaries such that $\Omega_a\subset\Omega_a'\subset \mathcal{D}_a$, $\Omega_b\subset\Omega_b'\subset\mathcal{D}_b$, $\overline{\Omega_a'}\cap\overline{\Omega_b'}=\emptyset$, dist$(\Omega_a, \bbr^3\backslash\Omega_a')>0$ and  that dist$(\Omega_b, \bbr^3\backslash\Omega_b')>0$.  Furthermore, by the condition $(D_4)$, the H\"older and the Sobolev inequalities, there exists $\Lambda_0>2\max\{1,d_a+\frac{d_a+C_{a,0}}{a_\infty}, d_b+\frac{d_b+C_{b,0}}{b_\infty}\}$
such that
\begin{eqnarray}
\int_{\bbr^3}a_0^-(x)u^2dx\leq\int_{B_R(0)}C_{a,0}u^2dx+\int_{\bbr^3\backslash B_R(0)}d_a(1+a(x))u^2dx
\leq\frac{\lambda}{2}\|u\|_a^2\label{eq0002}
\end{eqnarray}
and
\begin{equation}
\int_{\bbr^3}b_0^-(x)v^2dx\leq\int_{B_R(0)}C_{b,0}v^2dx+\int_{\bbr^3\backslash B_R(0)}d_b(1+b(x))v^2dx
\leq\frac\lambda2\|v\|_b^2\label{eq0003}
\end{equation}
for $\lambda\geq\Lambda_0$, where $B_R(0)=\{x\in\bbr^3\mid |x|< R\}$, $C_{a,0}=\sup_{B_R(0)}a_0^-(x)$ and $C_{b,0}=\sup_{B_R(0)}b_0^-(x)$.
Combining \eqref{eq0006}-\eqref{eq0003} and the H\"older inequality, we can see that $(\mathcal{P}_{\lambda,\beta})$ has a variational structure in the Hilbert space $E=E_{a}\times E_{b}$ for $\lambda\geq\Lambda_0$,   where $E$ is endowed with  the  norm  $\|(u,v)\|=\|u\|_{a}+\|v\|_{b}$.  The corresponding functional of $(\mathcal{P}_{\lambda,\beta})$ is given by \eqref{eq0081}.  Furthermore, by applying \eqref{eq0006}-\eqref{eq0003} in a standard way, we can also see that $J_{\lambda,\beta}(u,v)$ is $C^2$ in $E$ and the solution of $(\mathcal{P}_{\lambda,\beta})$ is equivalent to the positive critical point of $J_{\lambda,\beta}(u,v)$ in $E$ for $\lambda\geq\Lambda_0$. In the case of $(D_3')$, we can choose $\Omega_{a}'$ and $\Omega_{b}'$ as follows:


\begin{enumerate}
\item[$(I)$] $\Omega_a'=\underset{i_a=1}{\overset{n_a}{\cup }}\Omega'_{a,i_a}\subset\mathcal{D}_a$, where $\Omega_{a,i_a}\subset \Omega'_{a,i_a}$ and dist$(\Omega_{a,i_a}, \bbr^3\backslash\Omega_{a,i_a}')>0$ for all $i_a=1,\cdots, n_a$ and $\overline{\Omega'_{a,i_a}}\cap\overline{\Omega'_{a,j_a}}=\emptyset$ for $i_a\not=j_a$.
\item[$(II)$] $\Omega_b'=\underset{i_b=1}{\overset{n_b}{\cup }}\Omega'_{b,i_b}\subset\mathcal{D}_b$, where $\Omega_{b,i_b}\subset \Omega'_{b,i_b}$ and dist$(\Omega_{b,i_b}, \bbr^3\backslash\Omega_{b,i_b}')>0$ for all $j_b=1,\cdots, n_b$ and $\overline{\Omega'_{b,i_b}}\cap\overline{\Omega'_{b,j_b}}=\emptyset$ for $i_b\not=j_b$.
\item[$(III)$] $\overline{\Omega_a'}\cap\overline{\Omega_b'}=\emptyset$.
\end{enumerate}
Thus, \eqref{eq0002}-\eqref{eq0003} still hold for such $\Omega_{a}'$ and $\Omega_{b}'$ with $\lambda$ sufficiently large.  Without loss of generality, we may assume that \eqref{eq0002}-\eqref{eq0003} still hold for such $\Omega_{a}'$ and $\Omega_{b}'$ with $\lambda\geq\Lambda_0$.  It follows that the solution of $(\mathcal{P}_{\lambda,\beta})$ is also equivalent to the positive critical point of the $C^2$ functional $J_{\lambda,\beta}(u,v)$ in $E$ for $\lambda\geq\Lambda_0$ under the conditions $(D_1)$-$(D_2)$, $(D_3')$ and $(D_4)$.

\vskip0.1in

The remaining of this section will be devoted to an important estimate, which is used frequently in this paper and essentially due to Ding and Tanaka \cite{DT03}.

\vskip0.1in

\begin{lemma}
\label{lem0001}  Assume $(D_1)$-$(D_5)$.  Then there exist $\Lambda_1\geq\Lambda_0$ and $C_{a,b}>0$ such that
\begin{equation*}
\inf_{u\in E_{a}\backslash\{0\}}\frac{\int_{\bbr^3}|\nabla u|^2+(\lambda a(x)+a_0(x))u^2dx}{\int_{\bbr^3}u^2dx}\geq C_{a,b}
\end{equation*}
and
\begin{equation*}
\inf_{v\in E_{b}\backslash\{0\}}\frac{\int_{\bbr^3}|\nabla v|^2+(\lambda b(x)+b_0(x))v^2dx}{\int_{\bbr^3}v^2dx}\geq C_{a,b}
\end{equation*}
for all $\lambda\geq\Lambda_1$.
\end{lemma}
\begin{proof}
Since the conditions $(D_1)$-$(D_4)$ hold, by a similar argument as \cite[Lemma~2.1]{DT03}, we have
\begin{equation*}
\lim_{\lambda\to+\infty}\inf\sigma_{a,*}(-\Delta+\lambda a(x)+a_0(x))=\inf\sigma_{a}(-\Delta+a_0(x)),
\end{equation*}
where $\sigma_{a,*}(-\Delta+\lambda a(x)+a_0(x))$ is the spectrum of $-\Delta+\lambda a(x)+a_0(x)$ on $H^1(\Omega_{a}')$.  Denote $\inf\sigma_{a}(-\Delta+a_0(x))$ by $\nu_{a}$.  Then by the condition $(D_5)$, there exists $\Lambda_1'\geq\Lambda_0$ such that
\begin{equation}\label{eq0017}
\sigma_{a,*}(-\Delta+\lambda a(x)+a_0(x))\geq\frac{\nu_a}{2}\quad\text{for}\quad \lambda\geq\Lambda_1'.
\end{equation}
On the other hand, by the conditions $(D_2)$ and $(D_4)$, we have $a_0^-(x)\leq C_{a,0}+d_a+d_aa_\infty$ for $x\in\mathcal{D}_a$.  Let $\mathcal{D}_{a,\overline{R}}=\mathcal{D}_a\cap B^c_{\overline{R}}$, where $B^c_{\overline{R}}=\{x\in\bbr\mid |x|\geq\overline{R}\}$.  Then by the condition $(D_2)$ once more, $|\mathcal{D}_{a,\overline{R}}|\to0$ as $\overline{R}\to+\infty$, which then implies that there exists $\overline{R}_0>0$ such that $|\mathcal{D}_{a,\overline{R}_0}|S^{-1}(C_{a,0}+d_a+d_aa_\infty+1)\leq\frac12$.  Thanks to the conditions $(D_1)$-$(D_4)$, there exists $\Lambda_1=\Lambda_1(\overline{R}_0)\geq\Lambda_1'$ such that
\begin{equation*}
\lambda a(x)+a_0(x)+(C_{a,0}+d_a+d_aa_\infty+1)\chi_{\mathcal{D}_{a,\overline{R}_0}}\geq1\quad\text{for all }x\in\bbr^3\backslash\Omega_a'\text{ and }\lambda\geq\Lambda_1,
\end{equation*}
where $\chi_{\mathcal{D}_{a,\overline{R}_0}}$ is the characteristic function of the set $\mathcal{D}_{a,\overline{R}_0}$.  It follows from the H\"older and the Sobolev inequalities that
\begin{equation}    \label{eq0022}
\int_{\bbr^3\backslash\Omega_a'}u^2dx\leq(1+2|\mathcal{D}_{a,\overline{R}_0}|S^{-1})\int_{\bbr^3\backslash\Omega_a'}|\nabla u|^2+(\lambda a(x)+a_0(x))u^2dx
\end{equation}
for all $u\in E_a\backslash\{0\}$ and $\lambda\geq\Lambda_1$.
Combining \eqref{eq0017}-\eqref{eq0022} and the choice of $\Omega_a'$, we have
\begin{equation*}
\int_{\bbr^3}u^2dx\leq C_a\int_{\bbr^3}|\nabla u|^2+(\lambda a(x)+a_0(x))u^2dx \quad\text{for all }u\in E_a\backslash\{0\}\text{ and }\lambda\geq\Lambda_1,
\end{equation*}
where $C_a=\max\{\frac{2}{\nu_a}, 1+2|\mathcal{D}_{a,\overline{R}_0}|S^{-1}\}$.  By similar arguments as \eqref{eq0017} and \eqref{eq0022}, we can also have
\begin{equation*}
\int_{\bbr^3}v^2dx\leq C_b\int_{\bbr^3}|\nabla v|^2+(\lambda b(x)+b_0(x))v^2dx
\end{equation*}
for all $v\in E_b\backslash\{0\}$ and $\lambda\geq\Lambda_1$, where $C_b=\max\{\frac{2}{\nu_b}, 1+2|\mathcal{D}_{b,\overline{R}_0}|S^{-1}\}$, $\nu_b=\inf\sigma_{b}(-\Delta+b_0(x))$ and $\mathcal{D}_{b,\overline{R}_0}=\mathcal{D}_b\cap B^c_{\overline{R}_0}$.  We completes the proof by taking $C_{a,b}=(\min\{C_a,C_b\})^{-1}$.
\end{proof}

\begin{remark}
Under the conditions $(D_1)$-$(D_2)$, $(D_3')$, $(D_4)$ and $(D_5')$, we can see that
\begin{equation*}
\nu_a=\min_{i_a=1,2,\cdots,n_a}\bigg\{\inf\sigma_{a,i_a}(-\Delta+a_0(x))\bigg\}\quad\text{and}\quad
\nu_b=\min_{j_b=1,2,\cdots,n_b}\bigg\{\inf\sigma_{b,j_b}(-\Delta+b_0(x))\bigg\}.
\end{equation*}
Now, by a similar argument as \eqref{eq0017}, we   get that
\begin{equation}
\int_{\Omega_{a,i_a}'}u^2dx\leq\frac{2}{\nu_a}\int_{\Omega_{a,i_a}'}|\nabla u|^2+(\lambda a(x)+a_0(x))u^2 dx\label{eq0087}
\end{equation}
and
\begin{equation}
\int_{\Omega_{b,j_b}'}v^2dx\leq\frac{2}{\nu_b}\int_{\Omega_{b,j_b}'}|\nabla v|^2+(\lambda b(x)+b_0(x))v^2 dx\label{eq0210}
\end{equation}
for all $i_a=1,\cdots,n_a$ and $j_b=1,\cdots,n_b$ if $\lambda$ sufficiently large.  Without loss of generality, we may assume \eqref{eq0087} and \eqref{eq0210} hold for $\lambda\geq\Lambda_1$.  It follows that Lemma~\ref{lem0001} still holds under the conditions $(D_1)$-$(D_2)$, $(D_3')$, $(D_4)$ and $(D_5')$.
\end{remark}

\vskip0.1in

By Lemma~\ref{lem0001}, we  observe  that $\int_{\bbr^3}|\nabla u|^2+(\lambda a(x)+a_0(x))u^2dx$ and $\int_{\bbr^3}|\nabla v|^2+(\lambda b(x)+b_0(x))v^2dx$ are   norms of  $E_a$ and $E_b$ for $\lambda\geq\Lambda_1$, respectively.  Therefore, we set
\begin{equation*}
\|u\|_{a,\lambda}^2=\int_{\bbr^3}|\nabla u|^2+(\lambda a(x)+a_0(x))u^2dx;\quad \|v\|_{b,\lambda}^2=\int_{\bbr^3}|\nabla v|^2+(\lambda b(x)+b_0(x))v^2dx.
\end{equation*}

\section{A ground state solution}
Our interest in this section is to find a ground state solution to  $(\mathcal{P}_{\lambda,\beta})$ under the conditions $(D_1)$-$(D_5)$.  For the sake of convenience, we always assume the conditions $(D_1)$-$(D_5)$ hold in this section.  Since $J_{\lambda,\beta}(u,v)$, the corresponding energy functional of $(\mathcal{P}_{\lambda,\beta})$, is $C^2$ in $E$, it is well-known that all non-trivial solutions of $(\mathcal{P}_{\lambda,\beta})$ lie in the Nehari manifold of $J_{\lambda,\beta}(u,v)$, which is given by
\begin{equation*}
\mathcal{N}_{\lambda,\beta}=\{(u,v)\in E\mid u\not=0,v\not=0, \langle D[J_{\lambda,\beta}(u,v)],(u,0)\rangle_{E^*,E}=\langle D[J_{\lambda,\beta}(u,v)],(0,v)\rangle_{E^*,E}=0\},
\end{equation*}
where $D[J_{\lambda,\beta}(u,v)]$ is the Frech\'et derivative of the functional $J_{\lambda,\beta}$ in $E$ at $(u,v)$ and $E^*$ is the dual space of $E$.  If we can find $(u_{\lambda,\beta}, v_{\lambda,\beta})\in E$ such that $J_{\lambda,\beta}(u_{\lambda,\beta}, v_{\lambda,\beta})=m_{\lambda,\beta}$ and $D[J_{\lambda,\beta}(u_{\lambda,\beta}, v_{\lambda,\beta})]=0$ in $E^*$, then $(u_{\lambda,\beta}, v_{\lambda,\beta})$ must be a ground state solution of $(\mathcal{P}_{\lambda,\beta})$, where $m_{\lambda,\beta}=\inf_{\mathcal{N}_{\lambda,\beta}}J_{\lambda,\beta}(u,v)$.   In what follows, we drive some properties of  $\mathcal{N}_{\lambda,\beta}$.

\vskip0.1in


Let $(u,v)\in (E_a\backslash\{0\})\times(E_b\backslash\{0\})$ and
define $T_{\lambda,\beta,u,v}:\bbr^+\times\bbr^+\to\bbr$ by $T_{\lambda,\beta,u,v}(t,s)=J_{\lambda,\beta}(tu,sv)$.  These functions are called  the fibering maps of $J_{\lambda,\beta}(u,v)$, which are closely linked to $\mathcal{N}_{\lambda,\beta}$.  Clearly, $\frac{\partial T_{\lambda,\beta,u,v}}{\partial t}(t,s)=\frac{\partial T_{\lambda,\beta,u,v}}{\partial s}(t,s)=0$ is equivalent to $(tu,sv)\in\mathcal{N}_{\lambda,\beta}$.  In particular, $\frac{\partial T_{\lambda,\beta,u,v}}{\partial t}(1,1)=\frac{\partial T_{\lambda,\beta,u,v}}{\partial s}(1,1)=0$ if and only if $(u,v)\in\mathcal{N}_{\lambda,\beta}$.  Let
\begin{equation}
\mathcal{A}_{\beta}=\{(u,v)\in E\mid \mu_1\mu_2\|u\|_4^4\|v\|_4^4-\beta^2\|u^2v^2\|^2_1>0\}.\label{eq0069}
\end{equation}
Then $\mathcal{A}_{\beta}\not=\emptyset$ for every $\beta<0$.  Now, our first observation on $\mathcal{N}_{\lambda,\beta}$ can be stated as follows.

\vskip0.1in

\begin{lemma}
\label{lem0007}Assume $\lambda\geq\Lambda_1$ and $\beta<0$.  Then we have the following.
\begin{enumerate}
\item[$(1)$] If $(u,v)\in\mathcal{A}_{\beta}$, then there exists a unique $(t_{\lambda,\beta}(u,v),s_{\lambda,\beta}(u,v))\in\bbr^+\times\bbr^+$ such that
    \begin{equation*}
    (t_{\lambda,\beta}(u,v)u,s_{\lambda,\beta}(u,v)v)\in\mathcal{N}_{\lambda,\beta},
    \end{equation*}
    where $t_{\lambda,\beta}(u,v)$ and $s_{\lambda,\beta}(u,v)$ are given by
    \begin{equation}
    t_{\lambda,\beta}(u,v)=\bigg(\frac{\mu_2\|v\|_4^4\|u\|_{a,\lambda}^2-\beta\|u^2v^2\|_1\|v\|_{b,\lambda}^2}
    {\mu_1\mu_2\|u\|_4^4\|v\|_4^4-\beta^2\|u^2v^2\|^2_1}\bigg)^{\frac12}\label{eq0071}
    \end{equation}
    and by
    \begin{equation}
    s_{\lambda,\beta}(u,v)=\bigg(\frac{\mu_1\|u\|_4^4\|v\|_{b,\lambda}^2-\beta\|u^2v^2\|_1\|u\|_{a,\lambda}^2}
    {\mu_1\mu_2\|u\|_4^4\|v\|_4^4-\beta^2\|u^2v^2\|^2_1}\bigg)^{\frac12}.\label{eq0072}
    \end{equation}
    Moreover, $T_{\lambda,\beta,u,v}(t_{\lambda,\beta}(u,v),s_{\lambda,\beta}(u,v))=\max_{t\geq0,s\geq0}T_{\lambda,\beta,u,v}(t,s)$.
\item[$(2)$] If $(u,v)\in E\backslash\mathcal{A}_{\beta}$, then $\mathcal{B}_{u,v}\cap\mathcal{N}_{\lambda,\beta}=\emptyset$, where $\mathcal{B}_{u,v}=\{(tu,sv)\mid (t,s)\in\bbr^+\times\bbr^+\}$.
\end{enumerate}
\end{lemma}
\begin{proof}
$(1)$\quad  The proof is similar to \cite[Lemma~2.2]{CLZ14}, where $\mathcal{N}_{0,\beta}$ with $a_0(x)=a_0>0$ and $b_0(x)=b_0>0$ was studied.  For the convenience of readers, we give the details here.   Let $T^1_{\lambda,\beta,u,v}$ and $T^2_{\lambda,\beta,u,v}$ be two functions on $\bbr^+\times\bbr^+$  defined by
\begin{equation*}
T^1_{\lambda,\beta,u,v}(t,s)=\|u\|_{a,\lambda}^2-\mu_1\|u\|_4^4t^2-\beta\|u^2v^2\|_1s^2
\end{equation*}
and by
\begin{equation*}
T^2_{\lambda,\beta,u,v}(t,s)=\|v\|_{b,\lambda}^2-\mu_2\|v\|_4^4s^2-\beta\|u^2v^2\|_1t^2.
\end{equation*}
Then it is easy to see that
\begin{equation}
\frac{\partial T_{\lambda,\beta,u,v}}{\partial t}(t,s)=tT^1_{\lambda,\beta,u,v}(t,s)\quad\text{and}\quad\frac{\partial T_{\lambda,\beta,u,v}}{\partial s}(t,s)=sT^2_{\lambda,\beta,u,v}(t,s).\label{eq0073}
\end{equation}
Suppose $(u,v)\in\mathcal{A}_{\beta}$, $\lambda\geq\Lambda_1$ and $\beta<0$.  Then by Lemma~\ref{lem0001}, the two-component systems of algebraic equations, given by
\begin{equation}
\left\{\aligned T^1_{\lambda,\beta,u,v}(t,s)&=0,\\
T^2_{\lambda,\beta,u,v}(t,s)&=0,\endaligned\right.\label{eq0070}
\end{equation}
has a unique nonzero solution $(t_{\lambda,\beta}(u,v),s_{\lambda,\beta}(u,v))$ in $\bbr^+\times\bbr^+$, where $(t_{\lambda,\beta}(u,v),s_{\lambda,\beta}(u,v))$ is characterized as \eqref{eq0071} and \eqref{eq0072}.  Hence, by \eqref{eq0073}, $(t_{\lambda,\beta}(u,v),s_{\lambda,\beta}(u,v))$ is the unique one in $\bbr^+\times\bbr^+$ such that $\frac{\partial T_{\lambda,\beta,u,v}}{\partial t}(t,s)=\frac{\partial T_{\lambda,\beta,u,v}}{\partial s}(t,s)=0$, that is, $(t_{\lambda,\beta}(u,v),s_{\lambda,\beta}(u,v))$ is the unique one in $\bbr^+\times\bbr^+$ such that $(t_{\lambda,\beta}(u,v)u,s_{\lambda,\beta}(u,v)v)\in\mathcal{N}_{\lambda,\beta}$.  It remains to show that $T_{\lambda,\beta,u,v}(t_{\lambda,\beta}(u,v),s_{\lambda,\beta}(u,v))=\max_{t\geq0,s\geq0}T_{\lambda,\beta,u,v}(t,s)$.
Indeed, by a direct calculation, we have
\begin{equation*}
\frac{\partial^2T_{\lambda,\beta,u,v}}{\partial t^2}(t_{\lambda,\beta}(u,v),s_{\lambda,\beta}(u,v))=-2\mu_1\|u\|_4^4[t(u,v)]^2<0
\end{equation*}
and
\begin{eqnarray*}
&&\begin{vmatrix}
\frac{\partial^2T_{\lambda,\beta,u,v}}{\partial t^2}(t_{\lambda,\beta}(u,v),s_{\lambda,\beta}(u,v))&\frac{\partial^2T_{\lambda,\beta,u,v}}{\partial t\partial s}(t_{\lambda,\beta}(u,v),s_{\lambda,\beta}(u,v))\\
\frac{\partial^2T_{\lambda,\beta,u,v}}{\partial s\partial t}(t_{\lambda,\beta}(u,v),s_{\lambda,\beta}(u,v))&\frac{\partial^2T_{\lambda,\beta,u,v}}{\partial s^2}(t_{\lambda,\beta}(u,v),s_{\lambda,\beta}(u,v))
\end{vmatrix}\\
&=&
\begin{vmatrix}
-2[t_{\lambda,\beta}(u,v)]^2\mu_1\|u\|_4^4&-2t_{\lambda,\beta}(u,v)s_{\lambda,\beta}(u,v)\beta\|u^2v^2\|_1\\
-2t_{\lambda,\beta}(u,v)s_{\lambda,\beta}(u,v)\beta\|u^2v^2\|_1&-2[s_{\lambda,\beta}(u,v)]^2\mu_2\|v\|_4^4
\end{vmatrix}\\
&=&4[t_{\lambda,\beta}(u,v)]^2[s_{\lambda,\beta}(u,v)]^2(\mu_1\mu_2\|u\|_4^4\|v\|_4^4-\beta^2\|u^2v^2\|^2_1)>0,
\end{eqnarray*}
which   implies that $(t_{\lambda,\beta}(u,v),s_{\lambda,\beta}(u,v))$ is a local maximum of $T_{\lambda,\beta,u,v}(t,s)$ in $\bbr^+\times\bbr^+$.  It follows from the uniqueness of $(t_{\lambda,\beta}(u,v),s_{\lambda,\beta}(u,v))$, $T_{\lambda,\beta,u,v}(t,s)>0$ for $|(t,s)|$ sufficiently small and $T_{\lambda,\beta,u,v}(t,s)\to-\infty$ as $|(t,s)|\to+\infty$ that $(t_{\lambda,\beta}(u,v),s_{\lambda,\beta}(u,v))$ must be the global maximum of $T_{\lambda,\beta,u,v}(t,s)$ in $\bbr^+\times\bbr^+$.  Thus, we have
\begin{equation*}
T_{\lambda,\beta,u,v}(t_{\lambda,\beta}(u,v),s_{\lambda,\beta}(u,v))=\max_{t\geq0,s\geq0}T_{\lambda,\beta,u,v}(t,s).
\end{equation*}

\vskip0.1in

$(2)$\quad
Suppose $(u,v)\not\in\mathcal{A}_{\beta}$, $\lambda\geq\Lambda_1$ and $\beta<0$.  If $\mathcal{B}_{u,v}\cap\mathcal{N}_{\lambda,\beta}\not=\emptyset$, then there exists $(t,s)\in\bbr^+\times\bbr^+$ such that $\frac{\partial T_{\lambda,\beta,u,v}}{\partial t}(t,s)=\frac{\partial T_{\lambda,\beta,u,v}}{\partial s}(t,s)=0$.  It follows from \eqref{eq0073} that $(t,s)$ is a solution of \eqref{eq0070} in $\bbr^+\times\bbr^+$.  On the other hand, since $(u,v)\not\in\mathcal{A}_{\beta}$, $\lambda\geq\Lambda_1$ and $\beta<0$, by Lemma~\ref{lem0001}, \eqref{eq0070} has no solution in $\bbr^+\times\bbr^+$, which is a contradiction.  Hence, we must have $\mathcal{B}_{u,v}\cap\mathcal{N}_{\lambda,\beta}=\emptyset$ if $(u,v)\not\in\mathcal{A}_{\beta}$, $\lambda\geq\Lambda_1$ and $\beta<0$.
\end{proof}

\vskip0.1in

By Lemma~\ref{lem0007}, we know that $\mathcal{N}_{\lambda,\beta}\subset\mathcal{A}_{\beta}$ for $\lambda\geq\Lambda_1$ and $\beta<0$.  Moreover, $m_{\lambda,\beta}$ is well defined and nonnegative for $\lambda\geq\Lambda_1$ and $\beta<0$.  Let
\begin{eqnarray*}
I_{a,\lambda}(u)=\frac12\|u\|^2_{a,\lambda}-\frac{\mu_1}{4}\|u\|_4^4\quad\text{and}\quad
I_{b,\lambda}(v)=\frac12\|v\|^2_{b,\lambda}-\frac{\mu_2}{4}\|v\|_4^4.
\end{eqnarray*}
Then by \eqref{eq0006}-\eqref{eq0003}, $I_{a,\lambda}(u)$ is well defined on $E_{a}$ and $I_{b,\lambda}(v)$ is well defined on $E_{b}$.  Moreover, by a standard argument, we can see that $I_{a,\lambda}(u)$ and $I_{b,\lambda}(v)$ are of $C^2$ in $E_a$ and $E_b$, respectively.  Denote
\begin{eqnarray*}
\mathcal{N}_{a,\lambda}=\{u\in E_{a}\backslash\{0\}\mid I_{a,\lambda}'(u)u=0\}\quad\text{and}\quad
\mathcal{N}_{b,\lambda}=\{u\in E_{b}\backslash\{0\}\mid I_{b,\lambda}'(u)u=0\}.
\end{eqnarray*}
Clearly, $\mathcal{N}_{a,\lambda}$ and $\mathcal{N}_{b,\lambda}$ are nonempty,  which together with Lemma~\ref{lem0001}, implies  $m_{a,\lambda}=\inf_{\mathcal{N}_{a,\lambda}}I_{a,\lambda}(u)$ and $m_{b,\lambda}=\inf_{\mathcal{N}_{b,\lambda}}I_{b,\lambda}(v)$ are well defined and nonnegative.  Due to this fact, we have the following.

\vskip0.1in

\begin{lemma}
\label{lem0008}
Assume $\lambda\geq\Lambda_1$ and $\beta<0$.  Then  $m_{\lambda,\beta}\in[m_{a,\lambda}+m_{b,\lambda}, m_a+m_b]$,where $m_a$ and $m_b$ are the least energy of nonzero critical points for $I_{\Omega_a}(u)$ and $I_{\Omega_b}(v)$, respectively.
\end{lemma}
\begin{proof}
Suppose $(u,v)\in\mathcal{N}_{\lambda,\beta}$.  Then by Lemma~\ref{lem0001} and $(u,v)\in\mathcal{N}_{\lambda,\beta}\subset\mathcal{A}_{\beta}$, we can see that $\|u\|_{a,\lambda}^2>0$, $\|v\|_{b,\lambda}^2>0$, $\|u\|_4>0$ and $\|v\|_4>0$.  It follows that there exists a unique $(t^*(u), s^*(v))\in\bbr^+\times\bbr^+$ such that
$(t^*(u)u, s^*(v)v)\in\mathcal{N}_{a,\lambda}\times\mathcal{N}_{b,\lambda}$.
Note that $\beta<0$, so by Lemma~\ref{lem0007}, we have
\begin{equation*}
J_{\lambda,\beta}(u,v)\geq J_{\lambda,\beta}(t^*(u)u,s^*(v)v)\geq I_{a,\lambda}(t^*(u)u)+I_{b,\lambda}(s^*(v)v)\geq m_{a,\lambda}+m_{b,\lambda}.
\end{equation*}
Since $(u,v)\in\mathcal{N}_{\lambda,\beta}$ is arbitrary, we must have $m_{\lambda,\beta}\geq m_{a,\lambda}+m_{b,\lambda}$ for $\lambda\geq\Lambda_1$ and $\beta<0$.  It remains to show that $m_{\lambda,\beta}\leq m_a+m_b$ for $\lambda\geq\Lambda_1$ and $\beta<0$.  In fact, let $w_{a}\in H_0^1(\Omega_a)$ and $w_{b}\in H^1_0(\Omega_b)$   be the least energy nonzero critical points of $I_{\Omega_a}(u)$ and $I_{\Omega_b}(v)$, respectively.  Then by the conditions $(D_3)$ and $(D_5)$, it is well-known that
\begin{equation*}
I_{\Omega_a}(w_{a})=\max_{t\geq0}I_{\Omega_a}(tw_{a})\quad\text{and}\quad I_{\Omega_b}(w_{b})=\max_{s\geq0}I_{\Omega_b}(sw_{b}).
\end{equation*}
On the other hand, by the condition $(D_3)$, we can extend $w_{a}$ and $w_{b}$ to  $\bbr^3$ by  letting $w_a=0$ outside $\Omega_a$ and $w_b=0$ outside $\Omega_b$ such that $w_a,w_b\in\h$.  Thanks to the condition $(D_3)$ again, we can see that $(w_{a}, w_{b})\in\mathcal{A}_{\beta}$.  It follows from Lemma~\ref{lem0007} that there exists $(t_{\lambda,\beta}(w_{a}, w_{b}), s_{\lambda,\beta}(w_{a}, w_{b}))\in\bbr^+\times\bbr^+$ such that
\begin{equation*}
(t_{\lambda,\beta}(w_{a}, w_{b})w_a, s_{\lambda,\beta}(w_{a}, w_{b})w_b)\in\mathcal{N}_{\lambda,\beta}\quad\text{for }\lambda\geq\Lambda_1,
\end{equation*}
which together with the condition $(D_3)$ once more, implies
\begin{eqnarray*}
m_a+m_b&=&I_{\Omega_a}(w_{a})+I_{\Omega_b}(w_{b})\\
&\geq&I_{\Omega_a}(t_{\lambda,\beta}(w_{a}, w_{b})w_{a})+I_{\Omega_b}(s_{\lambda,\beta}(w_{a}, w_{b})w_{b})\\
&=&J_{\lambda,\beta}(t_{\lambda,\beta}(w_{a}, w_{b})w_{a},s_{\lambda,\beta}(w_{a}, w_{b})w_{b})\\
&\geq& m_{\lambda,\beta}
\end{eqnarray*}
for $\lambda\geq\Lambda_1$ and $\beta<0$.
\end{proof}

\vskip0.1in

Clearly, $m_{a,\lambda}$ and $m_{b,\lambda}$ are nondecreasing for $\lambda$.  On the other hand, since Lemma~\ref{lem0001} hold, by the conditions $(D_1)$-$(D_5)$,  it is easy to show that $m_{a,\lambda}$ and $m_{b,\lambda}$ are positive for $\lambda\geq\Lambda_1$ and $m_{a,\lambda}\to m_{a}$ and $m_{b, \lambda}\to m_{b}$ as $\lambda\to+\infty$.
This fact will help us to observe the following property of $\mathcal{N}_{\lambda,\beta}$, which is based on Lemma~\ref{lem0008}.

\vskip0.1in

\begin{lemma}
\label{lem0009}Assume $\lambda\geq\Lambda_1$ and $\beta<0$.  Then there exists $d_{\lambda,\beta}>0$ such that $\mathcal{N}_{\lambda,\beta}\subset\mathcal{A}_{\beta}^{d_{\lambda,\beta}}$, where $\mathcal{A}_{\beta}^{d_{\lambda,\beta}}=\{(u,v)\in E \mid u\not=0,v\not=0,\mu_1\mu_2\|u\|_4^4\|v\|_4^4-\beta^2\|u^2v^2\|^2_1>d_{\lambda,\beta}\}$.
\end{lemma}
\begin{proof}
A similar result was obtained in \cite{CZ131}.  But as we will see, some new ideas are needed due to the fact that $a_0(x)$ and $b_0(x)$ are sign-changing.  Suppose the contrary.  Since $\mathcal{N}_{\lambda,\beta}\subset\mathcal{A}_{\beta}$, there exists a sequence $\{(u_n,v_n)\}\subset\mathcal{N}_{\lambda,\beta}$ such that $\mu_1\mu_2\|u_n\|_4^4\|v_n\|_4^4=\beta^2\|u_n^2v_n^2\|^2_1+o_n(1)$, where $\mathcal{A}_{\beta}$ is given in \eqref{eq0069}.  Clearly, one of the following two cases must happen:
\begin{enumerate}
\item[$(a)$] $\|u_n\|_4^4\|v_n\|_4^4\geq C+o_n(1)$.
\item[$(b)$] $\|u_n\|_4^4\|v_n\|_4^4=o_n(1)$ up to a subsequence.
\end{enumerate}
Suppose case $(a)$ happen.  We claim that $\mu_1\|u_n\|_4^4+\beta\|u_n^2v_n^2\|_1=o_n(1)$ and $\mu_2\|v_n\|_4^4+\beta\|u_n^2v_n^2\|_1=o_n(1)$ up to a subsequence.  If not, then up to a subsequence, we have
\begin{equation*}
\mu_1\|u_n\|_4^4+\beta\|u_n^2v_n^2\|_1\geq C_1+o_n(1)\quad\text{and}\quad\mu_2\|v_n\|_4^4+\beta\|u_n^2v_n^2\|_1\geq C_2+o_n(1)
\end{equation*}
for $\lambda\geq\Lambda_1$ and $\beta<0$,
where $C_1, C_2$ are nonnegative constants with $C_1+C_2>0$.  It follows from $\beta<0$ that
\begin{eqnarray*}
\mu_1\mu_2\|u_n\|_4^4\|v_n\|_4^4&\geq&(C_1+o_n(1)+|\beta|\|u_n^2v_n^2\|_1)(C_2+o_n(1)+|\beta|\|u_n^2v_n^2\|_1)\\
&\geq&\beta^2\|u_n^2v_n^2\|^2_1+(C_1+C_2+o_n(1))|\beta|\|u_n^2v_n^2\|_1+C_1C_2+o_n(1)\\
&\geq&\beta^2\|u_n^2v_n^2\|^2_1+\frac12(C_1+C_2)\sqrt{C}+o_n(1)
\end{eqnarray*}
for $n$ large enough, which contradicts to $\mu_1\mu_2\|u_n\|_4^4\|v_n\|_4^4=\beta^2\|u_n^2v_n^2\|^2_1+o_n(1)$.  This together with $\{(u_n,v_n)\}\subset\mathcal{N}_{\lambda,\beta}$, implies  that  $\|u_n\|_{a,\lambda}=o_n(1)$ and $\|v_n\|_{b,\lambda}=o_n(1)$ up to a subsequence.  Note that $J_{\lambda,\beta}(u_n,v_n)=\frac14(\|u_n\|_{a,\lambda}^2+\|v_n\|_{b,\lambda}^2)$.  So $m_{\lambda,\beta}\leq0$ in case $(a)$, which is impossible for $\lambda\geq\Lambda_1$ and $\beta<0$ due to Lemma~\ref{lem0008}.
Now, we must have case $(b)$.  It follows that $\|u_n\|_4=o_n(1)$ or $\|v_n\|_4=o_n(1)$ up to a subsequence.  Without loss of generality, we assume $\|u_n\|_4=o_n(1)$.  Since $\mu_1\mu_2\|u_n\|_4^4\|v_n\|_4^4=\beta^2\|u_n^2v_n^2\|^2_1+o_n(1)$, we also have $\beta\|u_n^2v_n^2\|_1=o_n(1)$ in this case.  These together with $\{(u_n,v_n)\}\subset\mathcal{N}_{\lambda,\beta}$, imply
$\|u_n\|_{a,\lambda}=o_n(1)$.  Therefore, $J_{\lambda,\beta}(u_n,v_n)=I_{b,\lambda}(v_n)+o_n(1)$.  On the other hand, since $\lambda\geq\Lambda_1$ and $\{(u_n,v_n)\}\subset\mathcal{N}_{\lambda,\beta}$, by Lemma~\ref{lem0001} and $\mathcal{N}_{\lambda,\beta}\subset\mathcal{A}_{\beta}$, for all $n\in\bbn$, there exists a unique $t^*(u_n)>0$ such that $t^*(u_n)u_n\in\mathcal{N}_{a,\lambda}$.  It follows from Lemma~\ref{lem0007} and $\beta<0$ that
\begin{eqnarray*}
J_{\lambda,\beta}(u_n,v_n)&\geq& J_{\lambda,\beta}(t^*(u_n)u_n,v_n)\\
&\geq&I_{a,\lambda}(t^*(u_n)u_n)+I_{b,\lambda}(v_n)\\
&\geq&m_{a,\lambda}+I_{b,\lambda}(v_n)\\
&=&m_{a,\lambda}+J_{\lambda,\beta}(u_n,v_n)+o_n(1)
\end{eqnarray*}
for $\lambda\geq\Lambda_1$ and $\beta<0$,
which is also impossible for $n$ large enough.  Thus, there exists $d_{\lambda,\beta}>0$ such that $\mathcal{N}_{\lambda,\beta}\subset\mathcal{A}_{\beta}^{d_{\lambda,\beta}}$ for $\lambda\geq\Lambda_1$ and $\beta<0$.
\end{proof}

\vskip0.1in

We also have the following.

\vskip0.1in

\begin{lemma}
\label{lem0010}Suppose $\lambda\geq\Lambda_1$ and $\beta<0$.  Then $\mathcal{N}_{\lambda,\beta}$ is a natural constraint.
\end{lemma}
\begin{proof}
Let $\varphi_{\lambda,\beta}(u,v)=\langle D[J_{\lambda,\beta}(u,v)],(u,v)\rangle_{E^*,E}$.  Then by \eqref{eq0006}-\eqref{eq0003}, $\varphi_{\lambda,\beta}(u,v)$ is $C^2$ in $E$ for $\lambda\geq\Lambda_1$ and $\beta<0$.
Since $\lambda\geq\Lambda_1$ and $(u,v)\in\mathcal{N}_{\lambda,\beta}$, we have
\begin{eqnarray*}
\langle D[\varphi_{\lambda,\beta}(u,v)],(u,v)\rangle_{E^*,E}=-2(\mu_1\|u\|_4^4+\mu_2\|v\|_4^4+2\beta\|u^2v^2\|_1)
\leq-8m_{\lambda,\beta}.
\end{eqnarray*}
It follows from Lemma~\ref{lem0008} that $\mathcal{N}_{\lambda,\beta}$ is a natural constraint for $\lambda\geq\Lambda_1$ and $\beta<0$.
\end{proof}

\vskip0.1in

Now, we can obtain a ground state solution for $(\mathcal{P}_{\lambda,\beta})$.

\vskip0.1in

\begin{proposition}
\label{prop0002}There exists $\Lambda_2\geq\Lambda_1$ such that $(\mathcal{P}_{\lambda,\beta})$ has a ground state solution $(u_{\lambda,\beta},v_{\lambda,\beta})$ for $\lambda\geq\Lambda_2$ and $\beta<0$.  Furthermore, we have
\begin{equation}\label{eq0018}
\lim_{\lambda\to+\infty}\int_{\bbr^3\backslash\Omega_a'}|\nabla u_{\lambda,\beta}|^2+(\lambda a(x)+a_0(x))u_{\lambda,\beta}^2 dx=0
\end{equation}
and
\begin{equation}\label{eq0040}
\lim_{\lambda\to+\infty}\int_{\bbr^3\backslash\Omega_b'}|\nabla v_{\lambda,\beta}|^2+(\lambda b(x)+b_0(x))v_{\lambda,\beta}^2 dx=0.
\end{equation}
\end{proposition}
\begin{proof}
Let $\lambda\geq\Lambda_1$ and $\beta<0$.  Then for every $\{(u_n,v_n)\}\subset\mathcal{N}_{\lambda,\beta}$ satisfying $J_{\lambda,\beta}(u_n,v_n)=m_{\lambda,\beta}+o_n(1)$, we can see from Lemma~\ref{lem0001} that
\begin{eqnarray}
m_{\lambda,\beta}+o_n(1)&\geq&J_{\lambda,\beta}(u_n,v_n)-\frac14\langle D[J_{\lambda,\beta}(u_n,v_n)], (u_n,v_n)\rangle_{E^*,E}\notag\\
&=&\frac14\|u_n\|^2_{a,\lambda}+\frac14\|v_n\|^2_{b,\lambda}\notag\\
&\geq&\frac{1}{4C_{a,b}}(\|u_n\|_2^2+\|v_n\|_2^2),\label{eq0009}
\end{eqnarray}
which together with the condition $(D_4)$ and $\lambda\geq\Lambda_1$, implies
\begin{eqnarray}
m_{\lambda,\beta}+o_n(1)&\geq&\frac14\|u_n\|^2_{a,\lambda}+\frac14\|v_n\|^2_{b,\lambda}\notag\\
&\geq&\frac18\|(u_n,v_n)\|^2-C(m_{\lambda,\beta}+o_n(1)).\label{eq0332}
\end{eqnarray}
It follows that $\|(u_n,v_n)\|\leq 8(C+1)(m_{\lambda,\beta}+o_n(1))$.  Now, by Lemma~\ref{lem0009}, we can apply the implicit function theorem and the Ekeland variational principle in a standard way (cf. \cite{CZ121,LW13}) to show that there exists $\{(u_n, v_n)\}\subset\mathcal{N}_{\lambda,\beta}$ such that $D[J_{\lambda,\beta}(u_n,v_n)]=o_n(1)$ strongly in $E^*$ and $J_{\lambda,\beta}(u_n,v_n)=m_{\lambda,\beta}+o_n(1)$.  Since $m_{\lambda,\beta}\leq m_{a}+m_{b}$,
by similar arguments as  \eqref{eq0009} and \eqref{eq0332}, we have $\|(u_n,v_n)\|\leq8(C+1)(m_{a}+m_{b}+o_n(1))$
and $(u_n,v_n)\rightharpoonup(u_{\lambda,\beta},v_{\lambda,\beta})$ weakly in $E$ as $n\to\infty$ for some $(u_{\lambda,\beta},v_{\lambda,\beta})\in E$.  Clearly, $D[J_{\lambda,\beta}(u_{\lambda,\beta},v_{\lambda,\beta})]=0$ in $E^*$.  Suppose $u_{\lambda,\beta}=0$.  Then by the fact that
$E_a$ is embedded continuously into $\h$, we have
\begin{equation*}
u_n=o_n(1)\quad\text{strongly in }L^p_{loc}(\bbr^3)\quad\text{for }2\leq p<6.
\end{equation*}
Combining with the condition $(D_2)$ and the H\"older and the Sobolev inequalities, we get
\begin{eqnarray}
\int_{\bbr^3}|u_n|^4dx&=&\int_{\mathcal{D}_a}|u_n|^4dx
+\int_{\bbr^3\backslash\mathcal{D}_a}|u_n|^4dx\notag\\
&=&\int_{\bbr^3\backslash\mathcal{D}_a}|u_n|^4dx+o_n(1)\notag\\
&\leq&\bigg(\frac{1}{a_\infty}\bigg)^{\frac12}\int_{\bbr^3\backslash\mathcal{D}_a}[a(x)]^{\frac12}|u_n|^4dx+o_n(1)\notag\\
&\leq&\bigg(\frac{1}{a_\infty S^3}\bigg)^{\frac12}
\bigg(\int_{\bbr^3\backslash\mathcal{D}_a}a(x)|u_n|^2dx\bigg)^{\frac12}\bigg(\int_{\bbr^3}|\nabla u_n|^2dx\bigg)^{\frac32}+o_n(1). \label{eq0076}
\end{eqnarray}
Since $u_n=o_n(1)$ strongly in $L^p_{loc}(\bbr^3)$ for $2\leq p<6$, by the conditions $(D_2)$ and $(D_4)$, $\int_{\mathcal{D}_a}(\lambda a(x)+a_0(x))u_n^2dx=o_n(1)$.  It follows from \eqref{eq0009} and \eqref{eq0076} that
\begin{eqnarray}
\int_{\bbr^3}|u_n|^4dx&\leq&\bigg(\frac{1}{a_\infty S^3}\bigg)^{\frac12}
\bigg(\int_{\bbr^3\backslash\mathcal{D}_a}a(x)|u_n|^2dx\bigg)^{\frac12}
\bigg(\int_{\bbr^3}|\nabla u_n|^2dx\bigg)^{\frac32}+o_n(1)\notag\\
&\leq&\bigg(\frac{2}{a_\infty S^3\lambda}\bigg)^{\frac12}\|u_n\|_{a,\lambda}(\|u_n\|_{a,\lambda}^2+o_n(1))^{\frac32}+o_n(1)\notag\\
&\leq&\bigg(\frac{2}{a_\infty S^3\lambda}\bigg)^{\frac12}\|u_n\|_{a,\lambda}^4+o_n(1)\label{eq0078}
\end{eqnarray}
for $\lambda\geq\Lambda_1$.
Note that $\{(u_n, v_n)\}\subset\mathcal{N}_{\lambda,\beta}$ and  $\beta<0$,  from \eqref{eq0009} and \eqref{eq0078}, we have
\begin{eqnarray}      \label{eq0310}
\|u_n\|_{a,\lambda}^2\leq\bigg(\frac{2}{a_\infty S^3\lambda}\bigg)^{\frac12}\|u_n\|_{a,\lambda}^4+o_n(1)\leq4(m_a+m_b)\bigg(\frac{2}{a_\infty S^3\lambda}\bigg)^{\frac12}\|u_n\|_{a,\lambda}^2+o_n(1),
\end{eqnarray}
which then implies that there exists $\Lambda_2\geq\Lambda_1$ such that $\|u_n\|_{a,\lambda}=o_n(1)$ for $\lambda\geq\Lambda_2$ and $\beta<0$.  It follows from Lemma~\ref{lem0001}, the H\"older and the Sobolev inequalities and the boundedness of $\{(u_n,v_n)\}$ in $E$ that $\|u_n\|_4=o_n(1)$,  hence, $\mu_1\mu_2\|u_n\|_4^4\|v_n\|_4^4=o_n(1)$ for $\lambda\geq\Lambda_2$ and $\beta<0$.  However, it is impossible, since $\{(u_n, v_n)\}\subset\mathcal{N}_{\lambda,\beta}$, $\Lambda_2\geq\Lambda_1$ and Lemma~\ref{lem0009} holds for $\lambda\geq\Lambda_1$.  Therefore, there exists $\Lambda_2\geq\Lambda_1$ such that $u_{\lambda,\beta}\not=0$ for $\lambda\geq\Lambda_2$ and $\beta<0$.  Similarly, we can also show that $v_{\lambda,\beta}\not=0$ for $\lambda\geq\Lambda_2$ and $\beta<0$.  Since $(u_n,v_n)\rightharpoonup(u_{\lambda,\beta},v_{\lambda,\beta})$ weakly in $E$ as $n\to\infty$, by the fact that $E$ is embedded continuously into $\h\times\h$, we have $(u_n,v_n)\to(u_{\lambda,\beta},v_{\lambda,\beta})$ strongly in $L_{loc}^{p}(\bbr^3)\times L_{loc}^{p}(\bbr^3)$ as $n\to\infty$ for $2\leq p<6$.  It follows from the boundedness of $\{(u_n,v_n)\}$ in $E$ and the conditions $(D_2)$ and $(D_4)$ that $\int_{\mathcal{D}_a}a_0^-(x)u_n^2dx=\int_{\mathcal{D}_a}a_0^-(x)u_{\lambda,\beta}^2dx+o_n(1)$ and $\int_{\mathcal{D}_b}b_0^-(x)v_n^2dx=\int_{\mathcal{D}_b}b_0^-(x)v_{\lambda,\beta}^2dx+o_n(1)$, which together with
$D[J_{\lambda,\beta}(u_{\lambda,\beta},v_{\lambda,\beta})]=0$ in $E^*$, the Fatou lemma and the conditions $(D_2)$ and $(D_4)$, implies
\begin{eqnarray}
m_{\lambda,\beta}&\leq&J_{\lambda,\beta}(u_{\lambda,\beta},v_{\lambda,\beta})-\frac14\langle D[J_{\lambda,\beta}(u_{\lambda,\beta},v_{\lambda,\beta})],(u_{\lambda,\beta},v_{\lambda,\beta})\rangle_{E^*,E}\notag\\
&=&\frac14(\|u_{\lambda,\beta}\|_{a,\lambda}^2+\|v_{\lambda,\beta}\|_{b,\lambda}^2)\notag\\
&\leq&\liminf_{n\to\infty}\frac14(\|u_{n}\|_{a,\lambda}^2+\|v_{n}\|_{b,\lambda}^2)\label{eq7100}\\
&=&\liminf_{n\to\infty}(J_{\lambda,\beta}(u_{\lambda,\beta},v_{\lambda,\beta})-\frac14\langle D[J_{\lambda,\beta}(u_{n},v_{n})],(u_{n},v_{n})\rangle_{E^*,E})\notag\\
&=&m_{\lambda,\beta}+o_n(1).\notag
\end{eqnarray}
Therefore, $J_{\lambda,\beta}(u_{\lambda,\beta},v_{\lambda,\beta})=m_{\lambda,\beta}$.  Since $J_{\lambda,\beta}(|u_{\lambda,\beta}|,|v_{\lambda,\beta}|)=m_{\lambda,\beta}$ and $(|u_{\lambda,\beta}|,|v_{\lambda,\beta}|)\in\mathcal{N}_{\lambda,\beta}$, $(|u_{\lambda,\beta}|,|v_{\lambda,\beta}|)$ is a local minimizer of $J_{\lambda,\beta}(u,v)$ on $\mathcal{N}_{\lambda,\beta}$.  Note that by Lemma~\ref{lem0010}, $\mathcal{N}_{\lambda,\beta}$ is a natural constraint, we can follow the argument as used in \cite[Theorem~2.3]{BZ03} to show that $D[J_{\lambda,\beta}(|u_{\lambda,\beta}|,|v_{\lambda,\beta}|)]=0$ in $E^*$.  Thus, without loss of generality, we may assume $u_{\lambda,\beta}$ and $v_{\lambda,\beta}$ are both nonnegative.  Now, since $(u_{\lambda,\beta}, v_{\lambda,\beta})\in E$, by \eqref{eq0004} and \eqref{eq0005},  we have $u_{\lambda,\beta}, v_{\lambda,\beta}\in\h$.  It follows from the conditions $(D_1)$ and $(D_4)$ and the Calderon-Zygmund inequality that $u_{\lambda,\beta}, v_{\lambda,\beta}\in W_{loc}^{2,2}(\mathbb{R}^{3})$.  By combining the Sobolev embedding theorem and the Harnack inequality,
$u_{\lambda,\beta}$ and $v_{\lambda,\beta}$ are both positive.  Hence, $(u_{\lambda,\beta}, v_{\lambda,\beta})$ is a ground state solution of $(\mathcal{P}_{\lambda,\beta})$  for $\beta<0$ and $\lambda\geq\Lambda_2$.
It remains to show that \eqref{eq0018} and \eqref{eq0040} are true.  Indeed, let $\Omega_{a}''$ be a bounded domain with smooth boundary in $\bbr^3$ satisfying $\Omega_a\subset\Omega_{a}''\subset\Omega_a'$, $dist(\Omega_{a}'', \bbr^3\backslash\Omega_{a}')>0$ and $dist(\bbr^3\backslash\Omega_{a}'', \Omega_{a})>0$.
Then by a similar argument as \eqref{eq0017}, we can show that
\begin{equation}\label{eq0041}
\int_{\Omega_a''}|\nabla u_{\lambda,\beta}|^2+(\lambda a(x)+a_0(x))u_{\lambda,\beta}^2dx\geq\frac{\nu_a}{2}\int_{\Omega_a''}u_{\lambda,\beta}^2dx
\end{equation}
for $\lambda$ large enough.  Without loss of generality, we assume \eqref{eq0041} holds for $\lambda\geq\Lambda_2$.  Since $(u_{\lambda,\beta}, v_{\lambda,\beta})$ is a ground state solution for $\lambda\geq\Lambda_2$, by combining Lemma~\ref{lem0001}, \eqref{eq0041} and similar arguments of \eqref{eq0009} and \eqref{eq0332}, we can see that $\|(u_{\lambda,\beta}, v_{\lambda,\beta})\|^2\leq8(4(C_{a,0}+d_a)C_{a,b}+1)(m_a+m_b)$ and $8(4(C_{a,0}+d_a)C_{a,b}+1)(m_a+m_b)\geq\lambda\int_{\bbr^3\backslash\Omega_a''}a(x)u_{\lambda,\beta}^2dx$ for $\lambda\geq\Lambda_2$,which together with the conditions $(D_1)$-$(D_3)$, imply $\int_{(\bbr^3\backslash\Omega_a'')\cap\mathcal{D}_a}u_{\lambda,\beta}^2dx\to0$ as $\lambda\to+\infty$.  It follows from the condition $(D_2)$ and $8(4(C_{a,0}+d_a)C_{a,b}+1)(m_a+m_b)\geq\lambda\int_{\bbr^3\backslash\Omega_a''}a(x)u_{\lambda,\beta}^2dx$ once more that
\begin{equation}
\lim_{\lambda\to+\infty}\int_{\bbr^3\backslash\Omega_a''}u_{\lambda,\beta}^2dx=0.\label{eq0019}
\end{equation}
Now, we choose $\Psi\in C^\infty(\bbr^3, [0, 1])$ satisfying
\begin{equation}           \label{eq0086}
\Psi=\left\{\aligned&1,\quad&x\in\bbr^3\backslash\Omega_a',\\&0,\quad&x\in\Omega_a''.\endaligned\right.
\end{equation}
Then $u_{\lambda,\beta}\Psi\in E_a$.  Note  $D[J_{\lambda,\beta}(u_{\lambda,\beta},v_{\lambda,\beta})]=0$ in $E^*$ for $\lambda\geq\Lambda_2$, we have that
\begin{eqnarray*}
&&\int_{\bbr^3}(|\nabla u_{\lambda,\beta}|^2+(\lambda a(x)+a_0(x))u_{\lambda,\beta}^2)\Psi dx+\int_{\bbr^3}(\nabla u_{\lambda,\beta}\nabla \Psi) u_{\lambda,\beta} dx\\
&=&\mu_1\int_{\bbr^3}u_{\lambda,\beta}^4\Psi dx+\beta\int_{\bbr^3}v_{\lambda,\beta}^2u_{\lambda,\beta}^2\Psi dx.
\end{eqnarray*}
It follows from the H\"older and the Sobolev inequalities that
\begin{eqnarray}
&&\int_{\bbr^3\backslash\Omega_a'}(|\nabla u_{\lambda,\beta}|^2+(\lambda a(x)+a_0(x))u_{\lambda,\beta}^2)dx\notag\\
&\leq&\mu_1\int_{\bbr^3\backslash\Omega_a''}u_{\lambda,\beta}^4dx+\beta\int_{\bbr^3\backslash\Omega_a'}v_{\lambda,\beta}^2u_{\lambda,\beta}^2dx
+\int_{\Omega_a'\backslash\Omega_a''}|\nabla u_{\lambda,\beta}||\nabla\Psi||u_{\lambda,\beta}|dx\notag\\
&\leq&\mu_1S^{-\frac32}\|u_{\lambda,\beta}\|_a^3(\int_{\bbr^3\backslash\Omega_a''}u_{\lambda,\beta}^2dx)^{\frac12}
+\max_{\bbr^3}|\nabla\Psi|\|u_{\lambda,\beta}\|_a(\int_{\bbr^3\backslash\Omega_a''}u_{\lambda,\beta}^2dx)^{\frac12}.\label{eq0080}
\end{eqnarray}
Since $\|(u_{\lambda,\beta}, v_{\lambda,\beta})\|^2\leq8(4(C_{a,0}+d_a)C_{a,b}+1)(m_a+m_b)$,  we can conclude from \eqref{eq0022}, \eqref{eq0019} and \eqref{eq0080} that \eqref{eq0018} holds.  Similarly, we can also conclude that \eqref{eq0040} is true.
\end{proof}

\vskip0.1in

We close this section by

\vskip0.1in

\medskip\par\noindent{\bf Proof of Theorem~\ref{thm0002}:}\quad By Proposition~\ref{prop0002}, we know that there exists $\Lambda_2\geq\Lambda_1$ such that $(\mathcal{P}_{\lambda,\beta})$ has a ground state solution $(u_{\lambda,\beta},v_{\lambda,\beta})$ for $\lambda\geq\Lambda_2$ and $\beta<0$.  In what follows, we will show that $(u_{\lambda,\beta},v_{\lambda,\beta})$ has the concentration behaviors for $\lambda\to+\infty$ described as $(1)$-$(5)$.  We first verify $(3)$-$(5)$.
Let $(u_{\lambda_n,\beta},v_{\lambda_n,\beta})$ be the ground state solution of $(\mathcal{P}_{\lambda_n,\beta})$ obtained by Proposition~\ref{prop0002} with $\lambda_n\to+\infty$ as $n\to\infty$.   Then by Lemma~\ref{lem0008} and Proposition~\ref{prop0002}, $\{(u_{\lambda_n,\beta}, v_{\lambda_n,\beta})\}$ is bounded in $E$ with
\begin{equation}
\lim_{n\to+\infty}\int_{\bbr^3\backslash\Omega_a'}|\nabla u_{\lambda_n,\beta}|^2+({\lambda_n} a(x)+a_0(x))u_{\lambda_n,\beta}^2dx=0\label{eq0701}
\end{equation}
and
\begin{equation}
\lim_{n\to+\infty}\int_{\bbr^3\backslash\Omega_b'}|\nabla v_{\lambda_n,\beta}|^2+({\lambda_n} b(x)+b_0(x))v_{\lambda_n,\beta}^2dx=0\label{eq0702}.
\end{equation}
Without loss of generality, we assume $(u_{\lambda_n,\beta}, v_{\lambda_n,\beta})\rightharpoonup(u_{0,\beta}, v_{0,\beta})$ weakly in $E$ as $n\to\infty$ for some $(u_{0,\beta}, v_{0,\beta})\in E$.  By \eqref{eq0004} and \eqref{eq0005}, $(u_{0,\beta}, v_{0,\beta})\in \h\times\h$.  For the sake of clarity, the verification of $(3)$-$(5)$ is further performed through the following three steps.

\vskip0.1in

{\bf Step 1. }\quad We prove that $(u_{0,\beta}, v_{0,\beta})\in H_0^1(\Omega_{a})\times H_0^1(\Omega_{b})$ with $u_{0,\beta}=0$ on $\bbr^3\backslash\Omega_{a}$ and $v_{0,\beta}=0$ on $\bbr^3\backslash\Omega_{b}$.

Indeed, since $(u_{\lambda_n,\beta}, v_{\lambda_n,\beta})$ is a ground state solution of $(\mathcal{P}_{\lambda_n,\beta})$, by Lemma~\ref{lem0008} and a similar argument as  \eqref{eq0009}, we get that $4C_{a,b}(m_a+m_b)\geq(\|u_{\lambda_n,\beta}\|_2^2+\|v_{\lambda_n,\beta}\|_2^2)$.  Now, by the condition $(D_4)$ and a similar argument as \eqref{eq0009} again, we have
\begin{eqnarray*}
m_{a}+m_{b}&\geq&J_{\lambda_n,\beta}(u_{\lambda_n,\beta}, v_{\lambda_n,\beta})-\frac14\langle D[J_{\lambda_n,\beta}(u_{\lambda_n,\beta}, v_{\lambda_n,\beta})], (u_{\lambda_n,\beta}, v_{\lambda_n,\beta})\rangle_{E^*,E}\notag\\
&\geq&\frac14\int_{\bbr^3}(\lambda_n a(x)+a_0(x))u^2_{\lambda_n,\beta}dx+\frac14\int_{\bbr^3}({\lambda_n} b(x)+b_0(x))v_{\lambda_n,\beta}^2dx\notag\\
&\geq&\frac{\lambda_n}{8}\int_{\bbr^3}a(x)u^2_{\lambda_n,\beta}+b(x)v_{\lambda_n,\beta}^2dx-C(\|(u_{\lambda_n,\beta}\|_2^2+\|v_{\lambda_n,\beta})\|_2^2)
\notag\\
&\geq&\frac{\lambda_n}{8}\int_{\bbr^3}a(x)u^2_{\lambda_n,\beta}+b(x)v_{\lambda_n,\beta}^2dx-C'.
\end{eqnarray*}
It follows that $\lim_{n\to+\infty}\int_{\bbr^3}a(x)u_{\lambda_n,\beta}^2dx=\lim_{n\to+\infty}\int_{\bbr^3}b(x)v_{\lambda_n,\beta}^2dx=0$.  By the Fatou Lemma and the conditions $(D_1)$ and $(D_3)$, we can see that $u_{0,\beta}=0$ on $\bbr^3\backslash\Omega_a$ and $v_{0,\beta}=0$ on $\bbr^3\backslash\Omega_b$.  Since $(u_{0,\beta}, v_{0,\beta})\in \h$, by the condition $(D_3)$ again, we must have $u_{0,\beta}\in H^1_0(\Omega_a)$ and $v_{0,\beta}\in H^1_0(\Omega_b)$.

\vskip0.1in

{\bf Step 2. }\quad We prove that $(u_{\lambda_n,\beta},v_{\lambda_n,\beta})\to(u_{0,\beta},v_{0,\beta})$ strongly in $\h\times\h$ as $n\to\infty$ up to a subsequence.

Indeed,  by the choice of $\Omega_a'$ and the Sobolev embedding theorem, we can see that $u_{\lambda_n,\beta}\to u_{0,\beta}$ strongly in $L^2(\Omega_a')$ as $n\to\infty$ up to a subsequence.  Without loss of generality, we may assume $u_{\lambda_n,\beta}\to u_{0,\beta}$ strongly in $L^2(\Omega_a')$.  It follows from $u_{0,\beta}=0$ on $\bbr^3\backslash\Omega_a$, \eqref{eq0022} and \eqref{eq0701} that $u_{\lambda_n,\beta}\to u_{0,\beta}$ strongly in $L^2(\bbr^3)$ as $n\to\infty$.  By the H\"older and the Sobolev inequalities and the boundedness of  $\{(u_{\lambda_n,\beta}, v_{\lambda_n,\beta})\}$  in $E$, we can see that $u_{\lambda_n,\beta}\to u_{0,\beta}$ strongly in $L^4(\bbr^3)$ as $n\to\infty$.  On the other hand, by a similar argument as \eqref{eq7100}, we have
\begin{eqnarray*}
\int_{\bbr^3}|\nabla u_{0,\beta}|^2+a_0(x)u_{0,\beta}^2dx&=&\int_{\bbr^3}|\nabla u_{0,\beta}|^2dx+(\lambda_na(x)+a_0(x))u_{0,\beta}^2dx\\
&\leq&\liminf_{n\to\infty}\int_{\bbr^3}|\nabla u_{\lambda_n,\beta}|^2+(\lambda_na(x)+a_0(x))u_{\lambda_n,\beta}^2dx,
\end{eqnarray*}
which together with $D[J_{\lambda_n,\beta}(u_{\lambda_n,\beta}, v_{\lambda_n,\beta})]=0$ in $E^*$ and $\beta<0$, implies
\begin{equation*}
\int_{\bbr^3}|\nabla u_{0,\beta}|^2+a_0(x)u_{0,\beta}^2dx\leq\liminf_{n\to\infty}\mu_1\int_{\bbr^3}u_{\lambda_n,\beta}^4dx=\mu_1\int_{\Omega_{a}}u_{0,\beta}^4dx.
\end{equation*}
Note that $(u_{0,\beta},v_{0,\beta})\in H_0^1(\Omega_{a})\times H_0^1(\Omega_{b})$ with $u_{0,\beta}=0$ on $\bbr^3\backslash\Omega_a$ and $v_{0,\beta}=0$ on $\bbr^3\backslash\Omega_b$ and $D[J_{\lambda_n,\beta}(u_{\lambda_n,\beta}, v_{\lambda_n,\beta})]=0$ in $E^*$, by the condition $(D_3)$, we can show that $I_{\Omega_{a}}'(u_{0,\beta})=0$ in $H^{-1}(\Omega_{a})$ and $I_{\Omega_{b}}'(v_{0,\beta})=0$ in $H^{-1}(\Omega_{b})$.  Recalling that $u_{\lambda_n,\beta}\to u_{0,\beta}$ strongly in $L^2(\bbr^3)$ as $n\to\infty$, $\{(u_{\lambda_n,\beta}, v_{\lambda_n,\beta})\}$ is bounded in $E$ and the conditions $(D_2)$ and $(D_4)$ hold, we must have $\int_{\mathcal{D}_a}a_0^-(x)u_{\lambda_n,\beta}^2dx=\int_{\mathcal{D}_a}a_0^-(x)u_{0,\beta}^2dx+o_n(1)$.  It follows from the conditions $(D_2)$ and $(D_4)$ again and the Fatou Lemma that
\begin{eqnarray}
\int_{\bbr^3}|\nabla u_{0,\beta}|^2+a_0(x)u_{0,\beta}^2dx&=&\liminf_{n\to\infty}\int_{\bbr^3}|\nabla u_{\lambda_n,\beta}|^2+(\lambda_na(x)+a_0(x))u_{\lambda_n,\beta}^2dx\notag\\
&\geq&\liminf_{n\to\infty}(\int_{\bbr^3}|\nabla u_{\lambda_n,\beta}|^2+\frac{\lambda_n}{2} a(x)u_{\lambda_n,\beta}^2dx)\notag\\
&&+\int_{\bbr^3}a_0^+(x)u_{0,\beta}^2dx+\int_{\mathcal{D}_a}a_0^-(x)u_{0,\beta}^2dx\label{eq0703}.
\end{eqnarray}
By the conditions $(D_1)$-$(D_3)$, the Fatou lemma and the fact $u_{0,\beta}=0$ on $\bbr^3\backslash\Omega_a$, we can see from \eqref{eq0703} that $\nabla u_{\lambda_n,\beta}\to\nabla u_{0,\beta}$ strongly in $L^2(\bbr^3)$ as $n\to\infty$ up to a subsequence, which then implies $\liminf_{n\to\infty}\int_{\bbr^3}\lambda_na(x)u_{\lambda_n,\beta}^2dx=0$ and $\liminf_{n\to\infty}\int_{\bbr^3}a_0(x)u_{\lambda_n,\beta}^2dx=\int_{\bbr^3}a_0(x)u_{0,\beta}^2dx$.
These together with the conditions $(D_1)$-$(D_4)$, $u_{0,\beta}=0$ on $\bbr^3\backslash\Omega_a$ and $u_{\lambda_n,\beta}\to u_{0,\beta}$ strongly in $L^2(\bbr^3)$ as $n\to\infty$, imply $u_{\lambda_n,\beta}\to u_{0,\beta}$ strongly in $E_a$ as $n\to\infty$ up to a subsequence.  Without loss of generality, we assume $u_{\lambda_n,\beta}\to u_{0,\beta}$ strongly in $E_a$ as $n\to\infty$.  Similarly, we also have $v_{\lambda_n,\beta}\to v_{0,\beta}$ strongly in $E_b$ as $n\to\infty$, that is, $(u_{\lambda_n,\beta},v_{\lambda_n,\beta})\to(u_{0,\beta},v_{0,\beta})$ strongly in $E$ as $n\to\infty$.  Since $E$ is embedded continuously into $\h\times\h$, $(u_{\lambda_n,\beta},v_{\lambda_n,\beta})\to(u_{0,\beta},v_{0,\beta})$ strongly in $\h\times\h$ as $n\to\infty$.

\vskip0.1in

{\bf Step 3. }\quad We prove that $u_{0,\beta}$  and $v_{0,\beta}$   are    least energy nonzero critical points of $I_{\Omega_{a}}(u)$ and    $I_{\Omega_{b}}(v)$, respectively.

Indeed, since $(u_{\lambda_n,\beta},v_{\lambda_n,\beta})$ is the ground state solution of $(\mathcal{P}_{\lambda_n,\beta})$, by Lemma~\ref{lem0008}, we can see that
\begin{equation}    \label{eq0351}
\|u_{\lambda_n,\beta}\|_{a,\lambda_n}+\|v_{\lambda_n,\beta}\|_{b,\lambda_n}=4(m_a+m_b)+o_n(1).
\end{equation}
By a similar argument as used in Step 2, we can show that
\begin{equation*}
\lambda_n\int_{\bbr^3}a(x)u_{\lambda_n,\beta}^2dx=\lambda_n\int_{\bbr^3}b(x)v_{\lambda_n,\beta}^2dx=o_n(1)
\end{equation*}
and
\begin{equation*}
\int_{\bbr^3}a_0(x)u_{\lambda_n,\beta}^2dx=\int_{\bbr^3}a_0(x)u_{0,\beta}^2dx+o_n(1)\quad\text{and}\quad
\int_{\bbr^3}b_0(x)v_{\lambda_n,\beta}^2dx=\int_{\bbr^3}b_0(x)v_{0,\beta}^2dx+o_n(1).
\end{equation*}
These together with Step 2 and \eqref{eq0351}, imply
\begin{equation}    \label{eq0084}
\int_{\bbr^3}|\nabla u_{0,\beta}|^2+a_0(x)u_{0,\beta}^2dx+
\int_{\bbr^3}|\nabla v_{0,\beta}|^2+b_0(x)v_{0,\beta}^2dx=4(m_a+m_b).
\end{equation}
We claim that
\begin{equation}      \label{eq0204}
\int_{\bbr^3}u_{\lambda_n,\beta}^4dx\geq C+o_n(1)\quad\text{and}\quad\int_{\bbr^3}v_{\lambda_n,\beta}^4dx\geq C+o_n(1).
\end{equation}
Indeed, suppose the contrary, we have either $\int_{\bbr^3}u_{\lambda_n,\beta}^4dx=o_n(1)$ or $\int_{\bbr^3}v_{\lambda_n,\beta}^4dx=o_n(1)$ up to a subsequence.  Without loss of generality, we assume $\lim_{n\to\infty}\int_{\bbr^3}u_{\lambda_n,\beta}^4dx=0$.  By the boundedness of $\{(u_{\lambda_n,\beta}, v_{\lambda_n,\beta})\}$ in $E$ and the H\"older and Sobolev inequalities, $\beta\int_{\bbr^3}u_{\lambda_n,\beta}^2v_{\lambda_n,\beta}^2=o_n(1)$, which implies $J_{\lambda_n,\beta}(u_{\lambda_n,\beta},v_{\lambda_n,\beta})=I_{b,\lambda_n}(v_{\lambda_n,\beta})+o_n(1)$.  By Lemma~\ref{lem0001} and $\mathcal{N}_{\lambda,\beta}\subset\mathcal{A}_{\beta}$, for every $n$, there exists a unique $t^*_n(\beta)>0$ such that $t^*_n(\beta)u_{\lambda_n,\beta}\in\mathcal{N}_{a,\lambda_n}$.  It follows from Lemma~\ref{lem0007} and $\beta<0$ that
\begin{eqnarray}
J_{\lambda_n,\beta}(u_{\lambda_n,\beta},v_{\lambda_n,\beta})&\geq& J_{\lambda_n,\beta}(t^*_n(\beta)u_{\lambda_n,\beta},v_{\lambda_n,\beta})\notag\\
&\geq&I_{a,\lambda_n}(t^*_n(\beta)u_{\lambda_n,\beta})+I_{b,\lambda_n}(v_{\lambda_n,\beta})\notag\\
&\geq&m_{a,\lambda_n}+I_{b,\lambda_n}(v_{\lambda_n,\beta})\notag\\
&=&m_{a,\lambda_n}+J_{\lambda_n,\beta}(u_{\lambda_n,\beta},v_{\lambda_n,\beta})+o_n(1)\notag\\
&=&m_{a}+J_{\lambda_n,\beta}(u_{\lambda_n,\beta},v_{\lambda_n,\beta})+o_n(1).\label{eq0206}
\end{eqnarray}
Since $m_{a}>0$, $\eqref{eq0206}$ is impossible for $n$ large enough.  Now, \eqref{eq0204} together with Steps 1-2, implies
\begin{equation}      \label{eq0066}
\int_{\Omega_a}u_{0,\beta}^4dx\geq C>0\quad\text{and}\int_{\Omega_b}v_{0,\beta}^4dx\geq C>0.
\end{equation}
Note that $I_{\Omega_{a}}'(u_{0,\beta})=0$ in $H^{-1}(\Omega_{a})$ and $I_{\Omega_{b}}'(v_{0,\beta})=0$ in $H^{-1}(\Omega_{b})$, by \eqref{eq0066}, we have
\begin{equation*}
\int_{\bbr^3}|\nabla u_{0,\beta}|^2+a_0(x)u_{0,\beta}^2dx\geq4m_a\quad\text{and}\quad\int_{\bbr^3}|\nabla v_{0,\beta}|^2+b_0(x)v_{0,\beta}^2dx\geq4m_b.
\end{equation*}
It follows from \eqref{eq0084} that $u_{0,\beta}$ is a least energy nonzero critical point of $I_{\Omega_{a}}(u)$ and $v_{0,\beta}$ is a least energy nonzero critical point of $I_{\Omega_{b}}(v)$.

\vskip0.1in

We close the proof of Theorem \ref{thm0002} by verifying $(1)$ and $(2)$.  Supposing the contrary, there exists $\{\lambda_n\}$ with $\lambda_n\to+\infty$ as $n\to\infty$ such that one of the following cases must happen:
\begin{enumerate}
\item[$(a)$] $\int_{\bbr^3\backslash\Omega_{a}}|\nabla u_{\lambda_n,\beta}|^2+u_{\lambda_n,\beta}^2dx\geq C+o_n(1)$;
\item[$(b)$] $\int_{\bbr^3\backslash\Omega_{b}}|\nabla v_{\lambda_n,\beta}|^2+v_{\lambda_n,\beta}^2dx\geq C+o_n(1)$;
\item[$(c)$] $|\int_{\Omega_{a}}|\nabla u_{\lambda_n,\beta}|^2+a_0(x)u_{\lambda_n,\beta}^2dx-4m_{a}|\geq C+o_n(1)$;
\item[$(d)$] $|\int_{\Omega_{b}}|\nabla v_{\lambda,\beta}|^2+b_0(x)v_{\lambda_n,\beta}^2dx-4m_{b}|\geq C+o_n(1)$.
\end{enumerate}
By Steps 1-3 and the condition $(D_3)$, it is easy to see that $(c)$ and $(d)$ can not hold, which then implies that we must have $(a)$ or $(b)$.  Since \eqref{eq0351} holds for $\{(u_{\lambda_n,\beta}, v_{\lambda_n,\beta})\}$, by Steps 2-3 and the condition $(D_3)$, we can see that
\begin{equation*}
\int_{\bbr^3\backslash\Omega_a}|\nabla u_{\lambda_n,\beta}|^2+(\lambda_na(x)+a_0(x))u_{\lambda_n,\beta}^2dx+
\int_{\bbr^3\backslash\Omega_b}|\nabla v_{\lambda_n,\beta}|^2+(\lambda_nb(x)+b_0(x))v_{\lambda_n,\beta}^2dx\to0
\end{equation*}
as $n\to\infty$.  It follows from the conditions $(D_2)$ and $(D_4)$ and Steps 1-2 that
\begin{equation*}
\int_{(\bbr^3\backslash\Omega_a)\cap\mathcal{D}_a}a_0^-(x)u_{\lambda_n,\beta}^2dx
=\int_{(\bbr^3\backslash\Omega_b)\cap\mathcal{D}_b}b_0^-(x)v_{\lambda_n,\beta}^2dx=o_n(1),
\end{equation*}
which then together with the conditions $(D_2)$ and $(D_4)$ and and Steps 1-2 once more, implies
\begin{eqnarray*}
&&\int_{\bbr^3\backslash\Omega_a}|\nabla u_{\lambda_n,\beta}|^2+(\lambda_na(x)+a_0(x))u_{\lambda_n,\beta}^2dx
+\int_{\bbr^3\backslash\Omega_b}|\nabla v_{\lambda_n,\beta}|^2+(\lambda_nb(x)+b_0(x))v_{\lambda_n,\beta}^2dx\\
&\geq&\int_{\bbr^3\backslash\Omega_a}|\nabla u_{\lambda_n,\beta}|^2dx
+\int_{\bbr^3\backslash\Omega_a}\frac{\lambda_n}{2}a(x)u_{\lambda_n,\beta}^2dx
+\int_{\bbr^3\backslash\Omega_b}|\nabla v_{\lambda_n,\beta}|^2dx\\
&&+\int_{\bbr^3\backslash\Omega_b}\frac{\lambda_n}{2}b(x)v_{\lambda_n,\beta}^2dx+o_n(1)\\
&=&\int_{\bbr^3\backslash\Omega_a}|\nabla u_{\lambda_n,\beta}|^2+u_{\lambda_n,\beta}^2dx
+\int_{\bbr^3\backslash\Omega_b}|\nabla v_{\lambda_n,\beta}|^2+v_{\lambda_n,\beta}^2dx+o_n(1),
\end{eqnarray*}
Thus,
$\int_{\bbr^3\backslash\Omega_a}|\nabla u_{\lambda_n,\beta}|^2+u_{\lambda_n,\beta}^2dx
+\int_{\bbr^3\backslash\Omega_b}|\nabla v_{\lambda_n,\beta}|^2+v_{\lambda_n,\beta}^2dx\to0$
as $n\to\infty$ and it is a contradiction.  We now complete the proof by taking $\Lambda_*=\Lambda_2$.
\qquad\raisebox{-0.5mm}{\rule{1.5mm}{4mm}}\vspace{6pt}

\section{Multi-bump solutions}
The main task in this section is to find multi-bump solutions to  $(\mathcal{P}_{\lambda,\beta})$ described as in Theorem~\ref{thm0001}.  For the sake of convenience, in the present section, we always assume the conditions $(D_1)$-$(D_2)$, $(D_3')$, $(D_4)$ and $(D_5')$ hold.  Due to the conditions $(D_3')$ and $(D_5')$,  in the present section, $\Omega_a'$ and $\Omega_b'$ will be chosen as $(I)$-$(III)$ given in section~2.


\subsection{The penalized functional and the $(PS)$ condition}
Since we want to find multi-bump solutions of $(\mathcal{P}_{\lambda,\beta})$ described as in Theorem~\ref{thm0001}, we will make some modifications on $J_{\lambda,\beta}(u,v)$.   Similar technique was developed by del Pino and Felmer \cite{DF96} and was also used in several other literatures, see for example \cite{BT13,DT03,GT121,ST09} and the references therein.

\vskip0.1in

Let $J_a\times J_b$ be a given subset of $\{1,\cdots,n_a\}\times\{1,\cdots,n_b\}$ with $J_a\not=\emptyset$ and $J_b\not=\emptyset$.  Without loss of generality, we assume $J_a\times J_b=\{1,\cdots,k_a\}\times\{1,\cdots,k_b\}$ with $1\leq k_a\leq n_a$ and $1\leq k_b\leq n_b$.  Denote
$\Omega_a^{J_a}=\underset{i_a=1}{\overset{k_a}{\cup }}\Omega_{a,i_a}'$ and $\Omega_b^{J_b}=\underset{j_b=1}{\overset{k_b}{\cup }}\Omega_{b,j_b}'$.  We also denote  the characteristic functions of $\Omega_a^{J_a}$ and $\Omega_b^{J_b}$ by $\chi_{\Omega_a^{J_a}}$ and $\chi_{\Omega_b^{J_b}}$, respectively.  Now, let
\begin{equation}    \label{eq0039}
\delta_\beta^2=\frac{C_{a,b}}{2}\min\bigg\{1,\frac{1}{\mu_1+2|\beta|},\frac{1}{\mu_2+2|\beta|}\bigg\},
\end{equation}
where $C_{a,b}$ is given by Lemma~\ref{lem0001},
and define $f_a(x,t)=\chi_{\Omega_a^{J_a}}(t^+)^3+(1-\chi_{\Omega_a^{J_a}})f(t)$, $f_b(x,t)=\chi_{\Omega_b^{J_b}}(t^+)^3+(1-\chi_{\Omega_b^{J_b}})f(t)$ and $h(x,t,s)=(\chi_{\Omega_b^{J_b}}+\chi_{\Omega_a^{J_a}})t^+s^+
+(1-\chi_{\Omega_a^{J_a}})(1-\chi_{\Omega_b^{J_b}})h(t,s)$, where $t^+=\max\{0, t\}$, $s^+=\max\{0, s\}$,
\begin{equation*}
f(t)=\left\{\aligned
&0,\quad&t\leq0,\\
&t^3,\quad&0\leq t\leq\delta_\beta,\\
&\delta_\beta^2t,\quad&t\geq\delta_\beta,
\endaligned\right.
\quad\text{and}\quad
h(t,s)=\left\{\aligned
&0,\quad&\min\{t,s\}\leq0,\\
&ts,\quad&0\leq t,s\leq\delta_\beta,\\
&\delta_\beta t,\quad&0\leq t\leq\delta_\beta\leq s,\\
&\delta_\beta s,\quad&0\leq s\leq\delta_\beta\leq t,\\
&\delta_\beta^2,\quad&\delta_\beta\leq t,s.
\endaligned\right.
\end{equation*}
Then it is easy to see that $f_a(x,t)$ and $f_b(x,t)$ are the modifications of $t^3$ and $h(x,t,s)$ is the modification of $ts$.  Let us consider the following functional defined on $E$,
\begin{equation*}
J_{\lambda,\beta}^*(u,v)=\frac12\|u\|_{a,\lambda}^2+\frac12\|v\|_{b,\lambda}^2
-\mu_1\int_{\bbr^3}F_a(x,u)dx-\mu_2\int_{\bbr^3}F_b(x,v)dx-\beta\int_{\bbr^3}H(x,u,v)dx,
\end{equation*}
where $F_a(x,u)=\int_0^uf_a(x,t)dt$, $F_b(x,v)=\int_0^vf_b(x,t)dt$ and $H(x,u,v)=2\int_0^u\int_0^vh(x,t,s)dsdt$.  Clearly,
by the construction of $f_a(x,t)$, $f_b(x,t)$ and $h(x,t,s)$, we can see that
\begin{eqnarray*}
&&0\leq\int_{\bbr^3\backslash\Omega_a^{J_a}}F_a(x,u)dx\leq\frac{\delta_\beta^2}2\|u^+\|_2^2,\quad
0\leq\int_{\bbr^3\backslash\Omega_a^{J_b}}F_b(x,v)dx\leq\frac{\delta_\beta^2}2\|v^+\|_2^2,\\
&&0\leq\int_{\bbr^3\backslash(\Omega_a^{J_a}\cup\Omega_a^{J_b})}H(x,u,v)dx\leq2\delta_\beta^2\|u^+\|_2\|v^+\|_2\leq\delta_\beta^2(\|u^+\|_2^2+\|v^+\|_2^2).
\end{eqnarray*}
On the other hand,by Lemma~\ref{lem0001}, we have
\begin{equation*}
\|u^+\|_2^2\leq C_{a,b}^{-1}\|u\|_{a,\lambda}^2\leq\lambda C_{a,b}^{-1}\|u\|_a^2\quad\text{for }\lambda\geq\Lambda_1\text{ and }u\in E_a
\end{equation*}
and
\begin{equation*}
\|v^+\|_2^2\leq C_{a,b}^{-1}\|v\|_{b,\lambda}^2\leq\lambda C_{a,b}^{-1}\|v\|_b^2\quad\text{for }\lambda\geq\Lambda_1\text{ and }v\in E_b.
\end{equation*}
It follows that $J_{\lambda,\beta}^*(u,v)$ is well defined on $E$ for $\lambda\geq\Lambda_1$ and $\beta<0$.  Moreover, by a standard argument, we can see that for $\lambda\geq\Lambda_1$ and $\beta<0$, $J_{\lambda,\beta}^*(u,v)$ is $C^1$ on $E$ and the critical point of $J_{\lambda,\beta}^*(u,v)$ is the solution of the following two-component systems:

\begin{equation*}
\left\{\aligned&\Delta u-(\lambda a(x)+a_0(x))u+\mu_1f_a(x,u)+2\beta\int_0^vh(x,u,s)ds=0\quad&\text{in }\bbr^3,\\
&\Delta v-(\lambda b(x)+b_0(x))v+\mu_2f_b(x,v)+2\beta\int_0^uh(x,t,v)dt=0\quad&\text{in }\bbr^3,\\
&u,v\in\h,\quad u,v\geq0\quad\text{in }\bbr^3.\endaligned\right.\eqno{(\mathcal{P}_{\lambda,\beta}^*)}
\end{equation*}
In what follows, we will make some investigations  on the functional $J_{\lambda,\beta}^*(u,v)$.

\vskip0.1in

\begin{lemma}
\label{lem0002}Assume $(u,v)\in E$.  Then
\begin{eqnarray*}
&&\int_{\bbr^3}\frac14f_a(x,u)u-F_a(x,u)dx\geq-\frac{\delta_\beta^2}4\|u^+\|_2^2,\quad
\int_{\bbr^3}\frac14f_b(x,v)v-F_b(x,v)dx\geq-\frac{\delta_\beta^2}4\|v^+\|_2^2,\\
&&0\geq\int_{\bbr^3}\frac{u}{2}\int_0^vh(x,u,s)ds+\frac{v}{2}\int_0^uh(x,t,v)dt-H(x,u,v)dx\geq
-\delta_\beta^2\|u^+v^+\|_1.
\end{eqnarray*}
\end{lemma}
\begin{proof}
By the construction of $f_a(x,t)$, it is easy to see that $\frac14f_a(x,t)t-F_a(x,t)=0$ for $x\in\Omega_a^{J_a}$.  If $x\not\in\Omega_a^{J_a}$, then   by the construction of $f_a(x,t)$, we have
\begin{equation*}
\frac14f_a(x,t)t-F_a(x,t)=\left\{\aligned&0,\quad&t\leq\delta_\beta,\\
&\frac{\delta_\beta^4}{4}-\frac{\delta_\beta^2}{4}(t^+)^2,\quad&\delta_\beta\leq t.
\endaligned\right.
\end{equation*}
It follows that for every $u\in E_a$, we have
\begin{equation*}
\int_{\bbr^3}\frac14f_a(x,u)u-F_a(x,u)dx=
\int_{\{u(x)\geq\delta_\beta\}\cap(\bbr^3\backslash\Omega_a^{J_a})}\frac{\delta_\beta^4}{4}-\frac{\delta_\beta^2}{4}(u^+)^2dx\geq
-\frac{\delta_\beta^2}4\|u^+\|_2^2.
\end{equation*}
By a similar argument, for every $v\in E_b$, we have
\begin{equation*}
\int_{\bbr^3}\frac14f_b(x,v)v-F_b(x,v)dx\geq-\frac{\delta_\beta^2}4\|v^+\|_2^2.
\end{equation*}
On the other hand, since $\Omega_a^{J_a}\cap\Omega_b^{J_b}=\emptyset$, by the construction of $h(x,t,s)$, we can see that $\frac{t}{2}\int_0^sh(x,t,\tau)d\tau+\frac{s}{2}\int_0^th(x,\tau,s)d\tau-H(x,t,s)=0$ for $x\in\Omega_a^{J_a}\cup\Omega_{b}^{J_b}$.  If $x\not\in\Omega_a^{J_a}\cup\Omega_{b}^{J_b}$, then also by the construction of $h(x,t,s)$, we have
\begin{eqnarray*}
&&\frac{t}{2}\int_0^sh(x,t,\tau)d\tau+\frac{s}{2}\int_0^th(x,\tau,s)d\tau-H(x,t,s)\\
&=&\left\{\aligned
&0,\quad &t,s\leq\delta_\beta,\\
&\frac{(t^+)^2\delta_\beta}4(\delta_\beta-s),\quad &t\leq\delta_\beta\leq s,\\
&\frac{(s^+)^2\delta_\beta}4(\delta_\beta-t),\quad &s\leq\delta_\beta\leq t,\\
&\frac{\delta_\beta^2}{4}[(t-\delta_\beta)(\delta_\beta-2s)+(s-\delta_\beta)(\delta_\beta-2t)],\quad &\delta_\beta\leq t, s.
\endaligned\right.
\end{eqnarray*}
It follows that for every $(u,v)\in E$, we have
\begin{equation*}
0\geq\int_{\bbr^3}\frac{u}{2}\int_0^vh(x,u,s)ds+\frac{v}{2}\int_0^uh(x,t,v)dt-H(x,u,v)dx\geq
-\delta_\beta^2\|u^+v^+\|_1,
\end{equation*}
which completes the proof.
\end{proof}

\vskip0.1in

With Lemma~\ref{lem0002} in hands, we can verify that $J_{\lambda,\beta}^*(u,v)$ actually satisfies the $(PS)$ condition for $\lambda\geq\Lambda_1$ and $\beta<0$.

\vskip0.1in

\begin{lemma}
\label{lem0003}Assume $\lambda\geq\Lambda_1$ and $\beta<0$.  Then $J_{\lambda,\beta}^*(u,v)$ satisfies the $(PS)_c$ condition for all $c\in\bbr$, that is, every $\{(u_n,v_n)\}\subset E$ satisfying $J_{\lambda,\beta}^*(u_n,v_n)=c+o_n(1)$ and $D[J_{\lambda,\beta}^*(u_n,v_n)]=o_n(1)$ strongly in $E^*$ has a strongly convergent subsequence in $E$.
\end{lemma}
\begin{proof}
Suppose $\{(u_n,v_n)\}\subset E$ satisfying $J_{\lambda, \beta}^*(u_n,v_n)=c+o_n(1)$ and $D[J_{\lambda,\beta}^*(u_n,v_n)]=o_n(1)$ strongly in $E^*$.  Then by $\beta<0$, Lemmas~\ref{lem0001} and \ref{lem0002} and a similar argument of \eqref{eq0009}, we have
\begin{eqnarray}
c+o_n(1)+o_n(1)\|(u_n,v_n)\|
\geq\frac{1}{4}(1-\mu_1\delta_\beta^2C_{a,b}^{-1})\|u_n\|^2_{a,\lambda}+\frac{1}{4}(1-\mu_2\delta_\beta^2C_{a,b}^{-1})\|v_n\|^2_{b,\lambda}.\label{eq0109}
\end{eqnarray}
It follows from Lemma~\ref{lem0001} and \eqref{eq0039} that
$c+o_n(1)+o_n(1)\|(u_n,v_n)\|\geq\frac{C_{a,b}}{8}(\|u_n\|_2^2+\|v_n\|_2^2)$.
This together with Lemma~\ref{lem0001} and the condition $(D_4)$, implies
\begin{eqnarray*}
c+o_n(1)+o_n(1)\|(u_n,v_n)\|&\geq&\frac{1}{4}(1-\mu_1\delta_\beta^2C_{a,b}^{-1})\|u_n\|^2_{a,\lambda}
+\frac{1}{4}(1-\mu_2\delta_\beta^2C_{a,b}^{-1})\|v_n\|^2_{b,\lambda}\\
&\geq&\frac{1}{4}(1-\mu_1\delta_\beta^2C_{a,b}^{-1})\|u_n\|_a^2
+\frac{1}{4}(1-\mu_2\delta_\beta^2C_{a,b}^{-1})\|v_n\|_b^2\\
&&-C(\|u_n\|_2^2+\|v_n\|_2^2)\\
&\geq&\frac18\|(u_n,v_n)\|^2-C'(c+o_n(1)+o_n(1)\|(u_n,v_n)\|),
\end{eqnarray*}
since $\lambda\geq\Lambda_1$.  Thus, $\{(u_n,v_n)\}$ is bounded in $E$ and $(u_n,v_n)\rightharpoonup(u_0,v_0)$ weakly in $E$ as $n\to\infty$ for some $(u_0,v_0)\in E$ up to a subsequence.  Without loss of generality, we assume $(u_n,v_n)\rightharpoonup(u_0,v_0)$ weakly in $E$ as $n\to\infty$.  Since $D[J_{\lambda,\beta}^*(u_n,v_n)]=o_n(1)$ strongly in $E^*$, it is easy to see that $D[J_{\lambda,\beta}^*(u_0,v_0)]=0$ in $E^*$, which implies
\begin{eqnarray}
o_n(1)&=&\langle D[J_{\lambda,\beta}^*(u_n,v_n)]-D[J_{\lambda,\beta}^*(u_0,v_0)], (u_n,v_n)-(u_0,v_0)\rangle_{E^*,E}\notag\\
&=&\|u_n-u_0\|_{a,\lambda}^2+\|v_n-v_0\|_{b,\lambda}^2-\mu_1\int_{\bbr^3}(f_a(x,u_n)-f_a(x,u_0))(u_n-u_0)dx\notag\\
&&-\mu_2\int_{\bbr^3}(f_b(x,v_n)-f_b(x,v_0))(v_n-v_0)dx\notag\\
&&-2\beta\int_{\bbr^3}(\int_0^{v_n}h(x,u_n,s)ds-\int_0^{v_0}h(x,u_0,s)ds)(u_n-u_0)dx\notag\\
&&-2\beta\int_{\bbr^3}(\int_0^{u_n}h(x,t,v_n)dt-\int_0^{u_0}h(x,t,v_0)dt)(v_n-v_0)dx.\label{eq0010}
\end{eqnarray}
By the construction of $f_a(x,t)$, we can see that
\begin{eqnarray*}
&&|\int_{\bbr^3}(f_a(x,u_n)-f_a(x,u_0))(u_n-u_0)dx|\\
&\leq&\int_{\Omega_a^{J_a}}|(f_a(x,u_n)-f_a(x,u_0))(u_n-u_0)|dx
+\int_{\bbr^3\backslash\Omega_a^{J_a}}|(f_a(x,u_n)-f_a(x,u_0))(u_n-u_0)|dx\\
&\leq&\int_{\Omega_a^{J_a}}(|u_n|^3+|u_0|^3)|u_n-u_0|dx+2\delta_\beta^2\int_{\bbr^3\backslash\Omega_a^{J_a}}|u_0||u_n-u_0|dx+
\delta_\beta^2\|u_n-u_0\|_2^2.
\end{eqnarray*}
Since $(u_n,v_n)\rightharpoonup(u_0,v_0)$ weakly in $E$ as $n\to\infty$, by \eqref{eq0004} and the Sobolev embedding theorem, we have
\begin{equation}    \label{eq0011}
u_n\to u_0 \text{ strongly in }L^p_{loc}(\bbr^3)\text{ as }n\to\infty\text{ for }p\in[1, 6).
\end{equation}
Thus, $\int_{\Omega_a^{J_a}}(|u_n|^3+|u_0|^3)|u_n-u_0|dx=o_n(1)$ due to the choice of $\Omega_a^{J_a}$ and the H\"older inequality.  On the other hand, we also see from \eqref{eq0011} and the H\"older inequality that $2\delta_\beta^2\int_{\bbr^3\backslash\Omega_a^{J_a}}|u_0||u_n-u_0|dx=o_n(1)$.  Hence, we have
\begin{equation}     \label{eq0012}
|\int_{\bbr^3}(f_a(x,u_n)-f_a(x,u_0))(u_n-u_0)dx|\leq\delta_\beta^2\|u_n-u_0\|_2^2+o_n(1).
\end{equation}
Since \eqref{eq0005} holds, we can also obtain the following  estimates in a simiar way:
\begin{equation}     \label{eq0013}
|\int_{\bbr^3}(f_b(x,v_n)-f_b(x,v_0))(v_n-v_0)dx|\leq\delta_\beta^2\|v_n-v_0\|_2^2+o_n(1).
\end{equation}
On the other hand, by the construction of $h(x,t,s)$, we can see that
\begin{eqnarray}
&&|\int_{\bbr^3}(\int_0^{v_n}h(x,u_n,s)ds-\int_0^{v_0}h(x,u_0,s)ds)(u_n-u_0)dx|\notag\\
&\leq&\delta_\beta^2\int_{\bbr^3}|u_n-u_0||v_n-v_0|dx+2\delta_\beta^2\int_{\bbr^3}|v_0||u_n-u_0|dx+o_n(1)\label{eq0042}
\end{eqnarray}
and
\begin{eqnarray}
&&|\int_{\bbr^3}(\int_0^{u_n}h(x,t,v_n)dt-\int_0^{u_0}h(x,t,v_0)dt)(v_n-v_0)dx|\notag\\
&\leq&\delta_\beta^2\int_{\bbr^3}|u_n-u_0||v_n-v_0|dx+2\delta_\beta^2\int_{\bbr^3}|u_0||v_n-v_0|dx+o_n(1).\label{eq0043}
\end{eqnarray}
By using similar arguments of \eqref{eq0012} and \eqref{eq0013}, we can see from \eqref{eq0042} and \eqref{eq0043} that
\begin{eqnarray}
&&|\int_{\bbr^3}(\int_0^{v_n}h(x,u_n,s)ds-\int_0^{v_0}h(x,u_0,s)ds)(u_n-u_0)dx|\notag\\
&&\leq\delta_\beta^2\int_{\bbr^3}|u_n-u_0||v_n-v_0|dx+o_n(1),\label{eq0085}
\end{eqnarray}
and
\begin{eqnarray}
&&|\int_{\bbr^3}(\int_0^{u_n}h(x,t,v_n)dt-\int_0^{u_0}h(x,t,v_0)dt)(v_n-v_0)dx|\notag\\
&&\leq\delta_\beta^2\int_{\bbr^3}|u_n-u_0||v_n-v_0|dx+o_n(1).\label{eq0044}
\end{eqnarray}
Combining \eqref{eq0010}, \eqref{eq0012}-\eqref{eq0013} and \eqref{eq0085}-\eqref{eq0044}, we can conclude that
\begin{equation}   \label{eq0014}
o_n(1)\geq\|u_n-u_0\|_{a,\lambda}^2+\|v_n-v_0\|_{b,\lambda}^2
-\delta_\beta^2(\mu_1+2|\beta|)\|u_n-u_0\|_2^2-\delta_\beta^2(\mu_2+2|\beta|)\|v_n-v_0\|_2^2.
\end{equation}
It follows from Lemma~\ref{lem0001} and \eqref{eq0039} that
$o_n(1)\geq\frac{C_{a,b}}{2}(\|u_n-u_0\|_2^2+\|v_n-v_0\|_2^2)$,
which implies $u_n\to u_0$ and $v_n\to v_0$ strongly in $L^2(\bbr^3)$ as $n\to\infty$.  This together the condition $(D_4)$, implies
\begin{equation*}
o_n(1)\geq\|u_n-u_0\|_{a,\lambda}^2+\|v_n-v_0\|_{b,\lambda}^2\geq\|u_n-u_0\|_a^2+\|v_n-v_0\|_b^2+o_n(1)
\end{equation*}
for $\lambda\geq\Lambda_1$ and $\beta<0$.
Thus, $(u_n,v_n)\to(u_0,v_0)$ strongly in $E$ as $n\to\infty$ for $\lambda\geq\Lambda_1$ and $\beta<0$, which completes the proof.
\end{proof}

\vskip0.1in

In the final of this section, we will show that $J_{\lambda,\beta}^*(u,v)$ is actually a penalized functional of $J_{\lambda,\beta}(u,v)$ in the sense that, some special critical points of $J_{\lambda,\beta}^*(u,v)$ are also critical points of $J_{\lambda,\beta}(u,v)$.

\vskip0.1in

\begin{lemma}
\label{lem0004}Assume $\lambda\geq\Lambda_1$ and $\beta<0$.  Let $M>0$ be a constant and $(u_{\lambda,\beta},v_{\lambda,\beta})\in E$ satisfy $J_{\lambda,\beta}^*(u_{\lambda,\beta},v_{\lambda,\beta})\leq M$ and $D[J_{\lambda,\beta}^*(u_{\lambda,\beta},v_{\lambda,\beta})]=0$ in $E^*$.  Then
\begin{enumerate}
\item[$(1)$] There exists $M_1>0$ such that $\|(u_{\lambda,\beta},v_{\lambda,\beta})\|\leq M_1$.
\item[$(2)$] $\int_{\bbr^3\backslash\Omega_a^{J_a}}|\nabla u_{\lambda,\beta}|^2+(\lambda a(x)+a_0(x))u_{\lambda,\beta}^2dx\to0$ and $\int_{\bbr^3\backslash\Omega_b^{J_b}}|\nabla v_{\lambda,\beta}|^2+(\lambda b(x)+b_0(x))v_{\lambda,\beta}^2dx\to0$ as $\lambda\to+\infty$.
\item[$(3)$] There exists $\Lambda_1^*(\beta,M)\geq\Lambda_1$ such that $|u_{\lambda,\beta}|\leq\delta_\beta$ on $\bbr^3\backslash\Omega_a^{J_a}$ and $|v_{\lambda,\beta}|\leq\delta_\beta$ on $\bbr^3\backslash\Omega_b^{J_b}$ for $\lambda\geq\Lambda_1^*(\beta,M)$.
\end{enumerate}
\end{lemma}
\begin{proof}$(1)$\quad Since $\lambda\geq\Lambda_1$ and $\beta<0$, by a similar argument as \eqref{eq0109}, we can conclude that
\begin{equation}     \label{eq0015}
M\geq\frac{1}{4}(1-\mu_1\delta_\beta^2C_{a,b}^{-1})\|u_{\lambda,\beta}\|_{a,\lambda}^2
+\frac{1}{4}(1-\mu_2\delta_\beta^2C_{a,b}^{-1})\|v_{\lambda,\beta}\|_{b,\lambda}^2.
\end{equation}
It follows from Lemma~\ref{lem0001} and \eqref{eq0039} that
$8MC_{a,b}^{-1}\geq\|u_{\lambda,\beta}\|_2^2+\|v_{\lambda,\beta}\|_2^2$.
Now, applying the condition $(D_4)$, we can see that
\begin{equation}    \label{eq0020}
8M+8MC_{a,b}^{-1}(C_{a,0}+d_a+C_{b,0}+d_b)\geq\|(u_{\lambda,\beta},v_{\lambda,\beta})\|^2.
\end{equation}
We complete this proof by taking $M_1=8M+8MC_{a,b}^{-1}(C_{a,0}+d_a+C_{b,0}+d_b)$.

\vskip0.1in

$(2)$\quad Since $\bbr^3\backslash\Omega_a^{J_a}=(\bbr^3\backslash\Omega_a')\cup(\underset{i_a=k_a+1}{\overset{n_a}{\cup }}\Omega_{a,i_a}')$ and $\bbr^3\backslash\Omega_b^{J_b}=(\bbr^3\backslash\Omega_b')\cup(\underset{j_b=k_b+1}{\overset{n_b}{\cup }}\Omega_{b,j_b}')$, for the sake of clarity, we divide this proof into the following two steps.

\vskip0.1in

{\bf Step 1. }\quad We prove that
\begin{equation}          \label{eq0118}
\lim_{\lambda\to+\infty}\int_{\bbr^3\backslash\Omega_a'}|\nabla u_{\lambda,\beta}|^2+(\lambda a(x)+a_0(x))u_{\lambda,\beta}^2 dx=0
\end{equation}
and
\begin{equation}          \label{eq0140}
\lim_{\lambda\to+\infty}\int_{\bbr^3\backslash\Omega_b'}|\nabla v_{\lambda,\beta}|^2+(\lambda b(x)+b_0(x))v_{\lambda,\beta}^2 dx=0.
\end{equation}

Indeed, let $\{\Omega_{a,i_a}''\}$ be a sequence of bounded domains with smooth boundaries in $\bbr^3$ and satisfy
\begin{enumerate}
\item[$(a)$] $\Omega_{a,i_a}\subset\Omega_{a,i_a}''\subset\Omega_{a,i_a}'$ for all $i_a=1,\cdots,n_a$.
\item[$(b)$] $dist(\Omega_{a,i_a}'', \bbr^3\backslash\Omega_{a,i_a}')>0$ and $dist(\bbr^3\backslash\Omega_{a,i_a}'',\Omega_{a,i_a})>0$ for all $i_a=1,\cdots,n_a$.
\end{enumerate}
Denote $\Omega_a''=\underset{i_a=1}{\overset{n_a}{\cup }}\Omega_{a,i_a}''$.   Then by a similar argument  as \eqref{eq0017}, we can show that
\begin{equation}        \label{eq0341}
\int_{\Omega_a''}|\nabla u_{\lambda,\beta}|^2+(\lambda a(x)+a_0(x))u_{\lambda,\beta}^2dx\geq\frac{\nu_a}{2}\int_{\Omega_a''}u_{\lambda,\beta}^2dx
\end{equation}
for $\lambda$ large enough.  Without loss of generality, we assume \eqref{eq0341} holds for $\lambda\geq\Lambda_1$.  Since Lemma~\ref{lem0001} and \eqref{eq0015} hold, we can obtain $8((C_{a,0}+d_a)C_{a,b}^{-1}+1)M\geq\lambda\int_{\bbr^3\backslash\Omega_a''}a(x)u_{\lambda,\beta}^2dx$ for $\lambda\geq\Lambda_1$ by the condition $(D_4)$ and \eqref{eq0341}.  Since the condition $(D_3')$ is contained in the condition $(D_3)$, by a similar argument as \eqref{eq0019}, we can see that
\begin{equation}        \label{eq0119}
\int_{\bbr^3\backslash\Omega_a''}u_{\lambda,\beta}^2dx\to0\text{ as }\lambda\to+\infty.
\end{equation}
Let $\Psi\in C^\infty(\bbr^3)$ be given by \eqref{eq0086}.
Then $u_{\lambda,\beta}\Psi\in E_a$.   Note $D[J_{\lambda,\beta}^*(u_{\lambda,\beta},v_{\lambda,\beta})]=0$ in $E^*$, we get that
\begin{eqnarray*}
&&\int_{\bbr^3}(|\nabla u_{\lambda,\beta}|^2+(\lambda a(x)+a_0(x))u_{\lambda,\beta}^2)\Psi dx+\int_{\bbr^3}(\nabla u_{\lambda,\beta}\nabla \Psi) u_{\lambda,\beta} dx\\
&=&\mu_1\int_{\bbr^3}f_a(x,u_{\lambda,\beta})u_{\lambda,\beta}\Psi dx+2\beta\int_{\bbr^3}(\int_0^{v_{\lambda,\beta}}h(x,u_{\lambda,\beta},s)ds)u_{\lambda,\beta}\Psi dx.
\end{eqnarray*}
Since $\beta<0$, by \eqref{eq0022} and the construction of $f_a(x,t)$ and $h(x,t,s)$, we have
\begin{eqnarray}
0&\leq&\int_{\bbr^3\backslash\Omega_a'}|\nabla u_{\lambda,\beta}|^2+(\lambda a(x)+a_0(x))u_{\lambda,\beta}^2 dx\notag\\
&\leq&\mu_1\int_{\Omega_a^{J_a}\cap(\bbr^3\backslash\Omega_a'')}u_{\lambda,\beta}^4dx
+\mu_1\delta_\beta^2\int_{(\bbr^3\backslash\Omega_a^{J_a})\cap(\bbr^3\backslash\Omega_a'')}u_{\lambda,\beta}^2dx
+\int_{\Omega_a'\backslash\Omega_a''}|\nabla u_{\lambda,\beta}||\nabla\Psi||u_{\lambda,\beta}|dx\notag\\
&\leq&\mu_1\delta_\beta^2\int_{\bbr^3\backslash\Omega_a''}u_{\lambda,\beta}^2dx
+(\mu_1S^{-\frac32}\|u_{\lambda,\beta}\|_a^3+\max_{\bbr^3}|\nabla\Psi|\|u_{\lambda,\beta}\|_a)\bigg(\int_{\bbr^3\backslash\Omega_a''}u_{\lambda,\beta}^2dx\bigg)^{\frac12}.\label{eq0046}
\end{eqnarray}
Thanks to \eqref{eq0020} and \eqref{eq0119}, we  know from \eqref{eq0046} that \eqref{eq0118} holds.
By a similar argument, we can also conclude that \eqref{eq0140} is true.

\vskip0.1in

{\bf Step 2. }\quad We prove that
\begin{equation}         \label{eq0049}
\lim_{\lambda\to+\infty}\int_{\Omega_{a,i_a}'}|\nabla u_{\lambda,\beta}|^2+(\lambda a(x)+a_0(x))u_{\lambda,\beta}^2 dx=0\quad\text{for all }i_a\in\{1,\cdots,n_a\}\backslash J_a
\end{equation}
and
\begin{equation}        \label{eq0050}
\lim_{\lambda\to+\infty}\int_{\Omega_{b,j_b}'}|\nabla v_{\lambda,\beta}|^2+(\lambda b(x)+b_0(x))v_{\lambda,\beta}^2 dx=0\quad\text{for all }j_b\in\{1,\cdots,n_b\}\backslash J_b.
\end{equation}

In fact, let $\{\Omega_{a,i_a}'''\}$ be a sequence of bounded domains with smooth boundaries in $\bbr^3$ and satisfy
\begin{enumerate}
\item[$(i)$] $\Omega_{a,i_a}'\subset\Omega_{a,i_a}'''$ and dist$(\Omega_{a,i_a}', \bbr^3\backslash\Omega_{a,i_a}''')>0$ for all $i_a\in\{1,\cdots,n_a\}$.
\item[$(ii)$]$\overline{\Omega_{a,i_a}'''}\cap \overline{\Omega_{a,j_a}'''}=\emptyset$ for all $i_a\not=j_a$.
\item[$(iii)$] $(\overline{\underset{i_a=1}{\overset{n_a}{\cup }}\Omega_{a,i_a}'''})\cap \overline{\Omega_b'}=\emptyset$.
\end{enumerate}
For every $i_a\in\{1,\cdots,n_a\}\backslash J_a$, we choose $\Psi_{i_a}\in C^\infty(\bbr^3, [0, 1])$ satisfying
\begin{equation*}
\Psi_{i_a}=\left\{\aligned&1,\quad&x\in\Omega_{a,i_a}',\\&0,\quad&x\in\bbr^3\backslash\Omega_{a,i_a}'''.\endaligned\right.
\end{equation*}
Then by a similar argument as \eqref{eq0046}, the choice of $\Omega_{a,i_a}'''$ and the construction of $f_{a}(x,t)$ and $h(x,t,s)$, we can obtain that
\begin{eqnarray*}
&&\int_{\Omega_{a,i_a}'}|\nabla u_{\lambda,\beta}|^2+(\lambda a(x)+a_0(x))u_{\lambda,\beta}^2 dx\\
& & \leq\mu_1\delta_\beta^2\int_{\Omega_{a,i_a}'''}u_{\lambda,\beta}^2dx
+\int_{\Omega_{a,i_a}'''\backslash\Omega_{a,i_a}'}|\nabla u_{\lambda,\beta}||\nabla\Psi_{i_a}||u_{\lambda,\beta}|dx.
\end{eqnarray*}
Thanks to the choice of $\Omega_{a,i_a}'''$ and \eqref{eq0020}, for $i_a\in\{1,\cdots,n_a\}\backslash J_a$, we have
\begin{eqnarray*}
&&\int_{\Omega_{a,i_a}'}|\nabla u_{\lambda,\beta}|^2+(\lambda a(x)+a_0(x))u_{\lambda,\beta}^2 dx\notag\\
&& \leq \mu_1\delta_\beta^2\int_{\Omega_{a,i_a}'}u_{\lambda,\beta}^2dx
+C\bigg(\int_{\bbr^3\backslash\Omega_a''}u_{\lambda,\beta}^2dx\bigg)^{\frac12}.
\end{eqnarray*}
It follows from \eqref{eq0087}, \eqref{eq0039} and \eqref{eq0119} that \eqref{eq0049} holds.  A similar argument implies that \eqref{eq0050} holds too.
Now, the conclusion follows immediately from \eqref{eq0118}-\eqref{eq0140} and \eqref{eq0049}-\eqref{eq0050}.


\vskip0.1in

$(3)$\quad By \eqref{eq0087} and \eqref{eq0049}, we have
\begin{equation*}
\int_{\Omega_{a,i_a}'}u_{\lambda,\beta}^2dx\to0\quad\text{as }\lambda\to+\infty\quad\text{for } i_a\in\{1,\cdots,n_a\}\backslash J_a,
\end{equation*}
which together with \eqref{eq0119}, implies
\begin{equation}\label{eq0023}
\int_{\bbr^3\backslash\underset{i_a\in J_a}\cup\Omega_{a,i_a}''}u_{\lambda,\beta}^2dx\to0\text{ as }\lambda\to+\infty.
\end{equation}
Let $r=\frac13dist(\Omega_a',\Omega_a'')$.  Then for every $x\in \bbr^3\backslash\Omega_a^{J_a}$, $B_{2r}(x)\subset \bbr^3\backslash\underset{i_a\in J_a}\cup\Omega_{a,i_a}''$.  We define $\phi_L=\min\{|u_{\lambda,\beta}|^{\alpha-1}, L\}u_{\lambda,\beta}\overline{\rho}^2$, where $\overline{\rho}\in C_0^\infty(B_{2r}(x), [0, 1])$ with $\overline{\rho}=1$ on $B_{\frac{5r}{3}}(x)$ and $|\nabla\overline{\rho}|<\frac{C}{2r-\frac{5}{3}r}$, $\alpha>0$ and $L>0$.  Then $\phi_L\in E_a$.  Since $D[J_{\lambda,\beta}^*(u_{\lambda,\beta},v_{\lambda,\beta})]=0$
in $E^*$ and the conditions $(D_1)$ and $(D_4)$ hold, by multiplying $(\mathcal{P}_{\lambda,\beta}^*)$ with $(\phi_L, 0)$ and letting $L\to+\infty$, we have
\begin{eqnarray*}
\frac12\int_{\bbr^3}\overline{\rho}^2|u_{\lambda,\beta}|^{\alpha-1}|\nabla u_{\lambda,\beta}|^2dx\leq C_\beta\bigg(\int_{\bbr^3}\overline{\rho}^2u_{\lambda,\beta}^{\alpha+3}dx+\int_{\bbr^3}\overline{\rho}^2u_{\lambda,\beta}^{\alpha+1}dx\bigg)
+4\int_{\bbr^3}|\nabla\overline{\rho}|^2u_{\lambda,\beta}^{\alpha+1}dx,
\end{eqnarray*}
where $C_\beta=\mu_1+\mu_1\delta_\beta^2+C_{a,0}+d_a$.
By the Sobolev embedding theorem, we can see that
\begin{eqnarray}
\bigg(\int_{B_{\frac{5r}{3}}(x)}u_{\lambda,\beta}^{3(\alpha+1)}dx\bigg)^{\frac13}\leq C_\beta(\alpha+1)^2\bigg(\int_{B_{2r}(x)}u_{\lambda,\beta}^{\alpha+3}dx
+(1+\frac{24}{r^2})\int_{B_{2r}(x)}u_{\lambda,\beta}^{\alpha+1}dx\bigg).\label{eq0601}
\end{eqnarray}
Let $\alpha_n=3\alpha_{n-1}$ with $\alpha_0=2$ and $r_n=(1+(\frac{2}{3})^{n-1})r$, $n\in\bbn$.  Then \eqref{eq0601} can be re-written  as
\begin{equation}
\bigg(\int_{B_{r_1}(x)}u_{\lambda,\beta}^{3(\alpha_0+1)}dx\bigg)^{\frac13}\leq C_\beta(\alpha_0+1)^2\bigg(\int_{B_{r_0}(x)}u_{\lambda,\beta}^{\alpha_0+3}dx
+(1+\frac{4}{|r_0-r_1|^2})\int_{B_{r_0}(x)}u_{\lambda,\beta}^{\alpha_0+1}dx\bigg).\label{eq0711}
\end{equation}
We replace $\alpha_0$, $r_0$ and $r_1$ in \eqref{eq0711} by $\alpha_n$, $r_n$ and $r_{n+1}$.  Then we can obtain
\begin{eqnarray}
&&\bigg(\int_{B_{r_{n+1}}(x)}u_{\lambda,\beta}^{3(\alpha_n+1)}dx\bigg)^{\frac{1}{3(\alpha_n+1)}}\notag\\
& & \leq[C_\beta(\alpha_n+1)^2]^{\frac{1}{(\alpha_n+1)}}\bigg(\int_{B_{r_n}(x)}u_{\lambda,\beta}^{\alpha_n+3}dx\bigg)^{\frac{1}{(\alpha_n+1)}}\notag\\
&&\quad +[C_\beta(\alpha_n+1)^2]^{\frac{1}{(\alpha_n+1)}}(1+\frac{4}{|r_n-r_{n+1}|^2})^{\frac{1}{(\alpha_n+1)}}
\bigg(\int_{B_{r_n}(x)}u_{\lambda,\beta}^{\alpha_n+1}dx\bigg)^{\frac{1}{(\alpha_n+1)}}.\label{eq0610}
\end{eqnarray}
Clearly, one of the following two cases must happen:
\begin{enumerate}
\item[$(1)$] $\int_{B_{r_n}(x)}u_{\lambda,\beta}^{\alpha_{n}+1}dx\leq\int_{B_{r_n}(x)}u_{\lambda,\beta}^{\alpha_{n}+3}dx$ up to a subsequence.
\item[$(2)$] $\int_{B_{r_n}(x)}u_{\lambda,\beta}^{\alpha_{n}+3}dx\leq\int_{B_{r_n}(x)}u_{\lambda,\beta}^{\alpha_{n}+1}dx$ up to a subsequence.
\end{enumerate}
If case~$(1)$ happen, then by \eqref{eq0610}, we can see that
\begin{eqnarray*}
&&\bigg(\int_{B_{r_{n+1}}(x)}u_{\lambda,\beta}^{3(\alpha_{n}+1)}\bigg)^{\frac{1}{3(\alpha_{n}+1)}}\\
&  &  \leq \bigg(4C_\beta(\alpha_{n}+1)^2(1+\frac{1}{|r_n-r_{n+1}|^2})\bigg)^{\frac{1}{(\alpha_{n}+1)}}
\bigg(\int_{B_{r_{n}}(x)}u_{\lambda,\beta}^{\alpha_{n}+3}\bigg)^{\frac{1}{\alpha_{n}+1}}.
\end{eqnarray*}
By iterating \eqref{eq0610} and using the choice of $r_n$ and $\alpha_n$, we have
\begin{eqnarray}
&&\lim_{n\to+\infty}\bigg(\int_{B_{r_{n+1}}(x)}u_{\lambda,\beta}^{3(\alpha_{n}+1)}\bigg)^{\frac{1}{3(\alpha_{n}+1)}}\notag\\
& & \leq \bigg(\prod_{n=1}^{\infty}\bigg(4C_\beta(\alpha_{n}+1)^2(1+\frac{1}{|r_n-r_{n+1}|^2})\bigg)^{\frac{1}{(\alpha_{n}+1)}}
\bigg(\int_{\bbr^3\backslash\underset{i_a\in J_a}\cup\Omega_{a,i_a}''}u_{\lambda,\beta}^5dx\bigg)^{\frac15}\bigg)
^{\underset{n=1}{\overset{+\infty}{\prod }}\frac{\alpha_n+3}{\alpha_n+1}}\notag\\
& & \leq C_\beta'\bigg(\int_{\bbr^3\backslash\underset{i_a\in J_a}\cup\Omega_{a,i_a}''}u_{\lambda,\beta}^5dx\bigg)^{\frac {C}5},\label{eq0602}
\end{eqnarray}
where $C_\beta'$ is a constant independent of $\lambda$ and $x$.
If case~$(2)$ happen, then by iterating \eqref{eq0610} and using the choice of $r_n$ and $\alpha_n$ once more, we have
\begin{eqnarray}
&&\lim_{n\to+\infty}\bigg(\int_{B_{r_{n+1}}(x)}u_{\lambda,\beta}^{3(\alpha_{n}+1)}\bigg)^{\frac{1}{3(\alpha_{n}+1)}}\notag\\
& & \leq\prod_{n=1}^{\infty}\bigg(4C_\beta(\alpha_{n}+1)^2(1+\frac{1}{|r_n-r_{n+1}|^2})|B_{r_n}(x)|^{\frac{2}{\alpha_{n}+3}}\bigg)^{\frac{1}{(\alpha_{n}+1)}}
\bigg(\int_{\bbr^3\backslash\underset{i_a\in J_a}\cup\Omega_{a,i_a}''}u_{\lambda,\beta}^5dx\bigg)^{\frac15}\notag\\
& &\leq C_\beta'\bigg(\int_{\bbr^3\backslash\underset{i_a\in J_a}\cup\Omega_{a,i_a}''}u_{\lambda,\beta}^5dx\bigg)^{\frac15}.\label{eq0603}
\end{eqnarray}
By the H\"older and the Sobolev inequalities and \eqref{eq0020} and \eqref{eq0023}, we can conclude that
\begin{equation*}
\bigg(\int_{\bbr^3\backslash\underset{i_a\in J_a}\cup\Omega_{a,i_a}''}u_{\lambda,\beta}^5dx\bigg)^{\frac15}\to0\quad\text{as }\lambda\to+\infty.
\end{equation*}
It follows from \eqref{eq0602} and \eqref{eq0603} that $\|u_{\lambda,\beta}\|_{L^\infty(B_r(x))}\to0$ as $\lambda\to+\infty$, which then implies $\|u_{\lambda,\beta}\|_{L^\infty(\bbr^3\backslash\Omega_{a}^{J_a})}\to0$ as $\lambda\to+\infty$.  By similar arguments, we also have $\|v_{\lambda,\beta}\|_{L^\infty(\bbr^3\backslash\Omega_{b}^{J_b})}\to0$ as $\lambda\to+\infty$.  Now, we can choose $\Lambda_1^*(\beta,M)\geq\Lambda_1$ such that $|u_{\lambda,\beta}|\leq\delta_\beta$ a.e. on $\bbr^3\backslash\Omega_a^{J_a}$ and $|v_{\lambda,\beta}|\leq\delta_\beta$ a.e. on $\bbr^3\backslash\Omega_b^{J_b}$ for $\lambda\geq\Lambda_1^*(\beta)$.  Note that by a similar argument as used in Theorem~\ref{thm0002}, we can see that $(u_{\lambda,\beta}, v_{\lambda,\beta})\in C(\bbr^3)\times C(\bbr^3)$.  Hence, we must have $|u_{\lambda,\beta}|\leq\delta_\beta$ on $\bbr^3\backslash\Omega_a^{J_a}$ and $|v_{\lambda,\beta}|\leq\delta_\beta$ on $\bbr^3\backslash\Omega_b^{J_b}$ for $\lambda\geq\Lambda_1^*(\beta,M)$.
\end{proof}

\subsection{Construction of critical points}
In this section, we will construct critical values of $J_{\lambda,\beta}^*(u,v)$ by a minimax argument.  The idea of such a  construction  traces back to  S\'er\'e \cite{S92} and also was applied  in \cite{BT13,DT03,GT121,ST09}.

\vskip0.1in

We first recall some well-known results, which are useful in this construction.  For all $i_a=1,\cdots,n_a$ and $j_b=1,\cdots,n_b$, we  define $\mathcal{E}_{\Omega_{a,i_a}'}$ on $H^1(\Omega_{a,i_a}')$ and $\mathcal{E}_{\Omega_{b,j_b}'}$ on $H^1(\Omega_{b,j_b}')$ as follows:
\begin{eqnarray*}
\mathcal{E}_{\Omega_{a,i_a}'}(u)&=&\frac12\int_{\Omega_{a,i_a}'}|\nabla u|^2+(\lambda a(x)+a_0(x))u^2dx-\frac{\mu_1}{4}\int_{\Omega_{a,i_a}'}u^4dx,\\
\mathcal{E}_{\Omega_{b,j_b}'}(v)&=&\frac12\int_{\Omega_{b,j_b}'}|\nabla v|^2+(\lambda b(x)+b_0(x))v^2dx-\frac{\mu_2}{4}\int_{\Omega_{b,j_b}'}v^4dx.
\end{eqnarray*}
By \eqref{eq0087} and \eqref{eq0210}, $\mathcal{E}_{\Omega_{a,i_a}'}(u)$ and $\mathcal{E}_{\Omega_{b,j_b}'}(v)$ have the least energy nonzero critical point for all $i_a=1,\cdots,n_a$ and $j_b=1,\cdots,n_b$ if $\lambda\geq\Lambda_1$.  We  denote the ground state level of $\mathcal{E}_{\Omega_{a,i_a}'}(u)$ and $\mathcal{E}_{\Omega_{b,j_b}'}(v)$ by $m_{a,i_a,\lambda}$ and $m_{b,j_b,\lambda}$, respectively.
Since $\{\Omega_{a,i_a}'\}$ and $\{\Omega_{b,j_b}'\}$ are two sequences of bounded domains, by the conditions $(D_1)$-$(D_2)$, $(D_3')$, $(D_4)$ and $(D_5')$ and \eqref{eq0087}-\eqref{eq0210}, it is easy to show that $m_{a,i_a,\lambda}$ and $m_{b,j_b,\lambda}$ are positive for $\lambda\geq\Lambda_1$.  It follows that
\begin{equation}\label{eq0027}
m_{a,i_a,\lambda}=\inf\{\mathcal{E}_{\Omega_{a,i_a}'}(u)\mid \int_{\Omega_{a,i_a}'}u^4dx
=\frac{4m_{a,i_a,\lambda}}{\mu_1}\}\quad\text{for all }i_a=1,\cdots,n_a
\end{equation}
and
\begin{equation}\label{eq0028}
m_{b,j_b,\lambda}=\inf\{\mathcal{E}_{\Omega_{b,j_b}'}(v)\mid\int_{\Omega_{b,j_b}'}v^4dx
=\frac{4m_{b,j_b,\lambda}}{\mu_2}\}\quad\text{for all }j_b=1,\cdots,n_b.
\end{equation}
On the other hand, let $W_{a,i_a}\in H_0^1(\Omega_{a,i_a})$ and $W_{b,j_b}\in H_0^1(\Omega_{b,j_b})$ be the least energy nonzero critical points of $I_{\Omega_{a,i_a}}(u)$ and $I_{\Omega_{b,j_b}}(v)$, respectively.  Then by the conditions $(D_3')$ and $(D_5')$, it is well-known that
\begin{equation}\label{eq0029}
I_{\Omega_{a,i_a}}(W_{a,i_a})=\max_{t\geq0}I_{\Omega_{a,i_a}}(tW_{a,i_a})\quad\text{and}\quad
I_{\Omega_{b,j_b}}(W_{b,j_b})=\max_{t\geq0}I_{\Omega_{b,j_b}}(tW_{b,j_b}).
\end{equation}
Let $\gamma_{0,a}:[0, 1]^{k_a}\to E_a$ and $\gamma_{0,b}:[0, 1]^{k_b}\to E_b$ be
\begin{equation}     \label{eq0097}
\gamma_{0,a}(t_1,\cdots,t_{k_a})=\sum_{i_a=1}^{k_a}t_{i_a}RW_{a,i_a}
\end{equation}
and
\begin{equation}     \label{eq0098}
\gamma_{0,b}(s_1,\cdots,s_{k_b})=\sum_{j_b=1}^{k_b}s_{j_b}RW_{b,j_b},
\end{equation}
where $R>2$ is a large constant satisfying
\begin{eqnarray}
&I_{\Omega_{a,i_a}}(RW_{a,i_a})\leq0,\quad&R^4\int_{\Omega_{a,i_a}}W_{a,i_a}^4dx\geq2\frac{4m_{a,i_a}}{\mu_1},\label{eq0088}\\
&I_{\Omega_{b,j_b}}(RW_{b,j_b})\leq0,\quad&R^4\int_{\Omega_{b,j_b}}W_{b,j_b}^4dx\geq2\frac{4m_{b,j_b}}{\mu_2}.\label{eq0030}
\end{eqnarray}
for all $i_a=1,\cdots,k_a$ and $j_b=1,\cdots,k_b$.  By the condition $(D_3')$, we can extend $W_{a,i_a}$ and $W_{b,j_b}$ to the whole space $\bbr^3$ by letting  $W_{a,i_a}=0$ on $\bbr^3\backslash\Omega_{a,i_a}$ and $W_{b,j_b}=0$ on $\bbr^3\backslash\Omega_{b,j_b}$ such that $W_{a,i_a}\in\h$ and $W_{b,j_b}\in\h$ for all $i_a=1,\cdots,n_a$ and $j_b=1,\cdots,n_b$.  Now, we can define a minimax value of $J_{\lambda,\beta}^*(u,v)$ for $\lambda\geq\Lambda_1$ and $\beta<0$ as follows:
\begin{equation*}
m_{J_a,J_b,\lambda,\beta}=\inf_{(\gamma_a,\gamma_b)\in\Gamma}\sup_{[0, 1]^{k_a}\times[0, 1]^{k_b}}J_{\lambda,\beta}^*(\gamma_a,\gamma_b),
\end{equation*}
where
\begin{eqnarray*}
\Gamma&=&\bigg\{(\gamma_a,\gamma_b)\mid(\gamma_a,\gamma_b)\in C([0, 1]^{k_a}\times[0, 1]^{k_b}, E_a\times E_b),\\
&&(\gamma_a,\gamma_b)=(\gamma_{0,a},\gamma_{0,b})\text{ on }\partial([0, 1]^{k_a}\times[0, 1]^{k_b})\bigg\}.
\end{eqnarray*}
$m_{J_a,J_b,\lambda,\beta}$ may be a critical value of $J_{\lambda,\beta}^*(u,v)$.  In order to show it, we need the following.

\vskip0.1in

\begin{lemma}
\label{lem0005}Assume $(\gamma_a,\gamma_b)\in\Gamma$ and
\begin{eqnarray*}
(\xi_1,\cdots,\xi_{k_a},\eta_1,\cdots,\eta_{k_b})&\in&[0, R^4\int_{\Omega_{a,1}}W_{a,1}^4dx]\times\cdots\times[0, R^4\int_{\Omega_{a,k_a}}W_{a,k_a}^4dx]\times\\
&&[0, R^4\int_{\Omega_{b,1}}W_{b,1}^4dx]\times\cdots\times[0, R^4\int_{\Omega_{b,k_b}}W_{b,k_b}^4dx].
\end{eqnarray*}
Then there exist $(t_1',\cdots,t_{k_a}')\in[0, 1]^{k_a}$ and $(s_1',\cdots,s_{k_b}')\in[0, 1]^{k_b}$ such that
\begin{equation*}
\int_{\Omega_{a,i_a}'}[\gamma_a(t_1',\cdots,t_{k_a}')]^4(x)dx=\xi_{i_a}\quad\text{for all }i_a=1,\cdots,k_a
\end{equation*}
and
\begin{equation*}
\int_{\Omega_{b,j_b}'}[\gamma_b(s_1',\cdots,s_{k_b}')]^4(x)dx=\eta_{j_b}\quad\text{for all }j_b=1,\cdots,k_b.
\end{equation*}
\end{lemma}
\begin{proof}
For every $(\gamma_a,\gamma_b)\in\Gamma$, we define a map $\widetilde{\gamma}:[0, 1]^{k_a}\times[0, 1]^{k_b}\to\bbr^{k_a}\times\bbr^{k_b}$ as follows:
\begin{eqnarray*}
&&\widetilde{\gamma}(t_1,\cdots,t_{k_a},s_1,\cdots,s_{k_b})\\
&=&\bigg(\int_{\Omega_{a,1}'}[\gamma_a(t_1,\cdots,t_{k_a})]^4(x)dx,\cdots,
\int_{\Omega_{a,k_a}'}[\gamma_a(t_1,\cdots,t_{k_a})]^4(x)dx,\\
&&\int_{\Omega_{b,1}'}[\gamma_b(s_1,\cdots,s_{k_b})]^4(x)dx,\cdots,
\int_{\Omega_{b,k_b}'}[\gamma_b(s_1,\cdots,s_{k_b})]^4(x)dx\bigg).
\end{eqnarray*}
Note that for every $(\gamma_a,\gamma_b)\in\Gamma$, we have
\begin{equation*}
(\gamma_a(t_1,\cdots,t_{k_a}),\gamma_b(s_1,\cdots,s_{k_b}))=
(\gamma_{0,a}(t_1,\cdots,t_{k_a}),\gamma_{0,b}(s_1,\cdots,s_{k_b}))
\end{equation*}
if $(t_1,\cdots,t_{k_a},s_1,\cdots,s_{k_b})\in\partial([0, 1]^{k_a}\times[0, 1]^{k_b})$.  Then by the construction of $\{\Omega_{a,i_a}'\}$ and $\{\Omega_{b,j_b}'\}$, we can see that
\begin{equation*}
\int_{\Omega_{a,i_a}'}[\gamma_a(t_1,\cdots,t_{k_a})]^4(x)dx=t_{i_a}^4R^4\int_{\Omega_{a,i_a}}W_{a,i_a}^4dx
\end{equation*}
and
\begin{equation*}
\int_{\Omega_{b,j_b}'}[\gamma_b(s_1,\cdots,s_{k_b})]^4(x)dx=s_{j_b}^4R^4\int_{\Omega_{b,j_b}}W_{b,j_b}^4dx
\end{equation*}
for all $i_a=1,\cdots,k_a$, $j_b=1,\cdots,k_b$ and $(t_1,\cdots,t_{k_a},s_1,\cdots,s_{k_b})\in\partial([0, 1]^{k_a}\times[0, 1]^{k_b})$.
It follows that
\begin{equation*}
deg(\widetilde{\gamma},[0, 1]^{k_a}\times[0, 1]^{k_b}, (\xi_1,\cdots,\xi_{k_a},\eta_1,\cdots,\eta_{k_b}))=1,
\end{equation*}
which completes the proof.
\end{proof}

\vskip0.1in

With Lemma~\ref{lem0005} in hands, we can obtain the following energy estimate, which can be viewed  as a linking structure of $J_{\lambda,\beta}^*(u,v)$.

\vskip0.1in

\begin{lemma}
\label{lem0006} Assume $\lambda\geq\Lambda_1$ and $\beta<0$.  Then we have the following results.
\begin{enumerate}
\item[$(1)$] If $(t_1,\cdots,t_{k_a},s_1,\cdots,s_{k_b})\in \partial([0, 1]^{k_a}\times[0, 1]^{k_b})$, then
 \begin{eqnarray*}
 &&J_{\lambda,\beta}^*(\gamma_{0,a}(t_1,\cdots,t_{k_a}),\gamma_{0,b}(s_1,\cdots,s_{k_b}))\\
 & & \leq  \sum_{i_a=1}^{k_a}m_{a,i_a}+\sum_{j_b=1}^{k_b}m_{b,j_b}-\min\{m_{a,1},\cdots,m_{a,k_a},m_{b,1},\cdots,m_{b,k_b}\}.
 \end{eqnarray*}
\item[$(2)$] $\underset{i_a=1}{\overset{k_a}{\sum }}m_{a,i_a,\lambda}+\underset{j_b=1}{\overset{k_b}{\sum }}m_{b,j_b,\lambda}\leq m_{J_a,J_b,\lambda,\beta}\leq\underset{i_a=1}{\overset{k_a}{\sum }}m_{a,i_a}+\underset{j_b=1}{\overset{k_b}{\sum }}m_{b,j_b}$.
\end{enumerate}
\end{lemma}
\begin{proof}$(1)$\quad Since $(t_1,\cdots,t_{k_a},s_1,\cdots,s_{k_b})\in \partial([0, 1]^{k_a}\times[0, 1]^{k_b})$, there exists $i_a'\in\{1,\cdots,k_a\}$ or $j_b'\in\{1,\cdots,k_b\}$ such that $t_{i_a'}\in\{0, 1\}$ or $s_{j_b'}\in\{0, 1\}$.  Without loss of generality, we assume $t_1=1$.  It follows from \eqref{eq0029}-\eqref{eq0030} and the condition $(D_3')$ that
\begin{eqnarray*}
&&J_{\lambda,\beta}^*(\gamma_{0,a}(t_1,\cdots,t_{k_a}),\gamma_{0,b}(s_1,\cdots,s_{k_b}))\\
& & = I_{a,1}(RW_{a,1})+\sum_{i_a=2}^{k_a}I_{a,i_a}(t_{i_a}RW_{a,i_a})+\sum_{j_b=1}^{k_b}I_{b,j_b}(s_{j_b}RW_{b,j_b})\\
& & \leq \sum_{i_a=2}^{k_a}m_{a,i_a}+\sum_{j_b=1}^{k_b}m_{b,j_b}\\
& & \leq \sum_{i_a=1}^{k_a}m_{a,i_a}+\sum_{j_b=1}^{k_b}m_{b,j_b}
-\min\{m_{a,1},\cdots,m_{a,k_a},m_{b,1},\cdots,m_{b,k_b}\}.
\end{eqnarray*}

\vskip0.1in

$(2)$\quad Since $(\gamma_{0,a},\gamma_{0,b})\in\Gamma$ and $R>2$, by the condition $(D_3')$, we must have
\begin{eqnarray*}
m_{J_a,J_b,\lambda,\beta}&\leq&\sup_{[0, 1]^{k_a}\times[0, 1]^{k_b}}J_{\lambda,\beta}^*(\sum_{i_a=1}^{k_a}t_{i_a}RW_{a,i_a},\sum_{j_b=1}^{k_b}s_{j_b}RW_{b,j_b})\\
&\leq&\sum_{i_a=1}^{k_a}I_{a,i_a}(W_{a,i_a})+\sum_{j_b=1}^{k_b}I_{b,j_b}(W_{b,j_b})\\
&=&\sum_{i_a=1}^{k_a}m_{a,i_a}+\sum_{j_b=1}^{k_b}m_{b,j_b}.
\end{eqnarray*}
On the other hand, since the condition $(D_3')$ holds, by the construction of $\Omega_{a,i_a}'$ and $\Omega_{b,j_b}'$, it is easy to show that $m_{a,i_a,\lambda}\leq m_{a,i_a}$ and $m_{b,j_b,\lambda}\leq m_{b,j_b}$ for all $i_a=1,\cdots,n_a$, $j_b=1,\cdots,n_b$ and $\lambda>0$.  This together with \eqref{eq0088}-\eqref{eq0030} and Lemma~\ref{lem0005}, implies that for every $(\gamma_a,\gamma_b)\in\Gamma$, there exist $(t_1',\cdots,t_{k_a}')\in[0, 1]^{k_a}$ and $(s_1',\cdots,s_{k_b}')\in[0, 1]^{k_b}$ such that
\begin{equation}        \label{eq0031}
\int_{\Omega_{a,i_a}'}[\gamma_a(t_1',\cdots,t_{k_a}')]^4(x)dx=\frac{4m_{a,i_a,\lambda}}{\mu_1}\quad\text{for all }i_a=1,\cdots,k_a
\end{equation}
and
\begin{equation}        \label{eq0032}
\int_{\Omega_{b,j_b}'}[\gamma_b(s_1',\cdots,s_{k_b}')]^4(x)dx=\frac{4m_{b,j_b,\lambda}}{\mu_2}\quad\text{for all }j_b=1,\cdots,k_b.
\end{equation}
Denote $\gamma_a(t_1',\cdots,t_{k_a}')$ and $\gamma_b(s_1',\cdots,s_{k_b}')$ by $u_*$ and $v_*$.  Then by \eqref{eq0022}-\eqref{eq0210}, we have
\begin{equation*}
\int_{\bbr^3\backslash\Omega_a^{J_a}}u_*^2dx\leq C_{a,b}^{-1}\int_{\bbr^3\backslash\Omega_a^{J_a}}|\nabla u_*|^2+(\lambda a(x)+a_0(x))u_*^2dx
\end{equation*}
and
\begin{equation*}
\int_{\bbr^3\backslash\Omega_b^{J_b}}v_*^2dx\leq C_{a, b}^{-1}\int_{\bbr^3\backslash\Omega_b^{J_b}}|\nabla v_*|^2+(\lambda b(x)+b_0(x))v_*^2dx
\end{equation*}
for $\lambda\geq\Lambda_1$.
Note that $\{\Omega_{a,i_a}'\}$ and $\{\Omega_{b,j_b}'\}$ are two sequences of bounded domains with smooth boundaries, so the restriction of $u_*$ on $\Omega_{a,i_a}'$ lies in $H^1(\Omega_{a,i_a}')$ for every $i_a=1,\cdots,n_a, $ while the restriction of $v_*$ on $\Omega_{b,j_b}'$ lies in $H^1(\Omega_{b,j_b}')$ for every $j_b=1,\cdots,n_b$.  Now, by $\beta<0$, \eqref{eq0039} and the construction of $f_a(x,t)$, $f_b(x,t)$ and $h(x,t,s)$, we have
\begin{eqnarray}
J_{\lambda,\beta}^*(u_*,v_*)&\geq&\frac12\int_{\bbr^3}|\nabla u_*|^2+(\lambda a(x)+a_0(x))u_*^2dx-\mu_1\int_{\bbr^3}F_a(x,u_*)dx\notag\\
&&+\frac12\int_{\bbr^3}|\nabla v_*|^2+(\lambda b(x)+b_0(x))v_*^2dx-\mu_2\int_{\bbr^3}F_b(x,v_*)dx\notag\\
&\geq&\frac{1}{2}(\int_{\bbr^3\backslash\Omega_a^{J_a}}|\nabla u_*|^2+(\lambda a(x)+a_0(x))u_*^2dx-\delta_\beta^2\int_{\bbr^3\backslash\Omega_a^{J_a}}u_*^2dx)+\sum_{i_a=1}^{k_a}\mathcal{E}_{\Omega_{a,i_a}',\lambda}(u_*)\notag\\
&&+\frac{1}{2}(\int_{\bbr^3\backslash\Omega_b^{J_b}}|\nabla v_*|^2+(\lambda b(x)+b_0(x))v_*^2dx-\delta_\beta^2\int_{\bbr^3\backslash\Omega_b^{J_b}}v_*^2dx)+\sum_{j_b=1}^{k_b}\mathcal{E}_{\Omega_{b,j_b}',\lambda}(v_*)\notag\\
&\geq&\sum_{i_a=1}^{k_a}\mathcal{E}_{\Omega_{a,i_a}',\lambda}(u_*)+\sum_{j_b=1}^{k_b}\mathcal{E}_{\Omega_{b,j_b}',\lambda}(v_*).\label{eq0093}
\end{eqnarray}
Thanks to \eqref{eq0027}-\eqref{eq0028} and \eqref{eq0031}-\eqref{eq0032}, \eqref{eq0093} implies $J_{\lambda,\beta}^*(u_*,v_*)\geq\underset{i_a=1}{\overset{k_a}{\sum }}m_{a,i_a,\lambda}+\underset{j_b=1}{\overset{k_b}{\sum }}m_{b,j_b,\lambda}$ for $\lambda\geq\Lambda_1$ and $\beta<0$.  Since $(\gamma_a,\gamma_b)\in\Gamma$ is arbitrary, we must have $\underset{i_a=1}{\overset{k_a}{\sum }}m_{a,i_a,\lambda}+\underset{j_b=1}{\overset{k_b}{\sum }}m_{b,j_b,\lambda}\leq m_{J_a,J_b,\lambda,\beta}$ for $\lambda\geq\Lambda_1$ and $\beta<0$, which completes the proof.
\end{proof}

\vskip0.1in

Let $m_{a,b}:=\underset{i_a=1}{\overset{n_a}{\sum }}m_{a,i_a}+\underset{j_b=1}{\overset{n_b}{\sum }}m_{b,j_b}$.  Then we can obtain the following proposition.

\vskip0.1in

\begin{proposition}
\label{prop0001}Suppose $\beta<0$.  Then there exists $\Lambda_2^*(\beta)\geq\Lambda_1^*(\beta, m_{a,b})$ such that $m_{J_a,J_b,\lambda,\beta}$ is a critical value of $J_{\lambda,\beta}^*(u,v)$ for $\lambda\geq\Lambda_2^*(\beta)$, that is, for all $\lambda\geq\Lambda_2^*(\beta)$, there exists $(u_{\lambda,\beta},v_{\lambda,\beta})\in E$ satisfying $D[J_{\lambda,\beta}^*(u_{\lambda,\beta},v_{\lambda,\beta})]=0$ in $E^*$ and $J_{\lambda,\beta}^*(u_{\lambda,\beta},v_{\lambda,\beta})=m_{J_a,J_b,\lambda,\beta}$, where $\Lambda_1^*(\beta, m_{a,b})$ is given by Lemma~\ref{lem0004}.  Furthermore, for every $\{\lambda_n\}\subset[\Lambda_2^*(\beta), +\infty)$ satisfying $\lambda_n\to+\infty$ as $n\to\infty$, there exists $(u_{0,\beta}^{J_a},v_{0,\beta}^{J_b})\in E$ such that
\begin{enumerate}
\item[$(1)$] $(u_{0,\beta}^{J_a},v_{0,\beta}^{J_b})\in H_0^1(\Omega_{a,0}^{J_a})\times H_0^1(\Omega_{b,0}^{J_b})$ with $u_{0,\beta}^{J_a}=0$ on $\bbr^3\backslash\Omega_{a,0}^{J_a}$ and $v_{0,\beta}^{J_b}=0$ on $\bbr^3\backslash\Omega_{b,0}^{J_b}$.
\item[$(2)$] $(u_{\lambda_n,\beta},v_{\lambda_n,\beta})\to(u_{0,\beta}^{J_a},v_{0,\beta}^{J_b})$ strongly in $\h\times\h$ as $n\to\infty$ up to a subsequence.
\item[$(3)$] The restriction of $u_{0,\beta}^{J_a}$ on $\Omega_{a,i_a}$ lies in $H_0^1(\Omega_{a,i_a})$ and is a critical point of $I_{\Omega_{a,i_a}}(u)$ for every $i_a\in J_a$, while the restriction of $v_{0,\beta}^{J_b}$ on $\Omega_{b,j_b}$ lies in $H_0^1(\Omega_{b,j_b})$ and is a critical point of $I_{\Omega_{b,j_b}}(v)$ for every $j_b\in J_b$.
\end{enumerate}
\end{proposition}
\begin{proof}
Since the conditions $(D_1)$-$(D_2)$, $(D_3')$, $(D_4)$ and $(D_5')$ hold,
by a similar argument as \cite[Lemma~3.1]{DT03}, we can see that $\lim_{\lambda\to+\infty}m_{a,i_a,\lambda}=m_{a,i_a}$ and $\lim_{\lambda\to+\infty}m_{b,j_b,\lambda}=m_{b,j_b}$ for all $i_a=1,\cdots,n_a$ and $j_b=1,\cdots,n_b$.  Note that $\beta<0$, so by Lemma~\ref{lem0006}, there exists $\Lambda_2^*(\beta)\geq\Lambda_1^*(\beta, m_{a,b})$ such that $m_{J_a,J_b,\lambda,\beta}>J_{\lambda,\beta}^*(\gamma_{0,a}(t_1,\cdots,t_{k_a}),\gamma_{0,b}(s_1,\cdots,s_{k_b}))$ for $\lambda\geq\Lambda_2^*(\beta)$.  Thanks to the construction of $m_{J_a,J_b,\lambda,\beta}$ and Lemma~\ref{lem0003}, we can use the linking theorem (cf. \cite{AR73}) to show that $m_{J_a,J_b,\lambda,\beta}$ is a critical value of $J_{\lambda,\beta}^*(u,v)$ for $\lambda\geq\Lambda_2^*(\beta)$, that is, there exists $(u_{\lambda,\beta},v_{\lambda,\beta})\in E$ satisfying $D[J_{\lambda,\beta}^*(u_{\lambda,\beta},v_{\lambda,\beta})]=0$ in $E^*$ and $J_{\lambda,\beta}^*(u_{\lambda,\beta},v_{\lambda,\beta})=m_{J_a,J_b,\lambda,\beta}$ for all $\lambda\geq\Lambda_2^*(\beta)$.  In what follows, we will show that $(1)$-$(3)$ hold.  Suppose $\{\lambda_n\}\subset[\Lambda_2^*(\beta), +\infty)$ satisfying $\lambda_n\to+\infty$ as $n\to\infty$.  Then by Lemmas~\ref{lem0004} and \ref{lem0006}, $\{(u_{\lambda_n,\beta}, v_{\lambda_n,\beta})\}$ is bounded in $E$ with
\begin{equation}\label{eq0159}
\int_{\bbr^3\backslash\Omega_a^{J_a}}|\nabla u_{\lambda_n,\beta}|^2+(\lambda_n a(x)+a_0(x))u_{\lambda_n,\beta}^2dx\to0\quad\text{as }n\to\infty
\end{equation}
and
\begin{equation}\label{eq0160}
\int_{\bbr^3\backslash\Omega_b^{J_b}}|\nabla v_{\lambda_n,\beta}|^2+(\lambda_n b(x)+b_0(x))v_{\lambda_n,\beta}^2dx\to0\quad\text{as } n\to\infty.
\end{equation}
Without loss of generality, we assume $(u_{\lambda_n,\beta}, v_{\lambda_n,\beta})\rightharpoonup(u_{0,\beta}^{J_a},v_{0,\beta}^{J_b})$ weakly in $E$ as $n\to\infty$ for some $(u_{0,\beta}^{J_a},v_{0,\beta}^{J_b})\in E$.  For the sake of clarity, we divide the following proof into two steps.

\vskip0.1in

{\bf Step 1. }\quad We prove that $(u_{0,\beta}^{J_a},v_{0,\beta}^{J_b})\in H_0^1(\Omega_{a,0}^{J_a})\times H_0^1(\Omega_{b,0}^{J_b})$ with $u_{0,\beta}^{J_a}=0$ on $\bbr^3\backslash\Omega_{a,0}^{J_a}$ and $v_{0,\beta}^{J_b}=0$ on $\bbr^3\backslash\Omega_{b,0}^{J_b}$.

Indeed, since $\beta<0$, by Lemmas~\ref{lem0002} and \ref{lem0006} and a similar argument as used in Step 1 of the proof for Theorem~\ref{thm0002}, we can conclude that $u_{0,\beta}^{J_a}=0$ on $\bbr^3\backslash\Omega_{a}$ and $v_{0,\beta}^{J_b}=0$ on $\bbr^3\backslash\Omega_{b}$.  On the other hand, combining \eqref{eq0087} and \eqref{eq0159}, we can see that $\int_{\Omega_{a,i_a}'}u_{\lambda_n,\beta}^2dx\to0$ for $i_a\in\{1,\cdots,n_a\}\backslash J_a$ as $n\to\infty$, which together with the Fatou lemma, implies $u_{0,\beta}^{J_a}=0$ on $\Omega_{a,i_a}'$ for $i_a\in\{1,\cdots,n_a\}\backslash J_a$.  Since \eqref{eq0160} holds, by a similar argument, we also have $v_{0,\beta}^{J_b}=0$ on $\Omega_{b,j_b}'$ for $j_b\in\{1,\cdots,n_b\}\backslash J_b$.  Note that $\{\Omega_{a,i_a}\}$ and $\{\Omega_{b,j_b}\}$ are two sequences of disjoint bounded domains with smooth boundaries.  So $(u_{0,\beta}^{J_a},v_{0,\beta}^{J_b})\in H_0^1(\Omega_{a,0}^{J_a})\times H_0^1(\Omega_{b,0}^{J_b})$ with $u_{0,\beta}^{J_a}=0$ on $\bbr^3\backslash\Omega_{a,0}^{J_a}$ and $v_{0,\beta}^{J_b}=0$ on $\bbr^3\backslash\Omega_{b,0}^{J_b}$.

\vskip0.1in

{\bf Step 2. }\quad We prove that $(u_{\lambda_n,\beta},v_{\lambda_n,\beta})\to(u_{0,\beta}^{J_a},v_{0,\beta}^{J_b})$ strongly in $\h\times\h$ as $n\to\infty$ up to a subsequence, and the restriction of $u_{0,\beta}^{J_a}$ on $\Omega_{a,i_a}$ lies in $H_0^1(\Omega_{a,i_a})$ and is a critical point of $I_{\Omega_{a,i_a}}(u)$ for every $i_a\in J_a$, while the restriction of $v_{0,\beta}^{J_b}$ on $\Omega_{b,j_b}$ lies in $H_0^1(\Omega_{b,j_b})$ and is a critical point of $I_{\Omega_{b,j_b}}(v)$ for every $j_b\in J_b$.

Indeed, since $(u_{0,\beta}^{J_a},v_{0,\beta}^{J_b})\in H_0^1(\Omega_{a,0}^{J_a})\times H_0^1(\Omega_{b,0}^{J_b})$ with $u_{0,\beta}^{J_a}=0$ on $\bbr^3\backslash\Omega_{a,0}^{J_a}$ and $v_{0,\beta}^{J_b}=0$ on $\bbr^3\backslash\Omega_{b,0}^{J_b}$, by $D[J_{\lambda_n,\beta}^*(u_{\lambda_n,\beta}, v_{\lambda_n,\beta})]=0$ in $E^*$ and the condition $(D_3')$, we can see that the restriction of $u_{0,\beta}^{J_a}$ on $\Omega_{a,i_a}$, denoted by $u_{i_a,\beta}^{J_a}$, lies in $H_0^1(\Omega_{a,i_a})$ and $I_{\Omega_{a,i_a}}'(u_{i_a,\beta}^{J_a})=0$ in $H^{-1}(\Omega_{a,i_a})$ for all $i_a\in J_a$, while the restriction of $v_{0,\beta}^{J_b}$ on $\Omega_{b,j_b}$, denoted by $v_{j_b,\beta}^{J_b}$, lies in $H_0^1(\Omega_{b,j_b})$ and $I_{\Omega_{b,j_b}}'(v_{j_b,\beta}^{J_b})=0$ in $H^{-1}(\Omega_{b,j_b})$ for all $j_b\in J_b$.  Now, since $\beta<0$ and the condition $(D_3)$ contains the condition $(D_3')$, by the construction of $f_a(x,t)$, $f_b(x,t)$ and $h(x,t,s)$ and a similar argument as used in Step 2 of the proof for Theorem~\ref{thm0002}, we can conclude that $(u_{\lambda_n,\beta},v_{\lambda_n,\beta})\to (u_{0,\beta}^{J_a},b_{0,\beta}^{J_b})$ in $E$ as $n\to\infty$.  Since $E$ is embedded continuously into $\h\times\h$, $(u_{\lambda_n,\beta},v_{\lambda_n,\beta})\to (u_{0,\beta}^{J_a},v_{0,\beta}^{J_b})$ in $\h\times\h$ as $n\to\infty$.
\end{proof}

\vskip0.1in

In the following part, we will use a deformation argument to obtain the solution described by Theorem~\ref{thm0001}.  Let
$\ve_0=\frac14\min\{2\sqrt{m_{a,1}},\cdots,2\sqrt{m_{a,k_a}},2\sqrt{m_{b,1}},\cdots,2\sqrt{m_{b,k_b}}\}$.  For $0<\ve<\ve_0$, we define
\begin{eqnarray*}
\mathcal{D}_{a,\ve}&=&\bigg\{u\in E_a\mid\bigg(\int_{\bbr^3\backslash\Omega_{a,0}^{J_a}}|\nabla u|^2+u^2dx\bigg)^{\frac12}<\ve,\\
&&\bigg|\bigg(\int_{\Omega_{a,i_a}}|\nabla u|^2+a_0(x)u^2dx\bigg)^{\frac12}-2\sqrt{m_{a,i_a}}\bigg|<\ve,\forall i_a\in J_a\bigg\}
\end{eqnarray*}
and
\begin{eqnarray*}
\mathcal{D}_{b,\ve}&=&\bigg\{v\in E_b\mid\bigg(\int_{\bbr^3\backslash\Omega_{b,0}^{J_b}}|\nabla v|^2+v^2dx\bigg)^{\frac12}<\ve,\\
&&\bigg|\bigg(\int_{\Omega_{b,j_b}}|\nabla v|^2+b_0(x)v^2dx\bigg)^{\frac12}-2\sqrt{m_{b,j_b}}\bigg|<\ve, \forall j_b\in J_b\bigg\}.
\end{eqnarray*}
Let $\mathcal{D}_{\ve}=\mathcal{D}_{a,\ve}\cap\mathcal{D}_{b,\ve}$ and $J_{\lambda,\beta}^{m_{J_a,J_b}}=\bigg\{(u,v)\in E\mid J^*_{\lambda,\beta}(u,v)\leq m_{J_a,J_b}\bigg\}$,
where $m_{J_a,J_b}=\underset{i_a=1}{\overset{k_a}{\sum }}m_{a,i_a}+\underset{j_b=1}{\overset{k_b}{\sum }}m_{b,j_b}$.  Then we have the following.

\vskip0.1in

\begin{proposition}
\label{prop0004}  Assume $\beta<0$ and $0<\ve<\ve_0$.  Then there exists $\Lambda_3^*(\beta,\ve)\geq\Lambda_2^*(\beta)$ such that $J^*_{\lambda,\beta}(u,v)$ has a critical point in $\mathcal{D}_{\ve}\cap J_{\lambda,\beta}^{m_{J_a,J_b}}$ for $\lambda\geq\Lambda^*_3(\beta,\ve)$.
\end{proposition}
\begin{proof}
Suppose the contrary, since Lemma~\ref{lem0003} holds, there exist $\{\lambda_n\}$ and $\{c_{n,\beta}\}$ with $\lambda_n\to+\infty$ as $n\to\infty$ and $c_{n,\beta}>0$ for all $n$ such that
\begin{equation}     \label{eq0302}
\|D[J^*_{\lambda_n,\beta}(u,v)]\|_{E^*}\geq c_{n,\beta}\quad\text{for all }(u,v)\in\mathcal{D}_{\ve}\cap J_{\lambda_n,\beta}^{m_{J_a,J_b}}.
\end{equation}
For the sake of clarity, we divide the following proof into several steps.

\vskip0.1in

{\bf Step 1. } We prove that there exists $N\in\bbn$ and a constant $\sigma_0>0$ such that $\|J^*_{\lambda_n,\beta}(u,v)\|_{E^*}\geq\sigma_0$ for every $(u,v)\in(\mathcal{D}_{3\ve}\backslash\mathcal{D}_{\ve})\cap J_{\lambda_n,\beta}^{m_{J_a,J_b}}$ and $n\geq N$.

Suppose the contrary, there exists a subsequence of $\{\lambda_n\}$, still denoted by $\{\lambda_n\}$, such that $\|J^*_{\lambda_n,\beta}(u_{\lambda_n,\beta},v_{\lambda_n,\beta})\|_{E^*}\to0$ as $n\to\infty$ for some $(u_{\lambda_n,\beta},v_{\lambda_n,\beta})\in(\mathcal{D}_{3\ve}\backslash\mathcal{D}_{\ve})\cap J_{\lambda_n,\beta}^{m_{J_a,J_b}}$.  By a similar argument as used in Proposition~\ref{prop0001}, we can see that
$(u_{\lambda_n,\beta},v_{\lambda_n,\beta})\to(u_{0,\beta}^{J_a},v_{0,\beta}^{J_b})$ strongly in $\h\times\h$ and $E$ as $n\to\infty$ for some $(u_{0,\beta}^{J_a},v_{0,\beta}^{J_b})\in H_0^1(\Omega_{a,0}^{J_a})\times H_0^1(\Omega_{b,0}^{J_b})$ with $u_{0,\beta}^{J_a}=0$ on $\bbr^3\backslash\Omega_{a,0}^{J_a}$ and $v_{0,\beta}^{J_b}=0$ on $\bbr^3\backslash\Omega_{b,0}^{J_b}$.  The restriction of $u_{0,\beta}^{J_a}$ on $\Omega_{a,i_a}$, denoted by $u_{i_a,\beta}^{J_a}$, lies in $H_0^1(\Omega_{a,i_a})$ and $I_{\Omega_{a,i_a}}'(u_{i_a,\beta}^{J_a})=0$ in $H^{-1}(\Omega_{a,i_a})$ for all $i_a\in J_a$, while the restriction of $v_{0,\beta}^{J_b}$ on $\Omega_{b,j_b}$, denoted by $v_{j_b,\beta}^{J_b}$, lies in $H_0^1(\Omega_{b,j_b})$ and $I_{\Omega_{b,j_b}}'(v_{j_b,\beta}^{J_b})=0$ in $H^{-1}(\Omega_{b,j_b})$ for all $j_b\in J_b$.  Clearly, one of the following two cases must occur:
\begin{enumerate}
\item[$(1^*)$] $\int_{\Omega_{a,i_a}}(u_{i_a,\beta}^{J_a})^4dx\geq C$ for all $i_a\in J_a$ and $\int_{\Omega_{b,j_b}}(v_{j_b,\beta}^{J_b})^4dx\geq C$ for all $j_b\in J_b$.
\item[$(2^*)$] There exists $i_a'\in J_a$ or $j_b'\in J_b$ such that $\int_{\Omega_{a,i_a'}}(u_{i_a',\beta}^{J_a})^4dx=0$ or $\int_{\Omega_{b,j_b'}}(v_{j_b',\beta}^{J_b})^4dx=0$.
\end{enumerate}
If case~$(1^*)$ happens, then we must have $I_{\Omega_{a,i_a}}(u_{i_a,\beta}^{J_a})\geq m_{a,i_a}$ and $I_{\Omega_{b,j_b}}(v_{j_b,\beta}^{J_b})\geq m_{b,j_b}$ for all $i_a\in J_a$ and $j_b\in J_b$.  Since $(u_{\lambda_n,\beta},v_{\lambda_n,\beta})\in J_{\lambda_n,\beta}^{m_{J_a,J_b}}$, by a similar argument as used in Step 3 of the proof for Theorem~\ref{thm0002}, we can show that $I_{\Omega_{a,i_a}}(u_{i_a,\beta}^{J_a})= m_{a,i_a}$ and $I_{\Omega_{b,j_b}}(v_{j_b,\beta}^{J_b})= m_{b,j_b}$ for all $i_a\in J_a$ and $j_b\in J_b$.  It follows from the condition $(D_4)$ that
\begin{equation*}
\int_{\Omega_{a,i_a}}|\nabla u_{\lambda_n,\beta}|^2+a_0(x)u_{\lambda_n,\beta}^2dx\to4m_{a,i_a}\quad\text{for all }i_a\in J_a
\end{equation*}
and
\begin{equation*}
\int_{\Omega_{b,j_b}}|\nabla v_{\lambda_n,\beta}|^2+b_0(x)v_{\lambda_n,\beta}^2dx\to4m_{b,j_b}\quad\text{for all }j_b\in J_b
\end{equation*}
as $n\to\infty$, which then implies $(u_{\lambda_n,\beta},v_{\lambda_n,\beta})\in\mathcal{D}_{\ve}$ for $n$ large enough.  It is impossible since $(u_{\lambda_n,\beta},v_{\lambda_n,\beta})\in(\mathcal{D}_{3\ve}\backslash\mathcal{D}_{\ve})\cap J_{\lambda_n,\beta}^{m_{J_a,J_b}}$ for all $n$.  Thus, we must have the case~$(2^*)$.  Without loss of generality, we assume $\int_{\Omega_{a,1}}(u_{1,\beta}^{J_a})^4dx=0$.  It follows from the condition $(D_4)$ that
\begin{equation*}
\bigg|\bigg(\int_{\Omega_{a,1}}|\nabla u_{\lambda_n,\beta}|^2+a_0(x)u_{\lambda_n,\beta}^2dx\bigg)^{\frac12}-2\sqrt{m_{a,1}}\bigg|\to2\sqrt{m_{a,1}}=4\ve_0\quad\text{as }n\to\infty,
\end{equation*}
which implies $(u_{\lambda_n,\beta},v_{\lambda_n,\beta})\in E\backslash\mathcal{D}_{3\ve}$ for $n$ large enough.  It also contradicts to the fact that $(u_{\lambda_n,\beta},v_{\lambda_n,\beta})\in(\mathcal{D}_{3\ve}\backslash\mathcal{D}_{\ve})\cap J_{\lambda_n,\beta}^{m_{J_a,J_b}}$ for all $n$.

\vskip0.1in

{\bf Step 2. } We construct a descending flow on $J_{\lambda_n,\beta}^{m_{J_a,J_b}}$ for every $n\geq N$.

Let $\eta:E\to[0, 1]$ be a local Lipschitz continuous function and satisfy
\begin{equation*}
\eta(u,v)=\left\{\aligned&1,\quad&(u,v)\in\mathcal{D}_{\frac32\ve},\\
&0,\quad&(u,v)\in E\backslash\mathcal{D}_{2\ve}.\endaligned\right.
\end{equation*}
Since $J^*_{\lambda_n,\beta}(u,v)$ is $C^1$ for every $n\geq N$, there exists a pseudo-gradient vector field of $J^*_{\lambda_n,\beta}(u,v)$, denoted by $\widetilde{D[J^*_{\lambda_n,\beta}(u,v)]}=(\widetilde{D[J^*_{\lambda_n,\beta}(u,v)]}_1, \widetilde{D[J^*_{\lambda_n,\beta}(u,v)]}_2)$, which satisfies
\begin{enumerate}
\item[$(a_*)$] $\langle D[J^*_{\lambda_n,\beta}(u,v)], \widetilde{D[J^*_{\lambda_n,\beta}(u,v)]}\rangle_{E^*,E^*}\geq\frac12\|D[J^*_{\lambda_n,\beta}(u,v)]\|_{E^*}^2$ for all $(u,v)\in E$;
\item[$(b_*)$] $\|\widetilde{D[J^*_{\lambda_n,\beta}(u,v)]}\|_{E^*}\leq2\|D[J^*_{\lambda_n,\beta}(u,v)]\|_{E^*}$ for all $(u,v)\in E$.
\end{enumerate}
Let $\overrightarrow{\mathcal{V}}_{n}:J_{\lambda_n,\beta}^{m_{J_a,J_b}}\to E^*$ be a continuous map and given by
\begin{equation*}
\overrightarrow{\mathcal{V}}_{n}(u,v)=(\mathcal{V}_{1,n}(u,v),\mathcal{V}_{2,n}(u,v))
=-\frac{\eta(u,v)}{\|\widetilde{D[J^*_{\lambda_n,\beta}(u,v)]}\|_{E^*}}(\widetilde{D[J^*_{\lambda_n,\beta}(u,v)]}_1, \widetilde{D[J^*_{\lambda_n,\beta}(u,v)]}_2)
\end{equation*}
for $(u,v)\in J_{\lambda_n,\beta}^{m_{J_a,J_b}}\backslash\mathcal{K}=\{(u,v)\in J_{\lambda_n,\beta}^{m_{J_a,J_b}}\mid D[J^*_{\lambda_n,\beta}(u,v)]\not=0\}$ and
\begin{equation*}
\overrightarrow{\mathcal{V}}_{n}(u,v)=(\mathcal{V}_{1,n}(u,v),\mathcal{V}_{2,n}(u,v))
=(0, 0)
\end{equation*}
for $(u,v)\in J_{\lambda_n,\beta}^{m_{J_a,J_b}}\cap\mathcal{K}=\{(u,v)\in J_{\lambda_n,\beta}^{m_{J_a,J_b}}\mid D[J^*_{\lambda_n,\beta}(u,v)]=0\}$.  Clearly, $\|\overrightarrow{\mathcal{V}}_{n}(u,v)\|_{E^*}\leq1$ for all $(u,v)\in J_{\lambda_n,\beta}^{m_{J_a,J_b}}$.  Furthermore, by \eqref{eq0302},  Step 1 and the definition of $\eta$, we can see that $\overrightarrow{\mathcal{V}}_{n}(u,v)$ is locally  Lipschitz.  Now, let us consider the flow $\overrightarrow{\rho}_{n}(\tau)=(\rho_{1,n}(\tau),\rho_{2,n}(\tau))$ given by the following two-component system of ODE
\begin{equation*}
\left\{\aligned&\frac{d\rho_{1,n}(\tau)}{d\tau}=\mathcal{V}_{1,n}(\rho_{1,n}(\tau),\rho_{2,n}(\tau)),\\
&\frac{d\rho_{2,n}(\tau)}{d\tau}=\mathcal{V}_{2,n}(\rho_{1,n}(\tau),\rho_{2,n}(\tau)),\\
&(\rho_{1,n}(0),\rho_{2,n}(0))=(u,v)\in J_{\lambda_n,\beta}^{m_{J_a,J_b}}.\endaligned\right.
\end{equation*}
By $(a_*)$, $(b_*)$ and a direct calculation, we can see that
\begin{eqnarray*}
&&\frac{dJ^*_{\lambda_n,\beta}(\rho_{1,n}(\tau),\rho_{2,n}(\tau))}{d\tau}\\
& &=\bigg\langle(\frac{\partial J^*_{\lambda_n,\beta}}{\partial u}(\rho_{1,n},\rho_{2,n}),\frac{\partial J^*_{\lambda_n,\beta}}{\partial v}(\rho_{1,n},\rho_{2,n})),(\mathcal{V}_{1,n}(\rho_{1,n},\rho_{2,n}),
\mathcal{V}_{2,n}(\rho_{1,n},\rho_{2,n}))\bigg\rangle_{E^*,E^*}\\
& & \leq -\frac14\eta(\rho_{1,n},\rho_{2,n})\|D[J^*_{\lambda_n,\beta}(\rho_{1,n},\rho_{2,n})]\|_{E^*}\\
& &\leq 0.
\end{eqnarray*}
It follows that $\overrightarrow{\rho}_{n}(\tau)=(\rho_{1,n}(\tau),\rho_{2,n}(\tau))$ is a descending flow on $J_{\lambda_n,\beta}^{m_{J_a,J_b}}$.  Furthermore, for every $\tau>0$, we have $\overrightarrow{\rho}_{n}(\tau)=\overrightarrow{\mathcal{\rho}}_{n}(0)$ if $\overrightarrow{\mathcal{\rho}}_{n}(0)=(u,v)\in E\backslash\mathcal{D}_{2\ve}$.

\vskip0.1in

{\bf Step 3. } For every $n\geq N$, we construct a map $\overrightarrow{\rho}^0_{n}(\tau)=(\rho^0_{1,n}(\tau), \rho^0_{2,n}(\tau))\in\Gamma$ for all $\tau>0$ such that
\begin{equation}          \label{eq0095}
\sup_{[0, 1]^{k_a}\times[0, 1]^{k_b}}J^*_{\lambda_n,\beta}(\rho^0_{1,n}(\tau_n), \rho^0_{2,n}(\tau_n))<m_{J_a,J_b}-\sigma_0^*\quad\text{for some }\tau_n>0,
\end{equation}
where $\sigma_0^*>0$ is a constant.

Indeed, let $n\geq N$ and $\gamma_{0,a}$ and $\gamma_{0,b}$ be given by \eqref{eq0097} and \eqref{eq0098}.  We consider
$\overrightarrow{\rho}^0_{n}(\tau)=(\rho^0_{1,n}(\tau), \rho^0_{2,n}(\tau))$, where $(\rho^0_{1,n}(0), \rho^0_{2,n}(0))=(\gamma_{0,a},\gamma_{0,b})$.  Since $W_{a,i_a}$ and $W_{b,j_b}$ are  the least energy nonzero critical  points of $I_{\Omega_{a,i_a}}(u)$ and $I_{\Omega_{b,j_b}}(v)$ for all $i_a\in J_a$ and $j_b\in J_b$, by the choice of $R$ and $\ve$, we can see that
\begin{equation*}
(\gamma_{0,a}(t_1,\cdots,t_{k_a}),\gamma_{0,b}(s_1,\cdots,s_{k_b}))\in E\backslash\mathcal{D}_{2\ve}
\end{equation*}
for every $(t_1,\cdots,t_{k_a},s_1,\cdots,s_{k_b})\in\partial([0, 1]^{k_a}\times[0,1]^{k_b})$.  It follows from the construction of $(\rho^0_{1,n}(\tau), \rho^0_{2,n}(\tau))$ that
\begin{equation*}
(\rho^0_{1,n}(\tau)(t_1,\cdots,t_{k_a}),\rho^0_{2,n}(\tau)(s_1,\cdots,s_{k_b}))
=(\gamma_{0,a}(t_1,\cdots,t_{k_a}),\gamma_{0,b}(s_1,\cdots,s_{k_b}))
\end{equation*}
for every $(t_1,\cdots,t_{k_a},s_1,\cdots,s_{k_b})\in\partial([0, 1]^{k_a}\times[0,1]^{k_b})$ and $\tau>0$.  Thus, $\overrightarrow{\rho}^0_{n}(\tau)\in\Gamma$ for all $\tau>0$.  It remains to show \eqref{eq0095} holds.  For every $(t_1,\cdots,t_{k_a},s_1,\cdots,s_{k_b})\in[0, 1]^{k_a}\times[0,1]^{k_b}$, one of the following two cases must occur:
\begin{enumerate}
\item[$(a^*)$] $(\gamma_{0,a}(t_1,\cdots,t_{k_a}),\gamma_{0,b}(s_1,\cdots,s_{k_b}))\in E\backslash\mathcal{D}_{\ve}$.
\item[$(b^*)$] $(\gamma_{0,a}(t_1,\cdots,t_{k_a}),\gamma_{0,b}(s_1,\cdots,s_{k_b}))\in \mathcal{D}_{\ve}$.
\end{enumerate}
If case~$(a^*)$ happen, then by Step 2, we must have
\begin{equation*}
J^*_{\lambda_n,\beta}(\rho^0_{1,n}(\tau)(t_1,\cdots,t_{k_a}),\rho^0_{2,n}(\tau)(s_1,\cdots,s_{k_b}))
\leq J^*_{\lambda_n,\beta}(\gamma_{0,a}(t_1,\cdots,t_{k_a}),\gamma_{0,b}(s_1,\cdots,s_{k_b}))
\end{equation*}
for all $\tau>0$.  Moreover, by \eqref{eq0029} and the choice of $\{W_{a,i_a}\}$ and $\{W_{b,j_b}\}$, we can see that
\begin{eqnarray*}
J^*_{\lambda_n,\beta}(\gamma_{0,a}(t_1,\cdots,t_{k_a}),\gamma_{0,b}(s_1,\cdots,s_{k_b}))
=\sum_{i_a=1}^{k_a}m_{a,i_a}+\sum_{j_b=1}^{k_b}m_{b,j_b}
\end{eqnarray*}
if and only if $t_{i_a}=s_{j_b}=\frac1R$ for all $i_a\in J_a$ and $j_b\in J_b$.  Since $(\gamma_{0,a}(t_1,\cdots,t_{k_a}),\gamma_{0,b}(s_1,\cdots,s_{k_b}))\in E\backslash\mathcal{D}_{\ve}$ in this case, there exists $i_a'\in J_a$ or $j_b'\in J_b$ such that $t_{i_a'}\not=\frac1R$ or $s_{j_b'}\not=\frac1R$.  It follows the construction of $\gamma_{0,a}$ and $\gamma_{0,b}$ and the condition $(D_3')$ that
\begin{equation}      \label{eq0301}
m_{a,b}^*=\sup_{(u,v)\in\mathbb{P}}J^*_{\lambda_n,\beta}(u,v)=\sup_{(u,v)\in\mathbb{P}}(I_{\Omega_a}(u)+I_{\Omega_b}(v))
<\sum_{i_a=1}^{k_a}m_{a,i_a}+\sum_{j_b=1}^{k_b}m_{b,j_b},
\end{equation}
where $\mathbb{P}=(\gamma_{0,a}([0, 1]^{k_a})\times\gamma_{0,b}([0,1]^{k_b}))\backslash\mathcal{D}_{\ve}$.
If case~$(b^*)$ happen, then two subcases may occur:
\begin{enumerate}
\item[$(b^*_1)$] $(\rho^0_{1,n}(\tau)(t_1,\cdots,t_{k_a}),\rho^0_{2,n}(\tau)(s_1,\cdots,s_{k_b}))\in\mathcal{D}_{\frac32\ve}$ for all $\tau>0$.
\item[$(b^*_2)$] There exists $\tau_{n}^*>0$ such that $(\rho^0_{1,n}(\tau_n^*)(t_1,\cdots,t_{k_a}),\rho^0_{2,n}(\tau_n^*)(s_1,\cdots,s_{k_b}))\in E\backslash\mathcal{D}_{\frac32\ve}$.
\end{enumerate}
In the subcase~$(b^*_1)$, by Step 2 and the Taylor  expansion, we can calculate that
\begin{eqnarray}
&&J^*_{\lambda_n,\beta}(\rho^0_{1,n}(\tau)(t_1,\cdots,t_{k_a}),\rho^0_{2,n}(\tau)(s_1,\cdots,s_{k_b}))\notag\\
& & \leq  J^*_{\lambda_n,\beta}(\gamma_{0,a}(t_1,\cdots,t_{k_a}),\gamma_{0,b}(s_1,\cdots,s_{k_b}))\notag\\
&&\quad -\int_0^\tau\frac14\eta(\rho_{1,n},\rho_{2,n})\|D[J^*_{\lambda_n,\beta}(\rho_{1,n},\rho_{2,n})]\|_{E^*}d\nu\notag\\
& & \leq \sum_{i_a=1}^{k_a}m_{a,i_a}+\sum_{j_b=1}^{k_b}m_{b,j_b}-\frac14\tau\min\{c_{n,\beta},\sigma_0\},
\end{eqnarray}
where $c_{n,\beta}$ is given by \eqref{eq0302} and $\sigma_0$ is given by Step 1.
It follows that
\begin{equation*}
J^*_{\lambda_n,\beta}(\rho^0_{1,n}(\tau)(t_1,\cdots,t_{k_a}),\rho^0_{2,n}(\tau)(s_1,\cdots,s_{k_b}))<\underset{i_a=1}{\overset{k_a}{\sum }}m_{a,i_a}+\underset{j_b=1}{\overset{k_b}{\sum }}m_{b,j_b}-\sigma_0
\end{equation*}
for $\tau\geq\tau_n^0=\frac{8\sigma_0}{\min\{c_{n,\beta},\sigma_0\}}$.  In the subcase~$(b^*_2)$, there must exist $0\leq\tau_{n,a}^*<\tau_{n,b}^*\leq\tau_{n}^*$ such that
\begin{equation*}
(\rho^0_{1,n}(\tau_{n,a}^*)(t_1,\cdots,t_{k_a}),\rho^0_{2,n}(\tau_{n,a}^*)(s_1,\cdots,s_{k_b}))\in\partial\mathcal{D}_{\ve}
\end{equation*}
and
\begin{equation*}
(\rho^0_{1,n}(\tau_{n,b}^*)(t_1,\cdots,t_{k_a}),\rho^0_{2,n}(\tau_{n,b}^*)(s_1,\cdots,s_{k_b}))\in\partial\mathcal{D}_{\frac32\ve}.
\end{equation*}
For the sake of convenience, we respectively denote $(\rho^0_{1,n}(\tau_{n,a}^*)(t_1,\cdots,t_{k_a}),\rho^0_{2,n}(\tau_{n,a}^*)(s_1,\cdots,s_{k_b}))$ and $(\rho^0_{1,n}(\tau_{n,b}^*)(t_1,\cdots,t_{k_a}),\rho^0_{2,n}(\tau_{n,b}^*)(s_1,\cdots,s_{k_b}))$ by $(u_{n,1},v_{n,1})$ and $(u_{n,2}, v_{n,2})$.
Since Lemma~\ref{lem0001} holds for $\lambda\geq\Lambda_1$, by a similar argument as used for (4.14) in \cite{DT03}, we can see that one of the following four cases must happen:
\begin{enumerate}
\item[$(a^{**})$] $\int_{\Omega_a}|\nabla(u_{n,1}-u_{n,2})|^2+a_0(x)(u_{n,1}-u_{n,2})^2dx\geq\frac{\ve^2}{4}$.
\item[$(b^{**})$] $\int_{\Omega_b}|\nabla(v_{n,1}-v_{n,2})|^2+b_0(x)(v_{n,1}-v_{n,2})^2dx\geq\frac{\ve^2}{4}$.
\item[$(c^{**})$] $\int_{\bbr^3\backslash\Omega_{a,0}^{J_a}}|\nabla(u_{n,1}-u_{n,2})|^2+(u_{n,1}-u_{n,2})^2dx\geq\frac{\ve^2}{4}$.
\item[$(d^{**})$] $\int_{\bbr^3\backslash\Omega_{b,0}^{J_b}}|\nabla(v_{n,1}-v_{n,2})|^2+(v_{n,1}-v_{n,2})^2dx\geq\frac{\ve^2}{4}$.
\end{enumerate}
By similar arguments as  \eqref{eq0004}-\eqref{eq0005}, we can see that in any case, there exists a constant $C(\ve)>0$ such that $\|(u_{n,1},v_{n,1})-(u_{n,2}, v_{n,2})\|\geq C(\ve)$.  On the other hand, by Step 2, we can see that $(u_{n,1},v_{n,1})\in J_{\lambda_n,\beta}^{m_{J_a,J_b}}$ and $(u_{n,2},v_{n,2})\in J_{\lambda_n,\beta}^{m_{J_a,J_b}}$.  It follows from $\|\overrightarrow{\mathcal{V}}_{n}(u,v)\|_{E^*}\leq1$ for all $(u,v)\in J_{\lambda_n,\beta}^{m_{J_a,J_b}}$ and the Taylor expansion that $\tau_{n,b}^*-\tau_{n,a}^*\geq C(\ve)$.  Now, by Step ~1, we have
\begin{eqnarray}
&&J_{\lambda_n,\beta}(\rho^0_{1,n}(\tau_n^*)(t_1,\cdots,t_{k_a}),\rho^0_{2,n}(\tau_n^*)(s_1,\cdots,s_{k_b}))\notag\\
& & \leq  J_{\lambda_n,\beta}(\gamma_{0,a}(t_1,\cdots,t_{k_a}),\gamma_{0,b}(s_1,\cdots,s_{k_b}))\notag\\
&&\quad -\int_{\tau_{n,a}^*}^{\tau_{n,b}^*}\frac14\eta(\rho_{1,n},\rho_{2,n})\|D[J_{\lambda_n,\beta}(\rho_{1,n},\rho_{2,n})]\|_{E^*}d\nu\notag\\
&  & \leq \sum_{i_a=1}^{k_a}m_{a,i_a}+\sum_{j_b=1}^{k_b}m_{b,j_b}-\frac14C(\ve)\sigma_0.\label{eq0303}
\end{eqnarray}
Let $\tau_n=\max\{\tau_n^0,\tau_n^*\}$ and $\sigma_0^*=\min\{\frac14C(\ve)\sigma_0,\sigma_0,\underset{i_a=1}{\overset{k_a}{\sum }}m_{a,i_a}+\underset{j_b=1}{\overset{k_b}{\sum }}m_{b,j_b}-m_{a,b}^*\}$.  Then \eqref{eq0095} follows from \eqref{eq0301}-\eqref{eq0303} and Step 2.

Since $\overrightarrow{\rho}^0_{n}(\tau)=(\rho^0_{1,n}(\tau), \rho^0_{2,n}(\tau))\in\Gamma$ for all $\tau>0$, by the definition of $m_{J_a,J_b,\lambda_n,\beta}$ and Step 3, we can see that $m_{J_a,J_b,\lambda_n,\beta}\leq m_{J_a,J_b}-\sigma_0^*$, which is impossible since Lemma~\ref{lem0006} holds and $m_{a,i_a,\lambda_n}\to m_{a,i_a}$ and $m_{b,j_b,\lambda_n}\to m_{b,j_b}$ as $n\to\infty$ for all $i_a\in J_a$ and $j_b\in J_b$.
\end{proof}

\vskip0.1in

We close this section by

\vskip0.1in

\medskip\par\noindent{\bf Proof of Theorem~\ref{thm0001}:}\quad Suppose $\beta<0$.  Then by Proposition~\ref{prop0004}, for every $\ve\in(0, \ve_0)$, there exists $\Lambda^*_3(\beta,\ve)>\Lambda_2^*(\beta)$ such that $J_{\lambda,\beta}^*(u,v)$ has a critical point $(u_{\lambda,\beta}^{J_a},v_{\lambda,\beta}^{J_b})\in \mathcal{D}_{\ve}\cap J_{\lambda,\beta}^{m_{J_a,J_b}}$ for all $\lambda\geq\Lambda_3^*(\beta,\ve)$.  Thanks to Lemma~\ref{lem0004} and the choice of $\Lambda^*_3(\beta,\ve)$, $(u_{\lambda,\beta}^{J_a},v_{\lambda,\beta}^{J_b})$ is also a critical point of $J_{\lambda,\beta}(u,v)$.  Since $u_{\lambda,\beta}^{J_a}$ and $v_{\lambda,\beta}^{J_b}$ are both nonnegative by the construction of $J_{\lambda,\beta}^*(u,v)$, we can use a similar argument as used in the proof of Theorem~\ref{thm0002} to show that $u_{\lambda,\beta}^{J_a}$ and $v_{\lambda,\beta}^{J_b}$ are both positive, which implies $(u_{\lambda,\beta}^{J_a},v_{\lambda,\beta}^{J_b})$ is a solution of $(\mathcal{P}_{\lambda,\beta})$.  Clearly, $(u_{\lambda,\beta}^{J_a},v_{\lambda,\beta}^{J_b})$ satisfies the concentration behaviors of $(1)$ and $(2)$, since $\Lambda^*_3(\beta,\ve)\to+\infty$ as $\ve\to0$ and $(u_{\lambda,\beta}^{J_a},v_{\lambda,\beta}^{J_b})\in \mathcal{D}_{\ve}$.  Furthermore, since $(u_{\lambda,\beta}^{J_a},v_{\lambda,\beta}^{J_b})\in J_{\lambda,\beta}^{m_{J_a,J_b}}$, by a similar argument as used in  Proposition~\ref{prop0001}, we can see that the properties $(3)$ and $(4)$ are also hold.  It remains to show that the concentration behavior $(5)$ is also true.  Indeed, by a similar argument as used in Proposition~\ref{prop0001}, the restriction of $u_{0,\beta}^{J_a}$ on $\Omega_{a,i_a}$, denoted by $u^{J_a}_{i_a,\beta}$, lies in $H_0^1(\Omega_{a,i_a})$ and is a critical point of $I_{\Omega_{a,i_a}}(u)$ for every $i_a\in J_a$, while the restriction of $v_{0,\beta}^{J_b}$ on $\Omega_{b,j_b}$, denoted by $v^{J_b}_{j_b,\beta}$, lies in $H_0^1(\Omega_{b,j_b})$ and is a critical point of $I_{\Omega_{b,j_b}}(v)$ for every $j_b\in J_b$.  Since $(u_{\lambda,\beta}^{J_a},v_{\lambda,\beta}^{J_b})\in \mathcal{D}_{\ve}$ for all $\lambda\geq\Lambda_3^*(\beta,\ve)$, it is easy to see that $u^{J_a}_{i_a,\beta}$ and $v^{J_b}_{j_b,\beta}$ are nonzero for all $i_a\in J_a$ and $j_b\in J_b$.  It follows from $(u_{\lambda,\beta}^{J_a},v_{\lambda,\beta}^{J_b})\in J_{\lambda,\beta}^{m_{J_a,J_b}}$ that $u^{J_a}_{i_a,\beta}$ and $v^{J_b}_{j_b,\beta}$ are  the least energy nonzero critical points of $I_{\Omega_{a,i_a}}(u)$ and $I_{\Omega_{b,j_b}}(v)$ for every $i_a\in J_a$ and $j_b\in J_b$, respectively.  The proof of this theorem can be finished by taking $\Lambda_{*}(\beta)=\Lambda_3^*(\beta,\frac{\ve_0}{2})$.
\qquad\raisebox{-0.5mm}{\rule{1.5mm}{4mm}}\vspace{6pt}

\section{The phenomenon of phase separations}
In this section, we study the phenomenon of phase separations to $(\mathcal{P}_{\lambda,\beta})$, that is, we study the concentration behavior of the solutions for $(\mathcal{P}_{\lambda,\beta})$ as $\beta\to-\infty$.

\vskip0.2in

\noindent\textbf{Proof of Theorem~\ref{thm0003}:}\quad Suppose $\lambda\geq\Lambda_*$ and $\{\beta_n\}\subset(-\infty, 0)$ satisfying $\beta_n\to-\infty$ as $n\to\infty$.  Let $(u_{\lambda,\beta_n}, v_{\lambda,\beta_n})$ be the ground state solution of $(\mathcal{P}_{\lambda,\beta_n})$ obtained by Theorem~\ref{thm0002}.  Then by Lemma~\ref{lem0008} and a similar argument of \eqref{eq0332}, we can see that $\{(u_{\lambda,\beta_n}, v_{\lambda,\beta_n})\}$ is bounded in $E$.  Without loss of generality, we assume $(u_{\lambda,\beta_n}, v_{\lambda,\beta_n})\rightharpoonup(u_{\lambda,0}, v_{\lambda,0})$ weakly in $E$ as $n\to\infty$ for some $(u_{\lambda,0}, v_{\lambda,0})\in E$.  Since $E$ is embedded continuously into $\h\times\h$, we have $(u_{\lambda,0}, v_{\lambda,0})\in \h\times\h$.  Clearly, $u_{\lambda,0}\geq0$ and $v_{\lambda,0}\geq0$ in $\bbr^3$.  In what follows, we verify that $(u_{\lambda,0}, v_{\lambda,0})$ satisfies $(1)$-$(4)$.  For the sake of clarity, we divide the following proof into several steps.

\vskip0.1in

{\bf Step 1. }\quad We prove that there exists $\Lambda_{**}\geq\Lambda_*$ such that $u_{\lambda,0}\not=0$ and $v_{\lambda,0}\not=0$ in $\bbr^3$ for $\lambda\geq\Lambda_{**}$ in the sense of almost everywhere.

Indeed, suppose $u_{\lambda,0}=0$ a.e. in $\bbr^3$.  Since $u_{\lambda,\beta_n}\rightharpoonup u_{\lambda,0}$ weakly in $E_a$ as $n\to\infty$ and $\{(u_{\lambda,\beta_n}, v_{\lambda,\beta_n})\}$ is bounded in $E$, by a similar argument as \eqref{eq0310}, we can see that
\begin{equation*}
\|u_{\lambda,\beta_n}\|_{a,\lambda}^2\leq C\lambda^{-\frac12}\|u_{\lambda,\beta_n}\|_{a,\lambda}^2+o_n(1).
\end{equation*}
It follows that there exists $\Lambda_{**}\geq\Lambda_*$ such that $\|u_{\lambda,\beta_n}\|_{a,\lambda}^2=o_n(1)$ for $\lambda\geq\Lambda_{**}$, which together Lemma~\ref{lem0001} and the boundedness of $\{(u_{\lambda,\beta_n}, v_{\lambda,\beta_n})\}$ in $E$, implies $\|u_{\lambda,\beta_n}\|_4=o_n(1)$.  Thanks to the fact that $(u_{\lambda,\beta_n}, v_{\lambda,\beta_n})$ is the ground state solution of $(\mathcal{P}_{\lambda,\beta_n})$ with $\beta_n<0$ for every $n$, we also have $\beta_n\int_{\bbr^3}u_{\lambda,\beta_n}^2v_{\lambda,\beta_n}^2dx\to0$ as $n\to\infty$.  Hence, $J_{\lambda,\beta_n}(u_{\lambda,\beta_n}, v_{\lambda,\beta_n})=I_{b,\lambda}(v_{\lambda,\beta_n})+o_n(1)$.  On the other hand, by Lemma~\ref{lem0001}, for every $n\in\bbn$, there exists $t_n>0$ such that $t_nu_{\lambda,\beta_n}\in\mathcal{N}_{a,\lambda}$ with $\lambda\geq\lambda_{**}$.  It follows from Lemma~\ref{lem0007} and $\beta_n<0$ that
\begin{eqnarray}
J_{\lambda,\beta_n}(u_{\lambda,\beta_n}, v_{\lambda,\beta_n})&\geq&J_{\lambda,\beta_n}(t_nu_{\lambda,\beta_n}, v_{\lambda,\beta_n})\notag\\
&\geq& I_{a,\lambda}(t_nu_{\lambda,\beta_n})+I_{b,\lambda}(v_{\lambda,\beta_n})\notag\\
&\geq& m_{a,\lambda}+J_{\lambda,\beta_n}(u_{\lambda,\beta_n}, v_{\lambda,\beta_n})+o_n(1)\label{eq0404}.
\end{eqnarray}
Since $m_{a,\lambda}>0$ for $\lambda\geq\lambda_{**}$, \eqref{eq0404} is impossible for $n$ large enough.  By a similar argument, we can also show that $v_{\lambda,0}\not=0$ in $\bbr^3$ for $\lambda\geq\Lambda_{**}$ in the sense of almost everywhere.

\vskip0.1in

{\bf Step 2. }\quad We prove that $\{u_{\lambda,\beta_n}\}, \{v_{\lambda,\beta_n}\}\subset C(\bbr^3)$ and $\|u_{\lambda,\beta_n}\|_{C(\bbr^3)}\leq C_0$ and $\|v_{\lambda,\beta_n}\|_{C(\bbr^3)}\leq C_0$ for some $C_0>0$.

Indeed, by a similar argument as used in the proof of Theorem~\ref{thm0002}, we have $\{u_{\lambda,\beta_n}\}, \{v_{\lambda,\beta_n}\}\subset C(\bbr^3)$. On the other hand, since $(u_{\lambda,\beta_n}, v_{\lambda,\beta_n})$ is the ground state solution of $(\mathcal{P}_{\lambda,\beta_n})$ obtained by Theorem~\ref{thm0002} and $\{(u_{\lambda,\beta_n}, v_{\lambda,\beta_n})\}$ is bounded in $E$, we can use a similar argument as used in $(3)$ of Lemma~\ref{lem0004} to show that $\{u_{\lambda,\beta_n}\}$ and $\{v_{\lambda,\beta_n}\}$ are bounded in $L^\infty(\bbr^3)$, that is, $\|u_{\lambda,\beta_n}\|_{L^\infty(\bbr^3)}\leq C_0$ and $\|v_{\lambda,\beta_n}\|_{L^\infty(\bbr^3)}\leq C_0$ for some $C_0>0$, which implies $\|u_{\lambda,\beta_n}\|_{C(\bbr^3)}\leq C_0$ and $\|v_{\lambda,\beta_n}\|_{C(\bbr^3)}\leq C_0$.

\vskip0.1in

{\bf Step 3. }\quad We prove that $u_{\lambda,0}, v_{\lambda,0}\in C(\bbr^3)$ and are all local Lipschitz in $\bbr^3$.

Indeed, since the conditions $(D_1)$-$(D_5)$ hold, by \cite[Theorem~1.7]{SZ14} and Step 2, we can see that $\{\nabla u_{\lambda,\beta_n}\}$ and $\{\nabla v_{\lambda,\beta_n}\}$ are bounded in $L^\infty(\bbr^3)$.  On the other hand, for every $n$, by a similar argument as used in $(3)$ of Lemma~\ref{lem0004}, we can show that $u_{\lambda,\beta_n}, v_{\lambda,\beta_n}\in L^\gamma(\bbr^3)$ for all $\gamma\geq2$.  Thanks to the Calderon-Zygmund inequality and conditions $(D_1)$-$(D_5)$, we have $u_{\lambda,\beta_n}, v_{\lambda,\beta_n}\in W_{loc}^{2,\gamma}(\bbr^3)$ for all $\gamma\geq2$.  Together with the Sobolev embedding theorem, it implies $u_{\lambda,\beta_n}, v_{\lambda,\beta_n}\in C^1(\bbr^3)$.  It follows that $\{u_{\lambda,\beta_n}\}$ and $\{v_{\lambda,\beta_n}\}$ are bounded in $C^1(\bbr^3)$.  Now, by applying the Ascoli-Arzel\'a theorem, we can conclude that $u_{\lambda,\beta_n}\to u_{\lambda,0}$ and $v_{\lambda,\beta_n}\to v_{\lambda,0}$ strongly in $C_{loc}(\bbr^3)$ as $n\to\infty$ with $u_{\lambda,0}, v_{\lambda,0}\in C(\bbr^3)$.  This together with the boundness of $\{u_{\lambda,\beta_n}\}$ and $\{v_{\lambda,\beta_n}\}$ in $C^1(\bbr^3)$ again, implies $u_{\lambda,0}$ and $v_{\lambda,0}$ are all local Lipschitz in $\bbr^3$.

\vskip0.1in

{\bf Step 4. }\quad We prove that $(u_{\lambda,\beta_n}, v_{\lambda,\beta_n})\to(u_{\lambda,0}, v_{\lambda,0})$ strongly in $\h\times\h$ as $n\to\infty$.  Furthermore, $u_{\lambda,0}\in H_0^1(\{u_{\lambda,0}>0\})$ and is a least energy solution of \eqref{eq0401}, while $v_{\lambda,0}\in H_0^1(\{v_{\lambda,0}>0\})$ and is a least energy solution of \eqref{eq0402}.

Indeed, since $u_{\lambda,0}\in C(\bbr^3)$ and is local Lipschitz in $\bbr^3$, we can conclude that $\partial\{u_{\lambda,0}>0\}$, the boundary of the set $\{u_{\lambda,0}>0\}$, is local Lipschitz.  It follows from $u_{\lambda,0}\in H^1(\bbr^3)$ and $u_{\lambda,0}=0$ in $\bbr^3\backslash\{u_{\lambda,0}>0\}$ that $u_{\lambda,0}\in H_0^1(\{u_{\lambda,0}>0\})$.  Similarly, we have
$v_{\lambda,0}\in H_0^1(\{v_{\lambda,0}>0\})$.  Let $I_{a,\lambda}^*(u)$ and $I_{b,\lambda}^*(v)$ respectively be the corresponding functional of \eqref{eq0401} and \eqref{eq0402}.  By a similar argument as used in Lemma~\ref{lem0001}, we can show that
\begin{equation}      \label{eq0712}
C\int_{\{u_{\lambda,0}>0\}}u^2dx\leq\int_{\{u_{\lambda,0}>0\}}|\nabla u|^2+(\lambda a(x)+a_0(x))u^2dx\quad\text{for all }u\in H_0^1(\{u_{\lambda,0}>0\})
\end{equation}
and
\begin{equation}      \label{eq0713}
C\int_{\{v_{\lambda,0}>0\}}v^2dx\leq\int_{\{v_{\lambda,0}>0\}}|\nabla v|^2+(\lambda b(x)+b_0(x))v^2dx\quad\text{for all }v\in H_0^1(\{v_{\lambda,0}>0\})
\end{equation}
if $\lambda\geq\Lambda_1$.
It follows that the Nehari manifolds of $I_{a,\lambda}^*(u)$ and $I_{b,\lambda}^*(v)$ are both well defined if $\lambda\geq\Lambda_1$.  Let
\begin{equation*}
m_{a,\lambda}^*=\inf_{\mathcal{N}_{a,\lambda}^*}I_{a,\lambda}^*(u)\quad\text{and}\quad m_{b,\lambda}^*=\inf_{\mathcal{N}_{b,\lambda}^*}I_{b,\lambda}^*(v),
\end{equation*}
where $\mathcal{N}_{a,\lambda}^*$ and $\mathcal{N}_{b,\lambda}^*$ are respectively the Nehari manifolds of $I_{a,\lambda}^*(u)$ and $I_{b,\lambda}^*(v)$.  Then $m_{a,\lambda}^*>0$ and $m_{b,\lambda}^*>0$ if $\lambda\geq\Lambda_1$.  For every $\ve>0$, there exist $u_\ve\in \mathcal{N}_{a,\lambda}^*$ and $v_\ve\in \mathcal{N}_{b,\lambda}^*$ such that
\begin{equation}     \label{eq0403}
I_{a,\lambda}^*(u_\ve)<m_{a,\lambda}^*+\ve\quad\text{and}\quad I_{b,\lambda}^*(v_\ve)<m_{b,\lambda}^*+\ve.
\end{equation}
Since $\{(u_{\lambda,\beta_n}, v_{\lambda,\beta_n})\}$ is bounded in $E$, by the fact that $(u_{\lambda,\beta_n}, v_{\lambda,\beta_n})$ is the ground state solution $(\mathcal{P}_{\lambda,\beta_n})$ and $\beta_n\to-\infty$, we have $\int_{\bbr^3}u_{\lambda,\beta_n}^2v_{\lambda,\beta_n}^2dx\to0$ as $n\to\infty$.  It follows from the Fatou lemma that $\int_{\bbr^3}u_{\lambda,0}^2v_{\lambda,0}^2dx=0$, which implies $\{u_{\lambda,0}>0\}\cap\{v_{\lambda,0}>0\}=\emptyset$.  Hence, by $u_\ve\in \mathcal{N}_{a,\lambda}^*$ and $v_\ve\in \mathcal{N}_{b,\lambda}^*$, we can see that $(u_\ve,v_\ve)\in\mathcal{N}_{\lambda,\beta_n}$ for all $n$.  Now, by
\eqref{eq0403}, we have
\begin{equation*}
2\ve+m_{a,\lambda}^*+m_{b,\lambda}^*\geq I_{a,\lambda}^*(u_\ve)+I_{b,\lambda}^*(v_\ve)=J_{\lambda,\beta_n}(u_\ve,v_\ve)\geq m_{\lambda,\beta_n}\quad\text{for all }n.
\end{equation*}
Since $\ve>0$ and $n\in\bbn$ is arbitrary, we can conclude that
\begin{equation}    \label{eq0405}
m_{a,\lambda}^*+m_{b,\lambda}^*\geq\limsup_{n\to\infty}m_{\lambda,\beta_n}.
\end{equation}
On the other hand, note that $(u_{\lambda,\beta_n}, v_{\lambda,\beta_n})$ is the ground state solution $(\mathcal{P}_{\lambda,\beta_n})$ and $(u_{\lambda,\beta_n}, v_{\lambda,\beta_n})\rightharpoonup(u_{\lambda,0}, v_{\lambda,0})$ weakly in $E$ as $n\to\infty$, by $\beta_n<0$, we can see that
\begin{equation*}
\int_{\{u_{\lambda,0}>0\}}|\nabla u_{\lambda,0}|^2+(\lambda a(x)+a_0(x))u_{\lambda,0}^2dx\leq\mu_1\int_{\{u_{\lambda,0}>0\}}u_{\lambda,0}^4dx
\end{equation*}
and
\begin{equation*}
\int_{\{v_{\lambda,0}>0\}}|\nabla v_{\lambda,0}|^2+(\lambda b(x)+b_0(x))v_{\lambda,0}^2dx\leq\mu_2\int_{\{v_{\lambda,0}>0\}}v_{\lambda,0}^4dx.
\end{equation*}
It follows from \eqref{eq0712} and \eqref{eq0713} that there exist $0<t_0\leq1$ and $0<s_0\leq1$ such that $t_0u_{\lambda,0}\in \mathcal{N}_{a,\lambda}^*$ and $s_0v_{\lambda,0}\in \mathcal{N}_{b,\lambda}^*$.  Now, since $(u_{\lambda,\beta_n}, v_{\lambda,\beta_n})\rightharpoonup(u_{\lambda,0}, v_{\lambda,0})$ weakly in $E$ as $n\to\infty$, by a similar argument as \eqref{eq7100}, we can see that
\begin{eqnarray}
\liminf_{n\to+\infty}m_{\lambda,\beta_n}&=&\liminf_{n\to+\infty}J_{\lambda,\beta_n}(u_{\lambda,\beta_n}, v_{\lambda,\beta_n})\notag\\
&=&\frac14\liminf_{n\to+\infty}(\|u_{\lambda,\beta_n}\|_{a,\lambda}^2+\|v_{\lambda,\beta_n}\|_{b,\lambda}^2)\notag\\
&\geq&\frac14(\|u_{\lambda,0}\|_{a,\lambda}^2+\|v_{\lambda,0}\|_{b,\lambda}^2)\notag\\
&\geq&\frac14(\|t_0u_{\lambda,0}\|_{a,\lambda}^2+\|s_0v_{\lambda,0}\|_{b,\lambda}^2)\notag\\
&=&I_{a,\lambda}^*(t_0u_{\lambda,0})+I_{b,\lambda}^*(s_0v_{\lambda,0})\notag\\
&\geq&m_{a,\lambda}^*+m_{b,\lambda}^*.\label{eq0406}
\end{eqnarray}
Hence, by combining \eqref{eq0405} and \eqref{eq0406}, we must have the following results:
\begin{enumerate}
\item[$(a^{***})$] $\lim_{n\to\infty}\|u_{\lambda,\beta_n}\|_{a,\lambda}^2=\|u_{\lambda,0}\|_{a,\lambda}^2$ and $\lim_{n\to\infty}\|v_{\lambda,\beta_n}\|_{b,\lambda}^2=\|v_{\lambda,0}\|_{b,\lambda}^2$.
\item[$(b^{***})$] $u_{\lambda,0}\in \mathcal{N}_{a,\lambda}^*$ with $I_{a,\lambda}^*(u_{\lambda,0})=m_{a,\lambda}^*$ and $v_{\lambda,0}\in \mathcal{N}_{b,\lambda}^*$ with $I_{b,\lambda}^*(v_{\lambda,0})=m_{b,\lambda}^*$.
\end{enumerate}
By $(a^{***})$ and Lemma~\ref{lem0001}, we know that $\|u_{\lambda,\beta_n}-u_{\lambda,0}\|_{a,\lambda}=\|v_{\lambda,\beta_n}-v_{\lambda,0}\|_{b,\lambda}=o_n(1)$.  Thanks to Lemma~\ref{lem0001} once more and the condition $(D_4)$, we  observe  that $\|u_{\lambda,\beta_n}-u_{\lambda,0}\|_a=\|v_{\lambda,\beta_n}-v_{\lambda,0}\|_b=o_n(1)$ for $\lambda\geq\Lambda_1$.
Since by \eqref{eq0712} and \eqref{eq0713}, $\mathcal{N}_{a,\lambda}^*$ and $\mathcal{N}_{b,\lambda}^*$ are a natural constraint in $H_0^1(\{u_{\lambda,0}>0\})$ and $H_0^1(\{v_{\lambda,0}>0\})$, respectively, which together with $(b^{***})$, implies $u_{\lambda,0}$ and $v_{\lambda,0}$ are a least energy solution of \eqref{eq0401} and \eqref{eq0402},  respectively.

\vskip0.1in

{\bf Step 5. }\quad We prove that $\{x\in\bbr^3\mid u_{\lambda,0}(x)>0\}$ and $\{x\in\bbr^3\mid v_{\lambda,0}(x)>0\}$ are connected domains and $\{x\in\bbr^3\mid u_{\lambda,0}(x)>0\}=\bbr^3\backslash\overline{\{x\in\bbr^3\mid v_{\lambda,0}(x)>0\}}$.

Indeed, since $u_{\lambda,0}$ and $v_{\lambda,0}$ are respectively a least energy solution of \eqref{eq0401} and \eqref{eq0402}, we must have $\{x\in\bbr^3\mid u_{\lambda,0}(x)>0\}$ and $\{x\in\bbr^3\mid v_{\lambda,0}(x)>0\}$ are connected domains.  It remains to show that $\{x\in\bbr^3\mid u_{\lambda,0}(x)>0\}=\bbr^3\backslash\overline{\{x\in\bbr^3\mid v_{\lambda,0}(x)>0\}}$.  To the contrary,  we suppose that there exists an open set $\Omega$ satisfying $\{x\in\bbr^3\mid u_{\lambda,0}(x)>0\}\subsetneqq\Omega\subsetneqq\bbr^3\backslash\overline{\{x\in\bbr^3\mid v_{\lambda,0}(x)>0\}}$.  Furthermore, it has a locally  lipschitz boundary.  Since $\Omega\subsetneqq\bbr^3\backslash\overline{\{x\in\bbr^3\mid v_{\lambda,0}(x)>0\}}$, by a similar argument as in Step 3, we can show that $u_{\lambda,0}$ is a least energy solution of the following equation:
\begin{equation*}
-\Delta u+(\lambda a(x)+a_0(x))u=\mu_1 u^3,\quad u\in H_0^1(\Omega).
\end{equation*}
Since $u_{\lambda,0}\geq0$ in $\Omega$, by a similar argument as in the proof of Theorem~\ref{thm0002}, we can conclude that $u_{\lambda,0}>0$ on $\Omega$, which contradicts to $\{x\in\bbr^3\mid u_{\lambda,0}(x)>0\}\subsetneqq\Omega$.

We complete the proof by showing that $\beta^2\int_{\bbr^3}u_{\lambda,\beta}^2v_{\lambda,\beta}^2\to0$ as $\beta\to-\infty$, where $(u_{\lambda,\beta},v_{\lambda,\beta})$ is the ground state solution of $(\mathcal{P}_{\lambda,\beta})$ obtained by Theorem~\ref{thm0002}.  Indeed,  if not,  then  there exists $\{\beta_n\}\subset(-\infty, 0)$ such that $\beta_n^2\int_{\bbr^3}u_{\lambda,\beta_n}^2v_{\lambda,\beta_n}^2\leq -C$.  By \eqref{eq0405} and \eqref{eq0406}, we must have $\beta_n^2\int_{\bbr^3}u_{\lambda,\beta_n}^2v_{\lambda,\beta_n}^2\to0$ as $\beta_n\to-\infty$ up to a subsequence, which is a contradiction.
\qquad\raisebox{-0.5mm}{%
\rule{1.5mm}{4mm}}\vspace{6pt}

\section{Acknowledgements}
This work was partly completed when the first author was visiting Tsinghua University, and he is grateful to the members in the Department of Mathematical Sciences at Tsinghua University for their invitation and hospitality.  The first author was supported by the Fundamental Research Funds for the Central Universities (2014QNA67), China.  The second author was supported by the National Science Council and the National Center for Theoretical Sciences (South), Taiwan. The third author was supported by NSFC(11025106, 11371212, 11271386) and the Both-Side Tsinghua Fund.

Y. Wu and W. Zou  thank Prof. A. Szulkin for his suggestions and discussions  when he was visiting Tsinghua University.  They are also grateful to Dr. Zhijie Chen for careful reading the
manuscript and some suggestions to improve it.

\end{document}